\documentclass{amsart}
\usepackage{graphicx} 
\usepackage{amsmath, amssymb, amsthm}
\usepackage{mathrsfs}
\usepackage{geometry}
\usepackage{float}
\usepackage{xcolor}
\usepackage{mathabx}
\DeclareMathOperator*{\argmin}{\arg\min}
\DeclareMathOperator*{\argmax}{\arg\max}
\DeclareMathOperator{\Span}{span}
\DeclareMathOperator{\diag}{diag}
\DeclareMathOperator{\cl}{cl}

\newcommand{\inn}[2]{{\langle #1,#2 \rangle}}

\title[DCNN and data approximation using the FrFT]{Deep convolutional neural networks and data approximation using the fractional Fourier transform}

\author{M. H. A. Biswas}
\address{Department of Mathematics, Indian Institute of Science, Bangalore, India}
\email{mdhasanalibiswas4@gmail.com}

\author{P. Massopust}
\address{Center of Mathematics, Technical University of Munich, Garching by Munich, Germany}
\email{massopust@ma.tum.de}

\author{R. Ramakrishnan}
\address{Department of Mathematics, Indian Institute of Technology Madras, Chennai, India}
\email{radharam@iitm.ac.in}

\begin{document}
\newtheorem{theorem}{Theorem}[section]
\newtheorem{lemma}[theorem]{Lemma}
\theoremstyle{definition}
\newtheorem{definition}[theorem]{Definition}
\newtheorem{example}[theorem]{Example}
\newtheorem{xca}[theorem]{Exercise}
\newtheorem{corollary}[theorem]{Corollary}
\newtheorem{propn}[theorem]{Proposition}
\newtheorem{remark}[theorem]{Remark}

\numberwithin{equation}{section}
\newcommand{\glk}{g_{\lambda_k}}
\newcommand{\mkt}{\mathcal{M}_{k}^\theta}
\newcommand{\pkt}{\mathrm{P}_{k}^\theta}
\newcommand{\R}{\mathbb{R}}
\newcommand{\tc}{\star_\theta}
\newcommand{\ltn}{L^2(\mathbb{R}^n)}
\newcommand{\T}{\mathcal{T}}
\newcommand{\tlt}{\mathcal{T}^{\ell, \theta}}
\newcommand{\mlt}{M^{\ell, \theta}_N}
\newcommand{\ch}{\mathcal{C}_\theta}
\newcommand{\nch}{\mathcal{C}_{-\theta}}
\newcommand{\fft}{\mathscr{F}_\theta}
\newcommand{\lil}{\lambda\in\Lambda}
\newcommand{\C}{\mathbb{C}}
\newcommand{\intn}{\int_{\R^n}}
\newcommand{\diff}{\mathrm{d}}
\newcommand{\A}{\mathcal{A}}
\newcommand{\G}{\mathcal{G}}
\newcommand{\s}{\mathcal{S}}
\newcommand{\F}{\mathscr{F}}
\newcommand{\fl}{\mathcal{F}}
\newcommand{\N}{\mathbb{N}}
\newcommand{\Z}{\mathbb{Z}}
\newcommand{\Hi}{\mathcal{H}}
\newcommand*{\bigchi}{\mbox{\Large$\chi$}}
\newcommand*{\htau}{\mbox{\huge$\tau$}}
\newcommand*{\fiber}{\htau_{\Hm}}
\def\Msp{{\boldsymbol M}}
\newcommand{\Hm}{\mathrm{H}}
\newcommand{\sinc}{\mathrm{sinc}}
\newcommand{\lng}{\left \langle}
\newcommand{\rng}{\right \rangle}
\newcommand{\lc}{\left\lbrace}
\newcommand{\rc}{\right\rbrace}
\newcommand{\lnm}{\left\|}
\newcommand{\rnm}{\right\|}
\newcommand{\lp}{\left(}
\newcommand{\rp}{\right)}
\newcommand{\abs}[1]{\lvert#1\rvert}

\newcommand{\blankbox}[2]{%
  \parbox{\columnwidth}{\centering
    \setlength{\fboxsep}{0pt}%
    \fbox{\raisebox{0pt}[#2]{\hspace{#1}}}%
  }%
}

\begin{abstract}
In the first part of this paper, we define a deep convolutional neural network connected with the fractional Fourier transform (FrFT) using the $\theta$-translation operator, the translation operator associated with the FrFT. Subsequently, we study $\theta$-translation invariance properties of this network. Unlike the classical case, these networks are not translation invariant.
\par In the second part, we study data approximation problems using the FrFT. More precisely, given a data set $\fl=\{f_1,\cdots, f_m\}\subset L^2(\R^n)$, we obtain $\Phi=\{\phi_1,\cdots,\phi_\ell\}$ such that
\[
V_\theta(\Phi)=\argmin\sum_{j=1}^m \|f_j-P_{V}f_j\|^2,
\] 
where the minimum is taken over all $\theta$-shift invariant spaces generated by at most $\ell$ elements. Moreover, we prove the existence of a space of bandlimited functions in the FrFT domain which is ``closest" to $\fl$ in the above sense.
\end{abstract}

\maketitle

\section{Introduction}
In the recent years, {\it neural networks} have been widely studied by mathematicians, engineers, and physicists because of their high potential of applicability to several real life applications such as pattern recognition and the classification of images. These goals are achieved by extracting various features from a given image. In this article, we shall discuss the notion of feature extraction using a deep convolutional neural network (DCNN) based on the fractional Fourier transform (FrFT).\\

The reasons for choosing the FrFT are manifold. The FrFT is a generalization of the standard Fourier Transform, which provides a continuous range of transformations between the time and frequency domains. Whilst the Fourier Transform maps a function from time to frequency domain (or vice versa), the FrFT allows for transformations that are intermediate between these two extremes. This flexibility is useful in signal processing tasks where partial transformations between domains are required. The FrFT can be particularly advantageous in time-frequency analysis, where signals are analyzed in both time and frequency simultaneously. By varying the fractional order, the FrFT provides a more nuanced representation of a signal's characteristics in both domains, capturing details that might be missed by the standard Fourier Transform. In signal filtering applications, the FrFT can be used to design filters that operate in a fractional domain. This can lead to more efficient or effective filtering processes, particularly for signals where the standard frequency domain representation is not optimal. One feature that is attractive for our approach concerns chirp signals. These have a frequency that varies with time and are more effectively represented using the FrFT. The FrFT can align more closely with the structure of chirp signals, leading to better signal analysis and processing outcomes compared to the standard FT. Furthermore, as the FrFT can transform signals to an intermediate domain that are more aligned with the structure of the signal, better noise reduction techniques can be employed. The FrFT also offers more flexibility in selecting transformation parameters. This flexibility can be exploited to optimize various signal processing tasks, such as compression, encryption, and pattern recognition. The interested reader may consult \cite{Ozaktas,Zayed} for more information about the FrFT and its applications.\\

In his ground-breaking work \cite{Mallat}, Mallat has defined a DCNN using the wavelet transform. Afterwards in \cite{Wiatowski}, Wiatowski et al. have generalized Mallat's wavelet-based convolutional network using a countable collection of semi-discrete frames. In order to introduce their notion of DCNN, let us fix the following notation. Consider a sequence of ordered tuples $\Omega=\{(\Psi_k, P_k, M_k)\}$, where each $\Psi_k=\{g_{\lambda_k}:\lambda_k\in\Lambda_k\}$ generates a semi-discrete frame for $L^2(\R^n)$ and $M_k$ and $P_k$ are called the non-linearity and pooling operators, respectively. These operators are Lipschitz continuous on $\ltn$ with $P_kf=M_kf=0$, for $f=0$, $k\in \N$. For each $k\in \N$, we fix some $g_{\lambda_k^*}\in \Psi_k$, which we call {\it output generating atom} at the $(k-1)^{th}$ network level. With a little abuse of notation, we write $\Lambda_k$ for $\Lambda_k\setminus\{\lambda_k^*\}$ as well. Using this set-up the feature extractor is then defined as
\[
\Phi_\Omega(f) :=\bigcup_{k=0}^\infty \Phi^k_\Omega(f),
\]
where
\[
\Phi^{k}_{\Omega}(f) :=\lc \lp U^\theta[q]f\rp \star_\theta \varphi_k:q\in\Lambda^k\rc.
\]
Here, $\varphi_k$ is the output generating atom at the $k$-th network level,  $\Lambda^k :=\Lambda_1\times\cdots\times\Lambda_k$ and the operator $U$ is defined as
\[
U^k[\lambda_k]f :=D_kP_kM_k\lp f\tc\glk\rp \quad\text{and}\quad U[q]:=U[\lambda_1,\cdots,\lambda_k]=U_k[\lambda_k]\cdots U_1[\lambda_1],
\] 
where $D_k \textcolor{black}{:= D_{s_k}}$ is the dilation operator on $L^2(\R^n)$ by some $s_k>1$, called the {\it pooling factor}. For this feature extractor, the authors in \cite{Wiatowski} have obtained the following translation invariance property:
\[
\lim_{k\to\infty}|\|\Phi_{\Omega}^{k}(T_t f)-\Phi^{k}_{\Omega}(f)\||^2=\lim_{k\to\infty}\sum_{k=0}^\infty\sum_{\lambda\in\Lambda^k}\|U[\lambda]f\star\varphi_k\|^2=0.
\]
\textcolor{black}{Here, $T_t:L^2(\R^n)\to L^2(\R^n)$ denotes the shift operator by $t\in\R^n$.}

\par In this paper, we aim to study the DCNN in connection with the FrFT.  The FrFT is a particular instance of the special affine Fourier transform (SAFT). The SAFT was introduced by Abe and Sheridan in \cite{abe} in connection with optical wave functions. The SAFT of a function $f$ is formally defined as
\begin{equation}\label{eq:saft def eq}
\F_A(f)(\omega) :=\frac{1}{\sqrt{|b|}}\int_{\R} f(t)e^{\frac{\pi i}{b}\left( at^2+2pt-2\omega t+d\omega^2+2(bq-dp)\omega\right)}\diff t,~~\omega\in\R,
\end{equation}
where $A$ stands for the set of all real parameters $\{a,b,c,d,p,q\}$ with $b\neq 0$ and $ad-bc=1$. This transform generalizes many well-known classical transforms such as classical Fourier transform, the FrFT, the Fresnel transform and the linear canonical transform. Setting $A=\{0,1,-1,0,0,0\}$ and $A=\{0,-1,1,0,0,0\}$, we obtain the classical Fourier transform and the inverse Fourier transform, respectively. Taking $A=\{\cos\theta,\sin\theta,-\sin\theta,\cos\theta,0,0\}$ and $A=\{1,\lambda,0,1,0,0\}$ one obtains the FrFT and Fresnel transform, respectively. The SAFT has several applications in optics, signal processing, phase retrieval problem, and image edge detection. We refer to \cite{fu2023riesz,saftOptics} in this connection.
 \par It is relevant to mention that in \cite{Bhandari}, Bhandari and Zayed have studied the Shannon sampling theorem by introducing chirp modulated shift invariant spaces in the context of the SAFT. Later, in \cite{H2}, the authors have defined a translation operator associated with the SAFT, which they then used to obtain shift-invariant spaces in connection with the SAFT and to study sampling problems in those spaces.
\par The FrFT is defined by
\begin{equation}
    \F_\theta(f)(\omega):=\frac{1}{\sqrt{|\sin\theta|}}\int_\R f(t)e^{\pi i\left( (t^2+\omega^2)\cot\theta-2\omega t\csc\theta\right)}\diff t,~~\omega\in\R.
\end{equation}
One can extend the definition of the FrFT to functions on $\R^n$ by treating $t^2$ as $t\cdot t$ and $t\omega=t\cdot\omega$, where $``\cdot"$ denotes the usual inner product on $\R^n$. 
%
For the precise statement, see Definition \ref{def3.2}.
 
\par In this paper, we employ the $\theta$-translation operator $T^\theta$ introduced in \cite{H2} to define a semi-discrete frame associated with the FrFT. We characterize a semi-discrete system $\{T_s^\theta f g_\lambda: s\in\R^n, \lil\}\subset L^1(\R^n)\cap L^2(\R^n)$ to be a semi-discrete frame for $L^2(\R^n)$, where $\Lambda$ is a countable index set. In order to define the DCNN in the context of the FrFT, we use a collection of semi-discrete frames, convolution in connection with the FrFT, and non-linearities and pooling operators similar to that of \cite{Wiatowski}. Finally, we study the $\theta$-translation invariance property of this network. Unlike the classical case, the DCNN connected with the FrFT is not invariant under the $\theta$-translation. (See Corollary \ref{cor: translation invariance}.) \\

\par Another fascinating problem in signal and image processing is data approximation. More precisely, given a finite set of functional data $\fl=\{f_1,\cdots, f_m\}\subset L^2(\R^n)$, one is interested in finding $\Phi:=\{\phi_1,\cdots,\phi_\ell\}\subset L^2(\R^n)$ such that
\[
V(\Phi)=\text{arg min} \sum_{j=1}^m\|f_j-P_{V}f_j\|^2,
\] 
where the minimum is taken over all shift invariant spaces generated by at most $\ell$ elements, \textcolor{black}{$V(\Phi):=\cl_{L^2(\R^n)}{\Span}\{T_k\phi_j:k\in\Z^n,1\leq j\leq \ell\}$}, \textcolor{black}{and $P_V:L^2(\R^n)\to L^2(\R^n)$ denoting the orthogonal projection operator onto $V$.} Clearly, this is a least square problem in an infinite dimensional Hilbert space. In \cite{AldroubiACHA07}, Aldroubi et al. converted this infinite dimensional problem into an uncountable collection of least square problems in finite dimensional Hilbert spaces using the Fourier transform. Then, they obtained the solution to each finite dimensional least square problem by using the Eckart-Young theorem and by patching together these solutions to obtain a solution to the infinite dimensional data approximation problem. Later on, in \cite{cabrelli}, using multi-tiles of $\R^n$, Cabrelli et al. have obtained a space of band-limited functions which minimizes $\sum\limits_{j=1}^m\|f_j-P_{V}f_j\|^2$. 
\par In this paper, we study the data approximation problem using the FrFT. In order to do so, we need to first introduce the notion of a fiber map associated with the FrFT and study some properties of certain $\theta$-shift invariant spaces. Afterwards, we consider the data approximation problem. Specifically, given a set of functional data $\mathcal{F}=\{f_1,\cdots,f_m\}\subset L^2(\R^n)$ and $\ell\in\N$, we obtain $\phi_1,\cdots,\phi_\ell\in L^2(\R^n)$ such that
    \begin{equation}\label{eq: intro data approx}
    \sum_{i=1}^m\|f_i-P_{V_\theta}f_i\|^2\leq\sum_{i=1}^m\|f_i-P_{W_\theta}f_i\|^2,
    \end{equation}
    holds for all $\theta$-shift invariant spaces $W_\theta$ generated by at most $\ell$ elements, where $V_\theta=\cl_{L^2(\R^n)}{\Span}\{T_k^\theta\phi_i:i=1,\cdots,\ell;\,k\in\Z^n\}$. Further, we present a formula to compute the error in the above least square approximation. Moreover, we prove the existence of a space of band-limited functions in the FrFT domain which minimizes the right-hand side of \eqref{eq: intro data approx}.\\

\section{Preliminaries}   
In this section, we provide some necessary preliminaries.

\begin{definition}
A sequence of elements $\{f_k:k\in \N\}$ in a separable Hilbert space $\Hi$ is called a \textit{frame} for $\Hi$ if there exist $0<C_1\leq C_2<\infty$ such that
\begin{equation}\label{eq:frame def}
C_1\|f\|^2\leq \sum_{k=1}^\infty |\langle f,f_k \rangle |^2\leq C_2\|f\|^2,\quad \text{ for all } f\in \Hi.
\end{equation}
\par If the inequality on right-hand side is satisfied for a sequence in $\Hi$, then this sequence is called a \textit{Bessel sequence}. If the inequalities \eqref{eq:frame def} hold with $C_1=C_2=1$, then $\{f_k:k\in \N\}$ is called a \textit{Parseval frame}. The numbers $C_1$ and $C_2$ are called the {\it lower} and {\it upper} frame bounds, respectively.
\end{definition}	

\begin{definition}
A sequence $\{f_k:k\in\N\}$ in a separable Hilbert space $\Hi$ is called a {\it Riesz basis} if
\begin{itemize}
\item[(i)] $\cl_{\Hi}{\Span}\{f_k:k\in\N\}=\Hi$;
\item[(ii)] there exist $0<C_1\leq C_2<\infty$ such that
\[
C_1\sum_{k\in\N}|c_k|^2\leq\lnm\sum_{k\in\N}c_kf_k\rnm^2\leq C_2\sum_{k\in\N}|c_k|^2,~~\text{ for all }\{c_k\}\in c_{00}.
\]
\end{itemize}
\textcolor{black}{Here, $c_{00}$ denotes the space of finitely supported null sequences.} 
If a collection of vectors is a Riesz basis for its closed linear span, then it is called a {\it Riesz sequence}.  
\end{definition}

\begin{definition}
The \textit{Gramian} associated with a Bessel sequence $\{f_k:k\in \N\}$ is an operator on $\ell ^2(\N)$ whose $(j,k)^{th}$ entry in the matrix representation with respect to the canonical orthonormal basis is $\langle f_k,f_j\rangle .$
\end{definition}

\begin{definition}
Let $\{g_\lambda:\lambda\in \Lambda\}\subset L^1(\R^n)\cap L^2(\R^n)$, where $\Lambda$ is a countable index set. Then, the collection of functions
\[
\Psi_\Lambda=\{T_sIg_\lambda:s\in \R^n,\lambda\in \Lambda\}
\]
is called a \textit{semi-discrete frame} for $L^2(\R^n)$ if there exist $0<C_1\leq C_2<\infty $ such that
\begin{equation}\label{eq:2 sdf defn}
C_1\|f\|^2\leq \sum_{\lambda\in \Lambda}\int_{\R^n}\big|\big\langle f,T_sIg_\lambda\big\rangle\big|^2ds\leq C_2\|f\|^2,
\end{equation}
for all $f\in L^2(\R^n)$, where the involution operator $I: L^2(\R^n)\to L^2(\R^n)$ is given by $If(x):=\overline{f(-x)}$. If \eqref{eq:2 sdf defn} holds with $C_1=C_2=1$, then $\Psi_\Lambda$ is called a {\it semi-discrete Parseval frame}. The generators $g_\lambda,~\lambda\in \Lambda$, are called {\it atoms} of the semi-discrete frame $\Psi_\Lambda$.
\end{definition}

\begin{definition}
A closed subspace $V$ of $L^2(\R^n)$ is said to be a \textit{shift invariant space} if $f\in V \Rightarrow T_kf \in V,$ for all $k \in \Z^n$ and for all $f\in V$.
\par In particular, for $\Phi:=\{\phi_1,\cdots\phi_m\} \subset L^2(\R^n)$, $V(\Phi):=\cl_{L^2(\R^n)}{\Span}\{T_k\phi_\ell :k \in \Z,\, \ell=1,\cdots,m\}$ is called a \textit{finitely generated shift invariant space}.
\end{definition}


The translation and modulation operators associated with the special affine Fourier transform were introduced in \cite{Hasan3} and \cite{H2}, respectively. As we mentioned earlier, as the FrFT is a particular case of the SAFT (with $A=\{\cos\theta, \sin\theta, -\sin\theta, \cos\theta\}$), we provide here the definitions of translation, modulation and convolution operators connected with the FrFT. 

\begin{definition}[\cite{H2}]\label{def: preli a translation}
For any  $s \in \R^n$, the $\theta$-{\it translation} of a measurable function $f:\R^n\to\R^n$ by $s$, denoted by $T_s^\theta f$, is defined by
\begin{equation}
T_s^\theta f(t) :=e^{2\pi is(t-s)\cot\theta}f(t-s),~~t\in\R^n.
\end{equation}
\end{definition}

\begin{remark}
We can relate the $\theta$-translation operator and the classical translation operator as follows.
\begin{equation}\label{eq:T_xT_x^A}
\ch (T_s^\theta f)=e^{\pi i s^2\cot\theta}T_s(\ch f),
\end{equation}
where the operator $\ch$, known as {\it chirp modulation}, is defined by
\begin{equation}\label{eq: chirp modulation}
\ch f(s):=e^{\pi is^2\cot\theta}f(s).
\end{equation}
\end{remark}

\begin{example}
In Figure \ref{fig2}, we display the $\theta$-translation with $s_1 = s_2 = 1$ of the unit box function $\chi_{[-\frac12,\frac12]}\otimes\chi_{[-\frac12,\frac12]}$ and the Gaussian $e^{-t\cdot t}$ in $\R^2$.
\begin{figure}[h!]
\begin{center}
\includegraphics[width=5cm, height= 4cm]{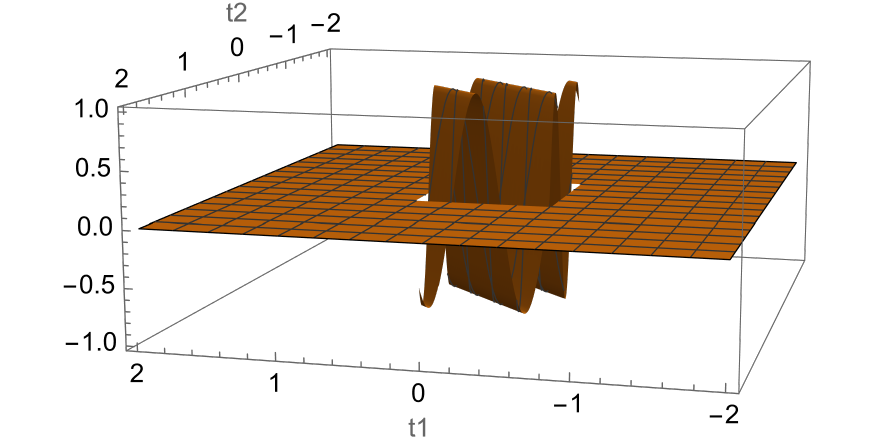}\hspace*{2cm}
\includegraphics[width=5cm, height= 4cm]{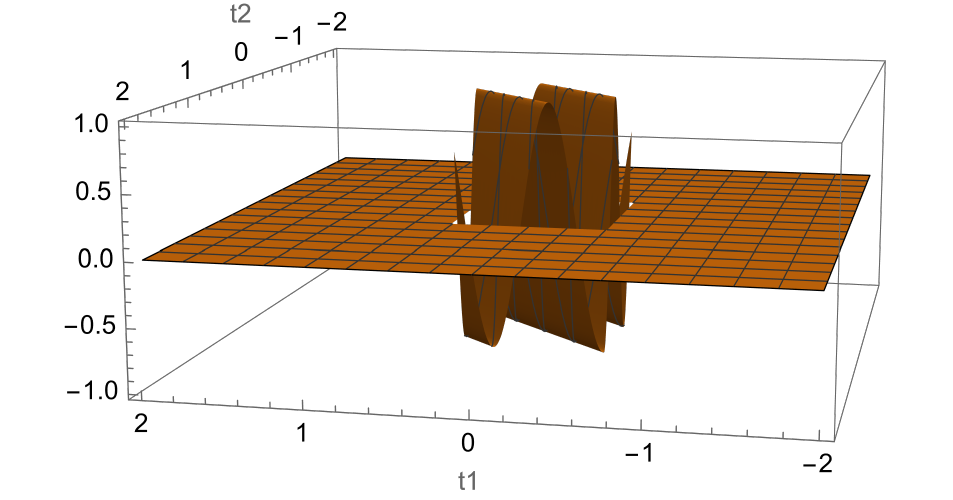}\\
\includegraphics[width=5cm, height= 4cm]{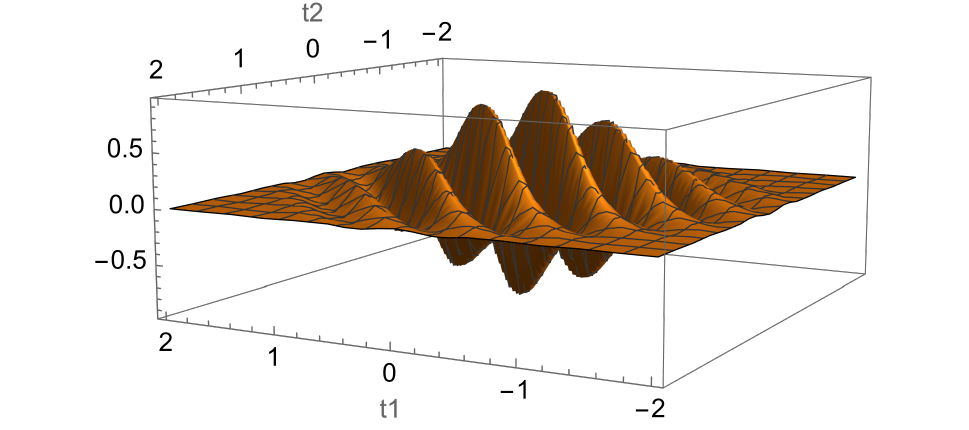}\hspace*{2cm}
\includegraphics[width=5cm, height= 4cm]{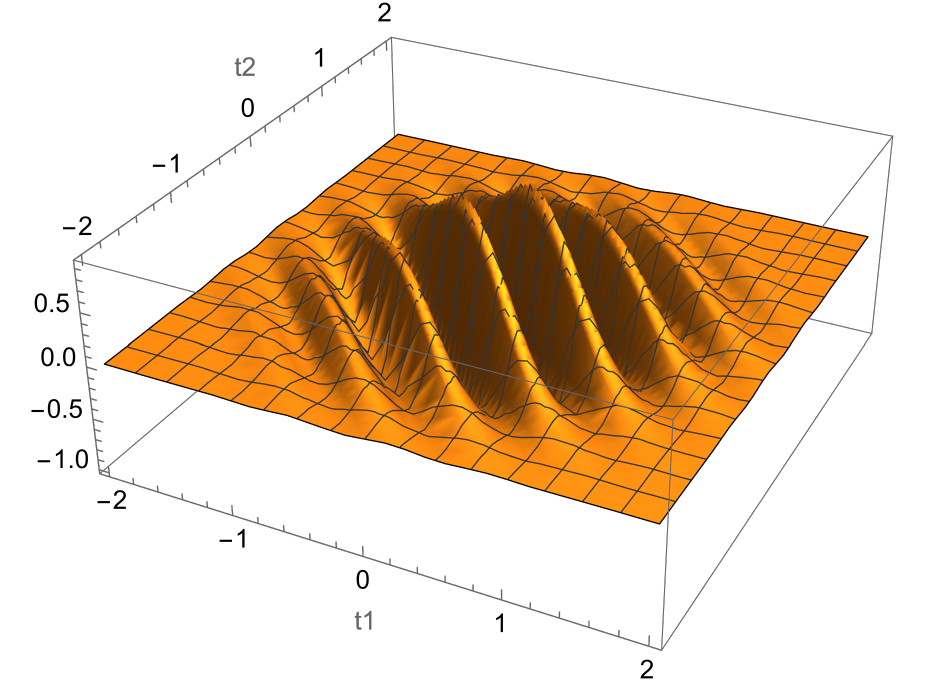}
\caption{The real and imaginary parts of the $\theta$-translation of the unit box function (above) and the Gaussian $e^{-t\cdot t}$ (bottom) in $\R^2$ for $\theta = \pi/5$ and  $s_1 = s_2 = 1$.}\label{fig2}
\end{center}
\end{figure}
\end{example}
\begin{example}
The real and imaginary parts of the chirp-modulated unit box function and the Gaussian $e^{-t\cdot t}$ in $\R^2$ are shown in Figure \ref{fig3}.
\begin{figure}[h!]
\begin{center}
\includegraphics[width=6cm, height= 4cm]{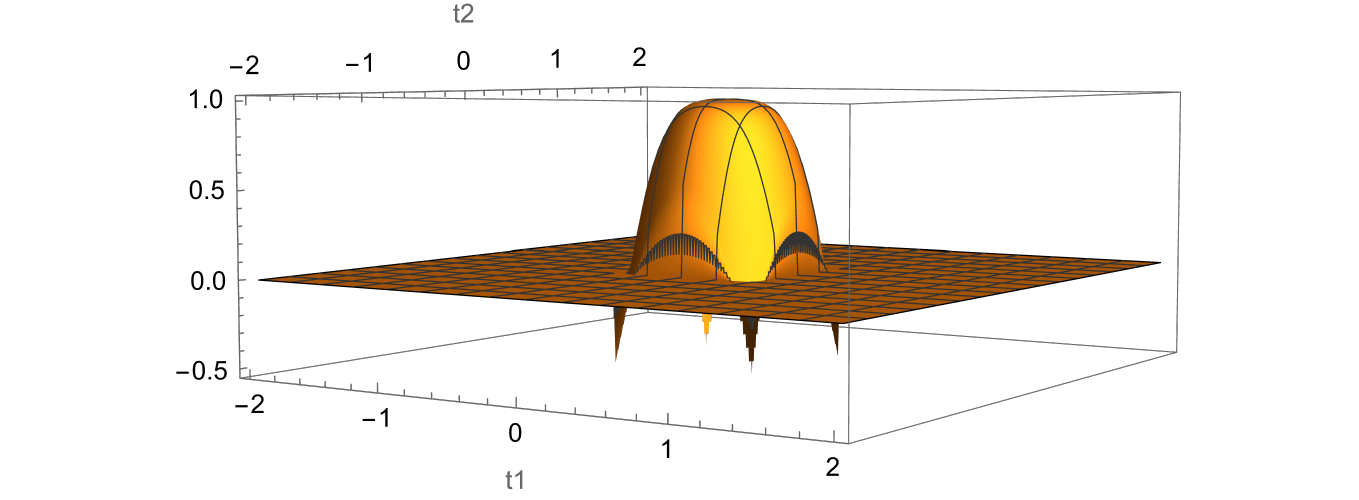}\hspace*{1cm}
\includegraphics[width=6cm, height= 4cm]{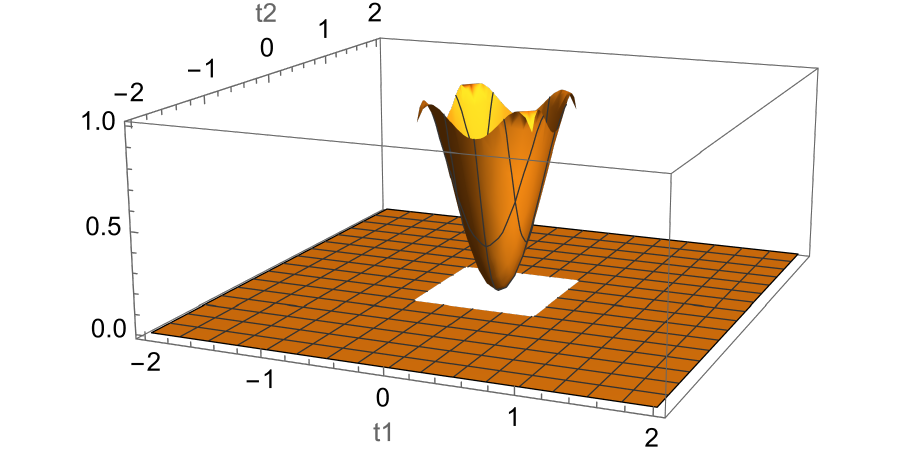}\\
\includegraphics[width=6cm, height= 4cm]{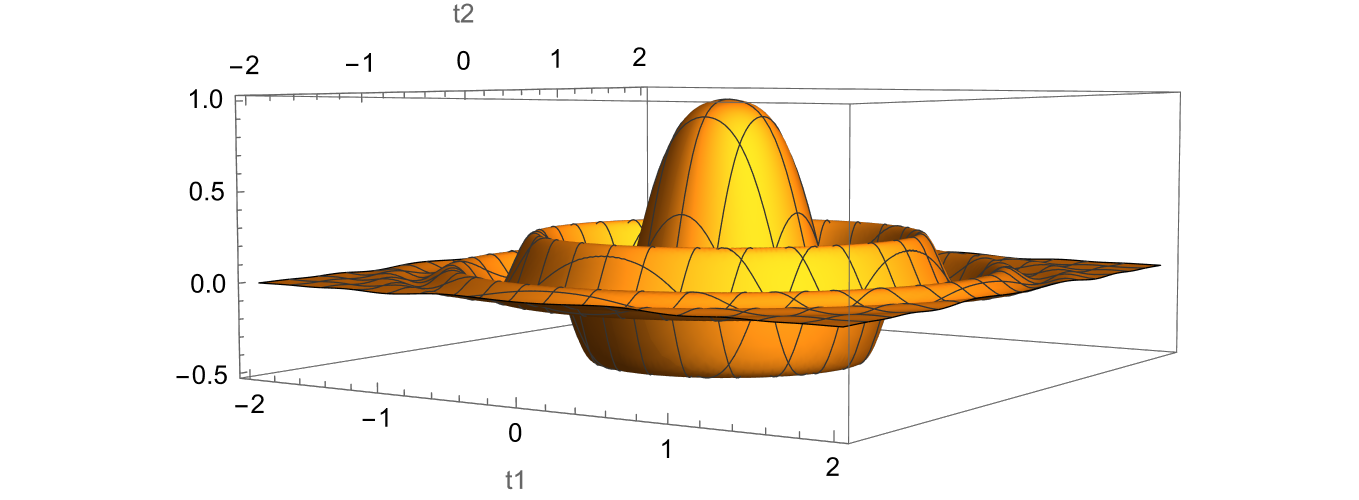}\hspace*{1cm}
\includegraphics[width=6cm, height= 4cm]{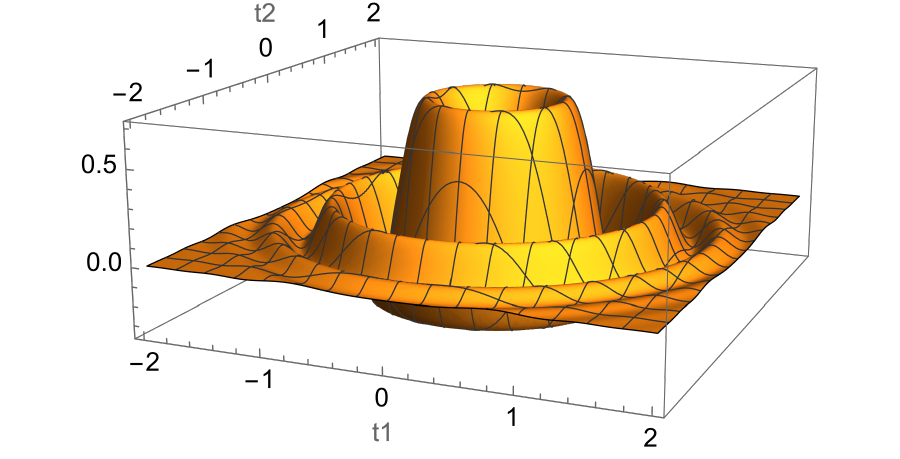}
\caption{The real and imaginary parts of the chirp modulation of the unit box function $\chi_{[-\frac12,\frac12]}\otimes\chi_{[-\frac12,\frac12]}$ (top) and the Gaussian $e^{-t\cdot t}$ (bottom) in $\R^2$ for $\theta = \pi/5$.}\label{fig3}
\end{center}
\end{figure}
\end{example}

\begin{definition}[\cite{Bhandari}, \cite{H2}]\label{def: preli a convolution}
Let $f,\,g\in L^1(\R^n)$. Then, the $\theta$-{\it convolution} of $f$ and $g$
is defined by
\begin{equation}\label{eq:A-convolution}
 (f\star_\theta g)(x):=\frac{1}{|\sin\theta|^{n/2}}\int_{\R^n} f(s) T_s^\theta g(x) \diff s.
\end{equation}
Furthermore, one has that
\begin{equation}\label{eq:A-convolution, convolution}
\ch (f\star_\theta g)=\frac{1}{|\sin\theta|^{n/2}}(\ch f\star \ch g),
\end{equation}
where
\[
f\star g(x):=\int_\R f(x-t)g(t)\diff t
\]
\textcolor{black}{is the usual convolution of two functions.}
\end{definition}
 
\begin{remark}[\cite{Bhandari}]\label{rk: preli convolution theorem}
  Let $f,\,g\in L^1(\R)$. Then,
  \begin{equation}\label{eq: conv F frft}
  \F_\theta(f\star_\theta g)(\omega)=e^{-\pi i\omega^2\cot\theta}\F_\theta(f)(\omega)\F_\theta(g)(\omega),~\omega\in\R.
  \end{equation}
\end{remark}


\begin{definition}[\cite{Hasan3}]\label{def: preli a modulation}
For any $s \in \R$, the $\theta$-{\it modulation} of a measurable function $f:\R^n\to\R^n$ is defined as:
 \begin{equation}
 M_s^\theta f(t):=e^{\pi i(s^2\cot\theta+2st\csc\theta)}f(t).
 \end{equation}
The $\theta$-translation and $\theta$-modulation are connected via the FrFT in the expected way, i.e., one has 
\begin{equation}\label{eq: preli translation modulation}
\F_\theta (T_s^\theta f)(\omega)=M_{-s}^\theta\F_\theta(f)(\omega).
\end{equation}
\end{definition}

\begin{definition}
    Let $\Omega\subset\R^n$ be measurable and let $L\subset\R^n$ be a countable set. We say that $\Omega$ {\it tiles} $\R^n$ when translated by $L$ at level $\ell\in \N$ if
    \[
    \sum_{t\in L}\bigchi_{\Omega}(\omega+t)=\ell, ~~\text{ a.e. }\omega\in\R^n.
    \]
    If $L=\Z^n\sin\theta$, then we say that $\Omega$ is a {\it fractional $\ell$ multi-tile}.
\end{definition}

\begin{example}
The real and imaginary parts of the $\theta$-modulation of the unit box function $\chi_{[-\frac12,\frac12]}\otimes\chi_{[-\frac12,\frac12]}$ and the Gaussian $e^{-t\cdot t}$ for $\theta = \pi/3$, $s_1 = \sqrt{2}$, and $s_2 = -\frac12$ are depicted in Figure \ref{fig4}.

\begin{figure}[h!]
\begin{center}
\includegraphics[width=6cm, height= 4cm]{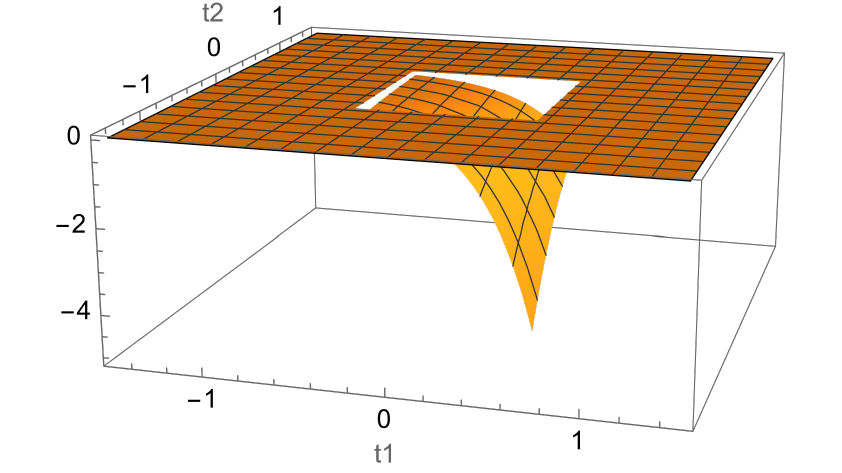}\hspace*{1cm}
\includegraphics[width=6cm, height= 4cm]{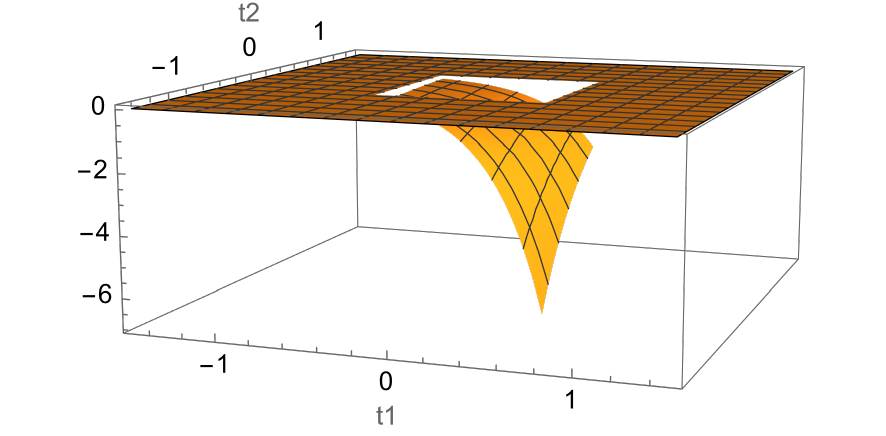}\\
\includegraphics[width=6cm, height= 4cm]{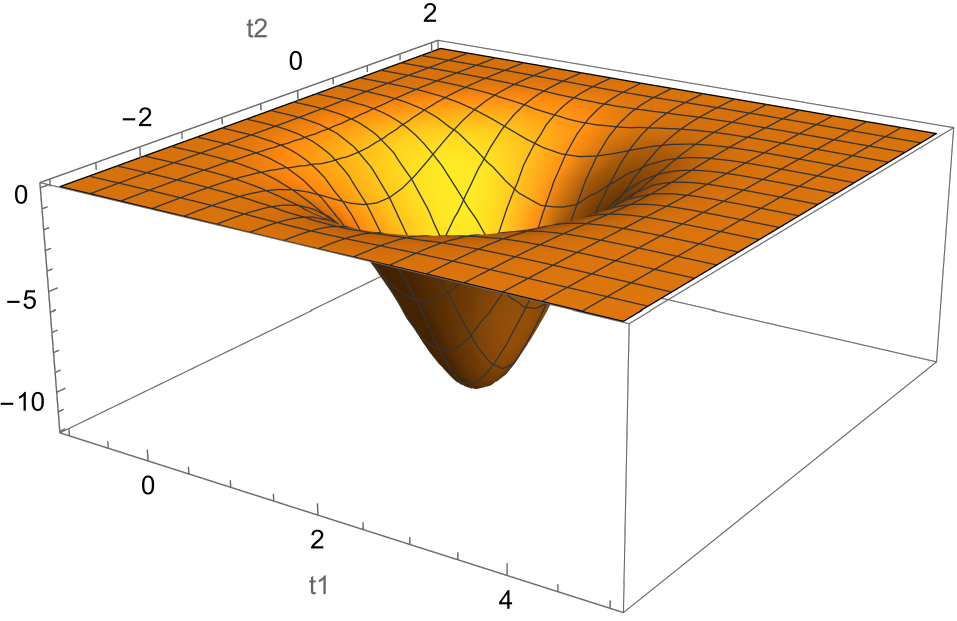}\hspace*{1cm}
\includegraphics[width=6cm, height= 4cm]{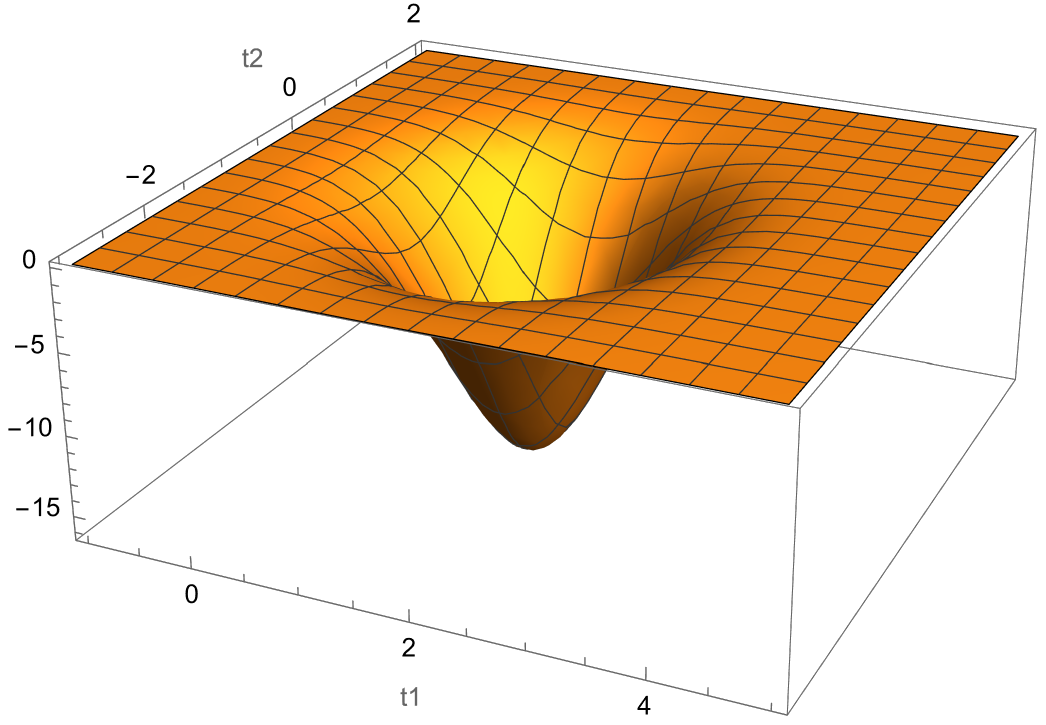}
\caption{The real and imaginary parts of the $\theta$-modulation of the unit box function $\chi_{[-\frac12,\frac12]}\otimes\chi_{[-\frac12,\frac12]}$ (top) and the Gaussian $e^{-t\cdot t}$  (bottom) for $\theta = \pi/3$, $s_1 = \sqrt{2}$, and $s_2 = -\frac12$.}\label{fig4}
\end{center}
\end{figure}
\end{example}

\begin{lemma}[\cite{kolountzakis}]\label{lemma: partition of tile}
    Suppose that $\Omega\subset\R^n$ is a measurable set which tiles $\R^n$ at level $\ell$ when translated by $L$. Then, we have
    \[
    \Omega=\Omega_1\cup\cdots\Omega_\ell\cup E,
    \]
    where $E$ is a set of measure zero, the $\Omega_j$'s are mutually disjoint, and each of them multi-tiles $\R^n$ when translated by $L$.
\end{lemma}

\section{Deep convolutional neural network in connection with the FrFT}
In this section, we introduce the notion of a semi-discrete system connected with the FrFT $\F_\theta$ using the $\theta$-translation operator. Then, using these semi-discrete systems, we define a {\it deep convolutional neural network} associated with the FrFT. Finally, we study the $\theta$-translation invariance properties for this network. We use the following notation in the subsequent proposition: \textcolor{black}{The idempotents $I,\, \widecheck{} :L^1(\R^n)\to L^1(\R^n)$ are defined by}
\begin{align*}
    If(x)&:=\overline{f(-x)},\\
    \textcolor{black}{\widecheck{f}}(x)&:=f(-x).
\end{align*}

\begin{propn}
Let $\Lambda$ be a countable index set and $\lc g_\lambda:\lil\rc$ be a collection of functions in $L^1(\R^n)\cap\ltn$. Then, the system of translates $\left\{T_sIg_\lambda:s\in\R^n,\,\lambda\in\Lambda\right\}$ is a semi-discrete frame for $L^2(\R^n)$ if and only if $\left\{T_s \widecheck{g}_\lambda:s\in\R^n,\,\lambda\in\Lambda\right\}$ is a semi-discrete frame for $L^2(\R^n)$.
\end{propn}
\begin{proof}
Consider
\begin{equation}\label{eq: sdf equiv 1}
\lng f,\,T_sIg_\lambda\rng=\int_{\R^n} f(t)\overline{T_sIg_\lambda (t)}\diff t=\int_{\R^n}f(t)g_\lambda(s-t)\diff t=f\star g_\lambda(s).\\
\end{equation}
Similarly,
\begin{equation}\label{eq: sdf equiv 2}
\lng T_s \widecheck{g}_\lambda,\, f\rng=g_\lambda\star\bar{f}(s).
\end{equation}
Now, the system $\lc T_sIg_\lambda:s\in\R^n,\,\lambda\in\Lambda \rc$ is a semi-discrete frame for $L^2(\R^n)$ if and only if there exist $0<C_1\leq C_2<\infty$ such that
\[
C_1\lnm f\rnm^2\leq\sum_{\lil}\intn\abs{\lng f,\,T_sIg_\lambda\rng}^2\diff s\leq C_2\lnm f\rnm^2,~~\text{ for all }f\in L^2(\R^n).
\]
This is same as
\[
C_1\lnm \bar{f}\rnm^2\leq\sum_{\lil}\intn\abs{\lng \bar{f},\,T_sIg_\lambda\rng}^2\diff s\leq C_2\lnm\bar{f}\rnm^2, \quad\text{ for all }f\in L^2(\R^n),
\]
which is equivalent to
\[
C_1\lnm f\rnm^2\leq\sum_{\lil}\intn\abs{\bar{f}\star g_\lambda(s)}^2\diff s\leq C_2\lnm f\rnm^2, \quad\text{ for all }f\in L^2(\R^n).
\]
using \eqref{eq: sdf equiv 1}. Finally, \eqref{eq: sdf equiv 2} leads to
\[
C_1\lnm f\rnm^2\leq\sum_{\lil}\intn\abs{\lng f, T_s \widecheck{g}_\lambda\rng}^2\diff s\leq C_2\lnm f\rnm^2,\quad\text{ for all }f\in L^2(\R^n).
\]
\end{proof}

Next, we characterize a system of semi-discrete frames using the FrFT.  Towards this end, we recall first the notion of the FrFT in higher dimensions.

\begin{definition}\label{def3.2}
Let $f\in L^1(\R^n)$. Then, the FrFT of $f$, denoted by $\F_\theta(f)$, is defined as
\[
\F_\theta(f)(\omega) :=
\begin{cases}
 \frac{1}{|\sin\theta|^{n/2}}\displaystyle{\intn} f(t)e^{-\pi i\lp(t^2+\omega^2)\cot\theta-2\omega t\csc\theta\rp}, &~\theta\neq k\pi; \\
 f(\omega),&~\theta=2k\pi;\\
 f(-\omega) \textcolor{black}{= \widecheck{f}(\omega)},&~\theta=(2k-1)\pi.
 \end{cases}
 \]
\end{definition}

\begin{example}
In Figure 4 we display, from top to bottom, (i) the Fourier transform of the unit box function $\chi_{[-\frac12,\frac12]}\otimes\chi_{[-\frac12,\frac12]}$ in $\R^2$, (ii) its FrFT for $\theta = \pi/3$, and (iii) the FrFT of the Gaussian $e^{-t\cdot t}$ in $\R^2$ for $\theta = \pi/4$.
\begin{figure}[h!]
\begin{center}
\includegraphics[width=5cm, height= 3cm]{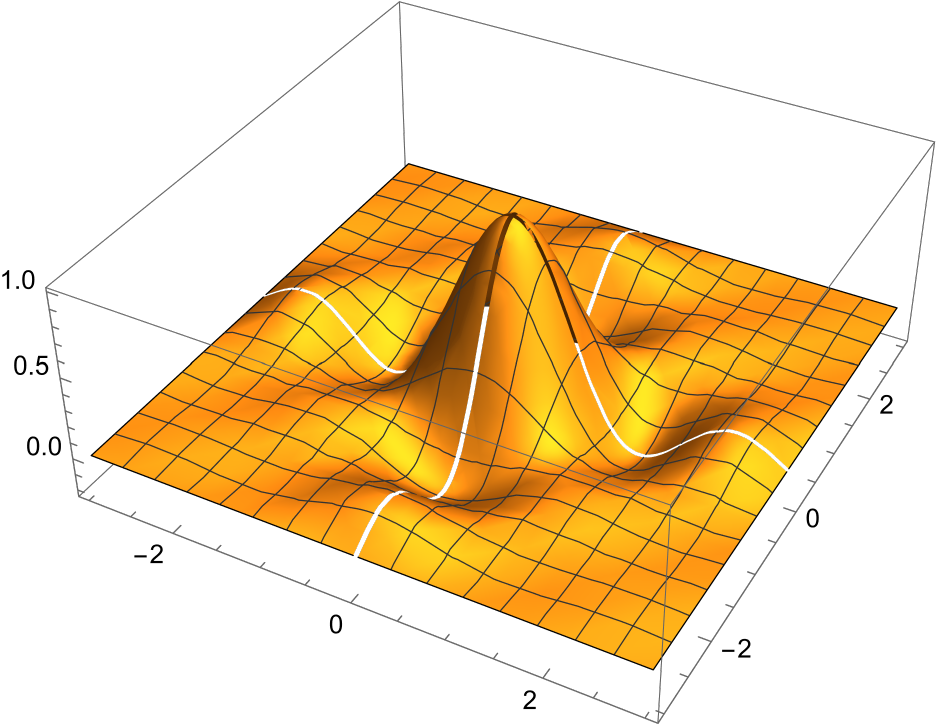}\\
\includegraphics[width=5cm, height= 3cm]{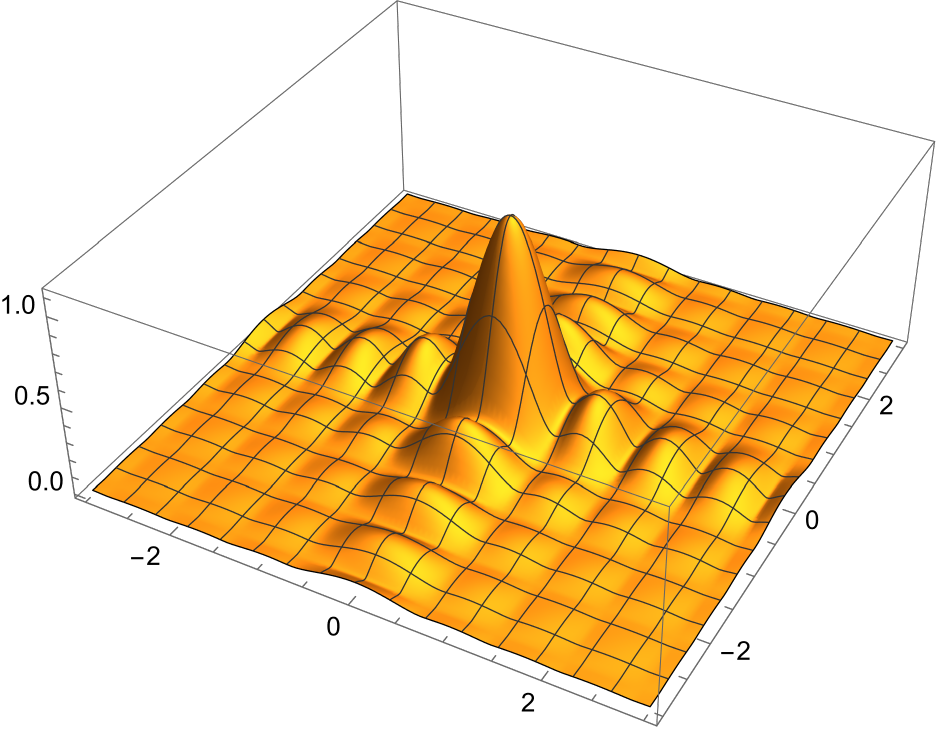}\quad
\includegraphics[width=5cm, height= 3cm]{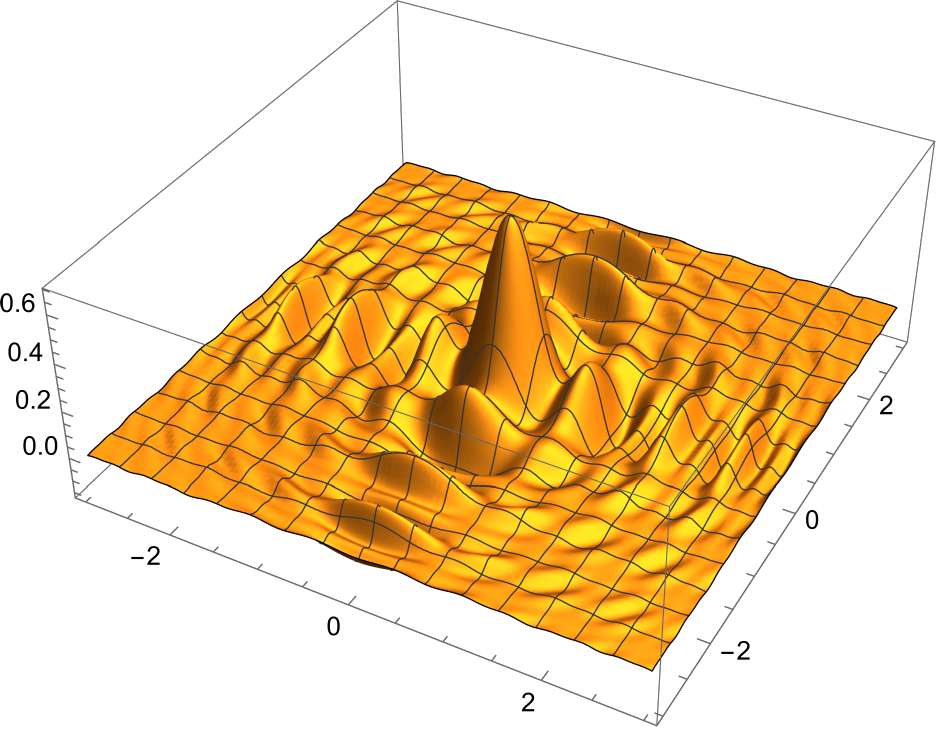}
\includegraphics[width=5cm, height= 3cm]{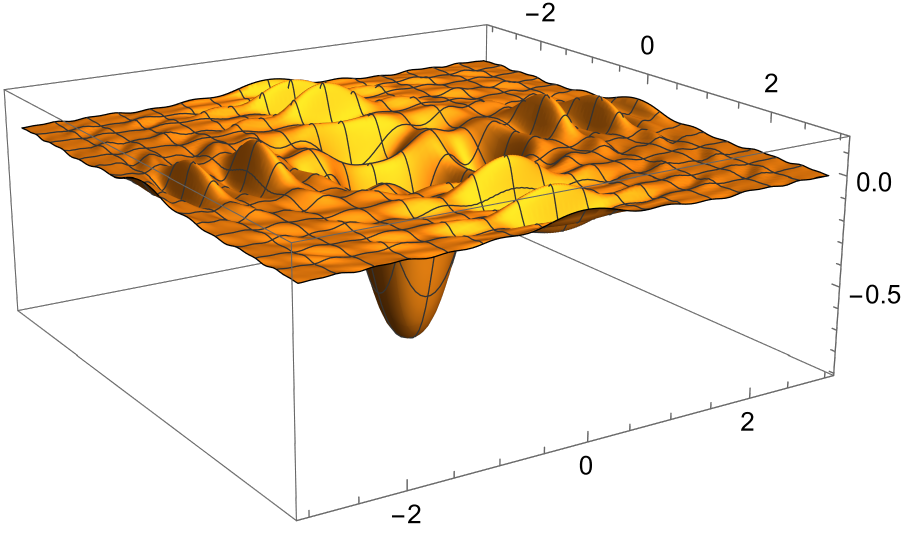}\\
\includegraphics[width=5cm, height= 3cm]{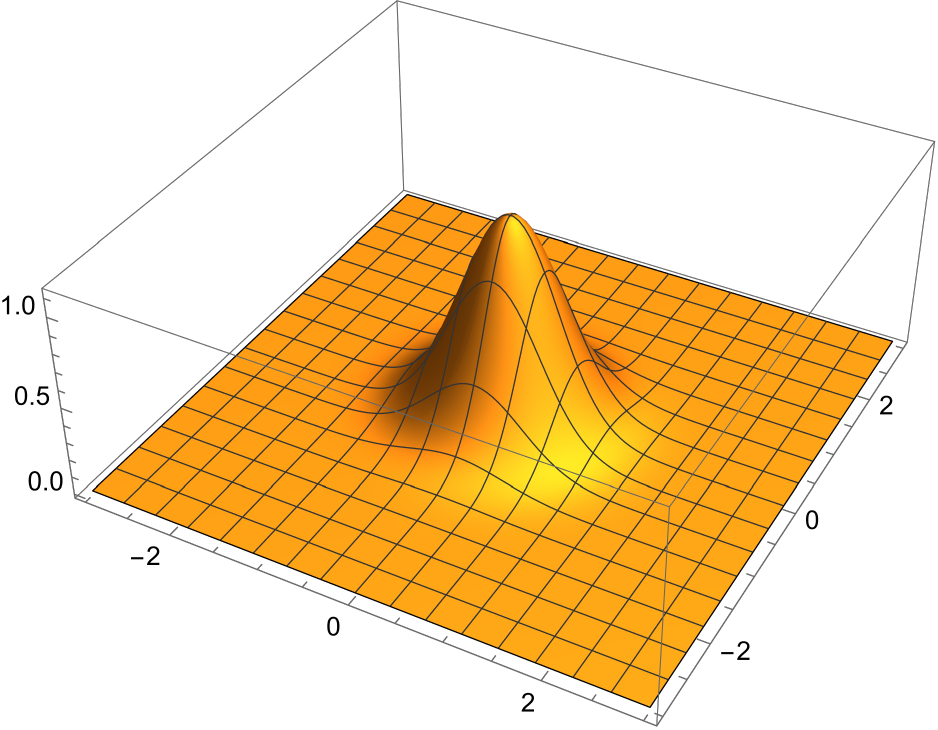}\quad
\includegraphics[width=5cm, height= 3cm]{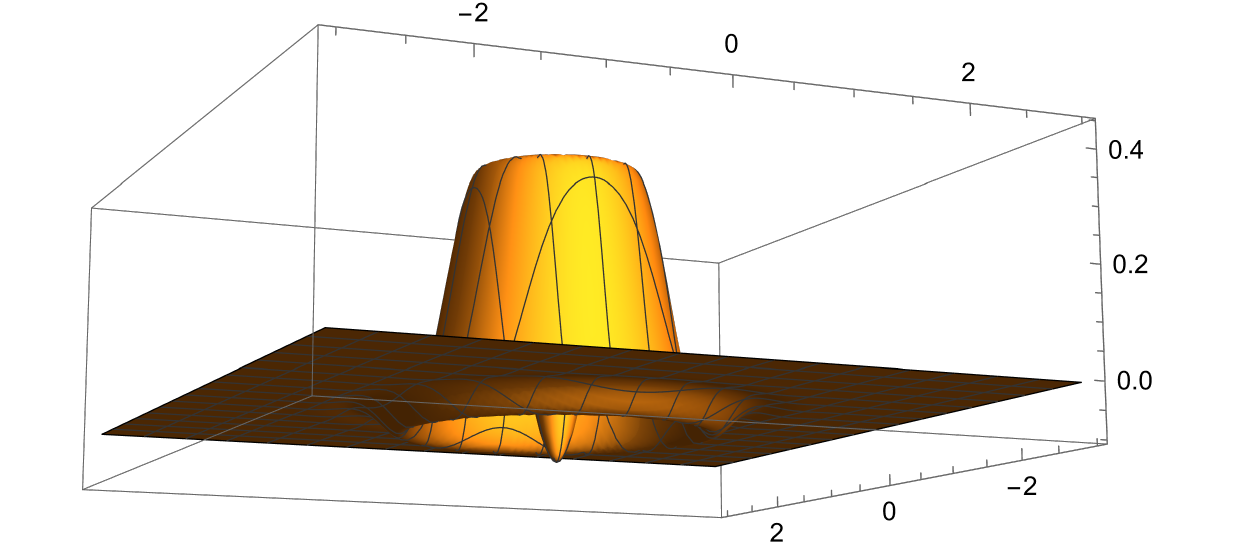}
\includegraphics[width=5cm, height= 3cm]{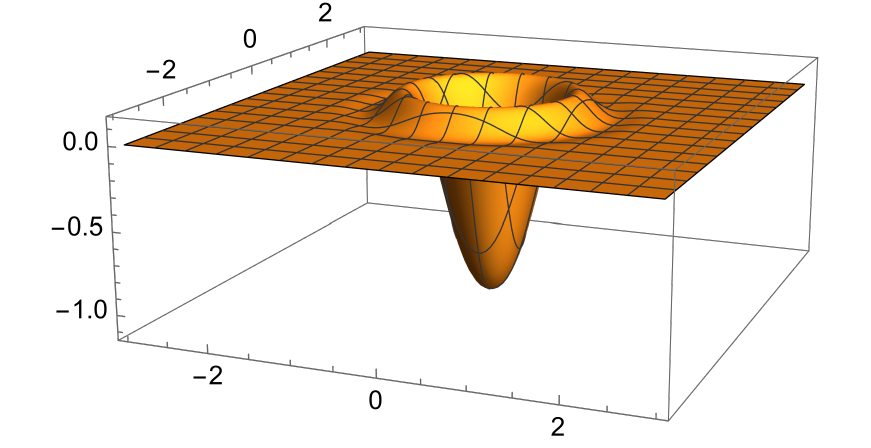}
\caption{The Fourier transform of the unit box function $\chi_{[-\frac12,\frac12]}\otimes\chi_{[-\frac12,\frac12]}$ (top), its absolute value, its real and imaginary parts under the FrFT with $\theta = \pi/3$ (middle), and the absolute value, real and imaginary part of the Gaussian  $e^{-t\cdot t}$ in $\R^2$ for $\theta = \pi/4$ (bottom).}\label{fig1}
\end{center}
\end{figure}
\end{example}

\begin{propn}
A semi-discrete system $\lc T_s^\theta \widecheck{g}_\lambda :s\in\R^n,\lil\rc$ is a semi-discrete frame for $L^2(\R^n)$ with bounds $C_1,C_2$ if and only if
\begin{equation}
\frac{C_1}{\abs{\sin\theta}^n}\leq\sum_{\lil}\abs{\F_\theta(g_\lambda)(\omega)}^2\leq \frac{C_2}{\abs{\sin\theta}^n},\quad\text{ for a.e. }\omega\in\R^n.
\end{equation}
\end{propn}

\begin{proof}
Consider
\begin{align*}
\lng T_s^\theta \widecheck{g}_\lambda,\,f\rng &=\intn T_s^\theta \widecheck{g}_\lambda(t)\overline{f(t)}\diff t=\intn e^{-2\pi is(t-s)\cot\theta}g_\lambda(s-t)\overline{f(t)}\diff t\\
&=e^{\pi is^2\cot\theta}\intn \ch g_\lambda(s-t)\nch \overline{f(t)}\diff t.
\end{align*}
Thus,
\begin{align}\label{eq: sdf char 1}
\nonumber\lng T_s \widecheck{g}_\lambda,\,\nch^2f\rng&=e^{\pi is^2\cot\theta}\intn \ch g_\lambda(s-t)\nonumber\nch\overline{\nch^2f(t)}\diff t\\
\nonumber &=e^{\pi is^2\cot\theta}\intn \ch g_\lambda(s-t)\ch\overline{f(t)}\diff t\\
&=e^{\pi is^2\cot\theta}(\ch g_\lambda\star\ch\bar{f})(s)=|\sin\theta|^{n/2}e^{2\pi is^2\cot\theta}(g_\lambda\star_\theta\bar{f})(s),
\end{align}
using \eqref{eq:A-convolution, convolution}. 

Now, assume that $\lc T_s^\theta \widecheck{g}_\lambda :s\in\R^n,\lil\rc$ is a semi-discrete frame for $\ltn$ with bounds $C_1, C_2$. Then 
\[
C_1\|f\|^2\leq\sum_{\lil}\intn\abs{\lng T_s^\theta \widecheck{g}_\lambda,\,f\rng}^2\diff s\leq C_2\|f\|^2,\quad\text{ for all } f\in\ltn.
\]
This is same as
\[
C_1\|f\|^2\leq\sum_{\lil}\intn\left|\lng T_s^\theta \widecheck{g}_\lambda,\, \nch^2f\rng\right|^2\diff s\leq C_2\|f\|^2,\quad\text{ for all } f\in\ltn.
\]
Now, using \eqref{eq: sdf char 1}, we obtain
\[
C_1\|f\|^2\leq\sum_{\lil}\intn|\sin\theta|^{n}\abs{g_\lambda\star_\theta \bar{f}(s)}^2\diff s\leq C_2\|f\|^2,\quad\text{ for all } f\in\ltn.
\]
This leads to
\[
C_1\|f\|^2\leq |\sin\theta|^{n}\sum_{\lil}\intn\abs{g_\lambda\star_\theta f(s)}^2\diff s\leq C_2\|f\|^2,\quad\text{ for all } f\in\ltn.
\]
Thus,
\[
C_1\lnm\fft(f)\rnm^2\leq|\sin\theta|^{n} \sum_{\lil}\intn\abs{\fft\lp g_\lambda\tc f\rp(\omega)}^2\diff\omega\leq C_2\lnm\fft(f)\rnm^2,\quad\text{ for all } f\in\ltn,
\]
by appealing to Parseval's formula for the FrFT. This implies that
\[
C_1\lnm\fft(f)\rnm^2\leq|\sin\theta|^{n}\sum_{\lil}\intn\abs{\fft(g_\lambda)(\omega)}^2\abs{\fft(f)(\omega)}^2\diff\omega\leq C_2\lnm\fft(f)\rnm^2,
\]
for all $f\in\ltn$, using \eqref{eq: conv F frft}. Hence,
\[
\frac{C_1}{\abs{\sin\theta}^n}\leq\sum_{\lil}\abs{\fft(g_\lambda)(\omega)}^2\leq\frac{C_2}{\abs{\sin\theta}^n},\quad\text{for a.e. }\omega\in\R^n.
\]
Retracing the steps back, we obtain the converse.
\end{proof}

Now, we introduce the deep convolutional neural network connected with the FrFT. We begin with the following definitions.

\begin{definition}\label{def: theta ms}
Let $\lc\Lambda_k\rc_{k\in\N}$ be a sequence of countable index sets. For $k\in\N$, let $\Psi_k^\theta=\{ T_s^\theta \widecheck{g}_{\lambda_k}:s\in\R^n,\,\lambda_k\in\Lambda_k\}$ be a collection of semi-discrete frames for $\ltn$. Further, assume that $\mkt:\ltn\to\ltn$ and $\pkt:\ltn\to\ltn$ be Lipschitz continuous operators such that 
\[
\mkt f=\pkt f=0, ~~\text{ for }f=0,\,k\in\N.
\] 
Then, the sequence of triples 
\[
\Omega^\theta:=\lc\lp\Psi_k^\theta,\,\mkt,\,\pkt\rp\rc
\]
is called a {\it $\theta$-module sequence}.
\end{definition}


\begin{definition}
For $s>0$, we define the \textcolor{black}{{\it $\theta$-dilation operator}} connected with the FrFT as
\begin{equation}
D_s^\theta:=\nch D_s\ch.
\end{equation}
\end{definition}

\begin{example}
The real and imaginary parts of the $\theta$-dilation of the unit box function $\chi_{[-\frac12,\frac12]}\otimes\chi_{[-\frac12,\frac12]}$ and the Gaussian $e^{-t\cdot t}$ are shown in Figure \ref{fig5} for $\theta = \pi/3$,  and $s = 2$.
\begin{figure}[h!]
\begin{center}
\includegraphics[width=5cm, height= 4cm]{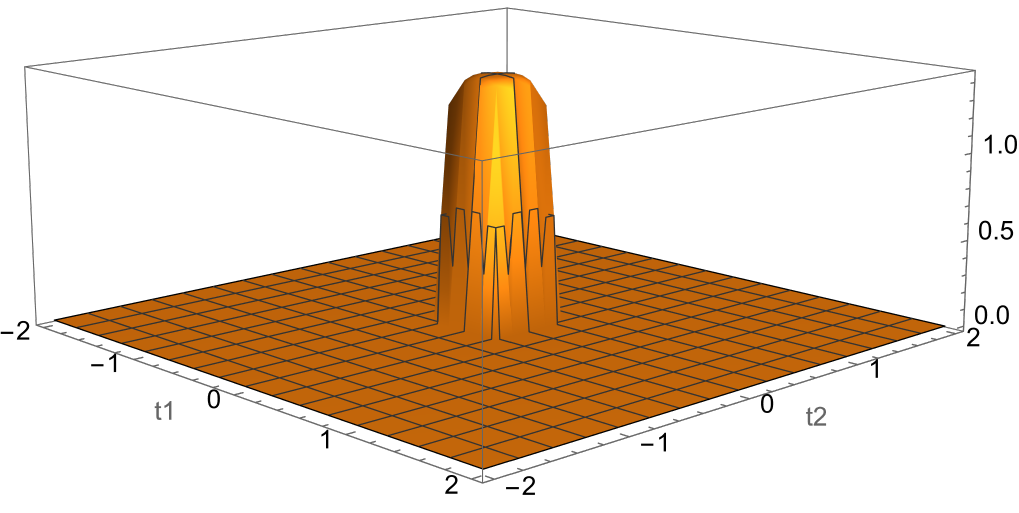}\hspace*{1cm}
\includegraphics[width=5cm, height= 4cm]{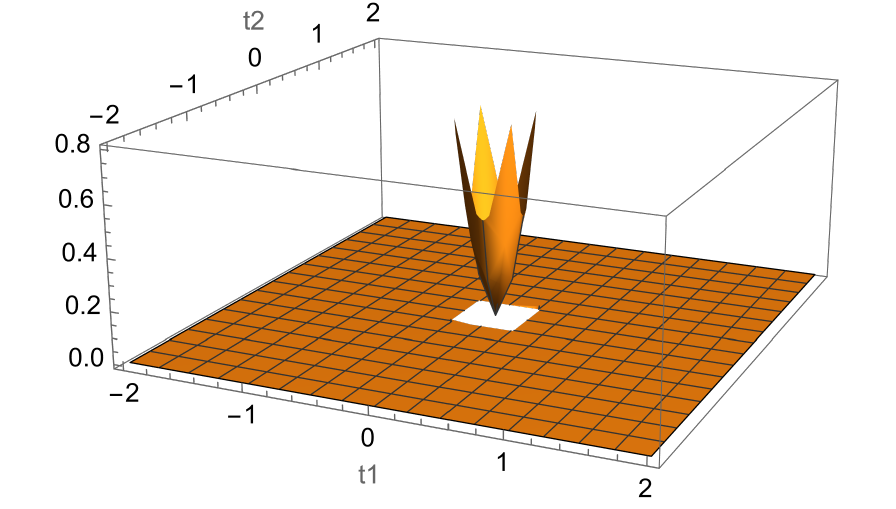}\\
\includegraphics[width=5cm, height= 4cm]{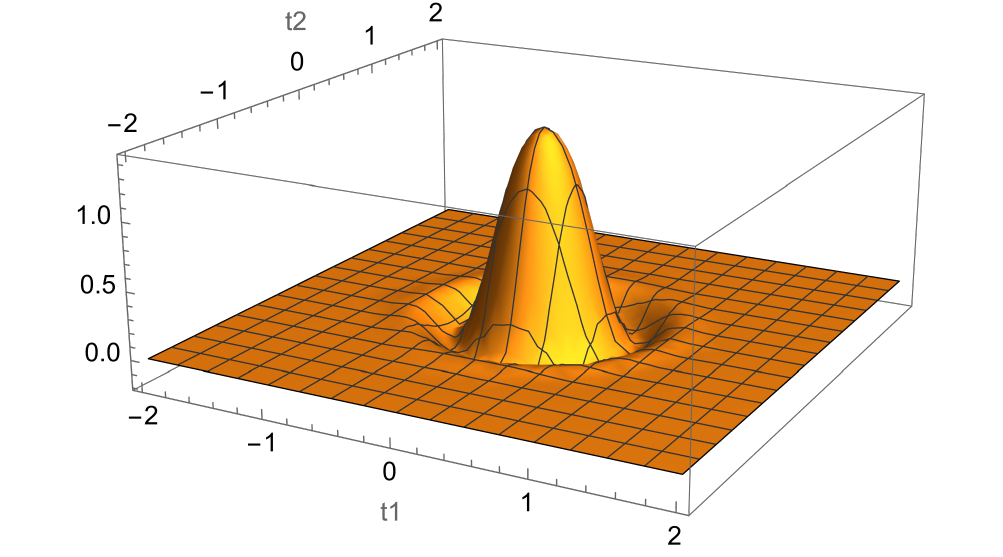}\hspace*{1cm}
\includegraphics[width=5cm, height= 4cm]{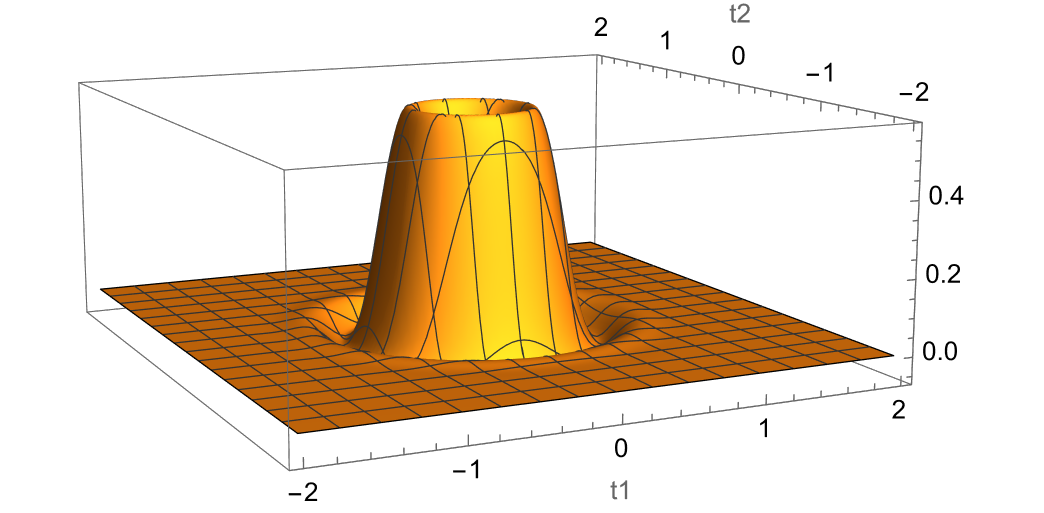}
\caption{The real and imaginary parts of the $\theta$-dilation of the unit box function $\chi_{[-\frac12,\frac12]}\times\chi_{[-\frac12,\frac12]}$ (top) and the Gaussian $e^{-t\cdot t}$  (bottom) for $\theta = \pi/3$,  and $s = 2$.}\label{fig5}
\end{center}
\end{figure}
\end{example}

In the following definition, we use a sequence of real numbers $\{s_k:s_k>1\}$ called {\it polling factors}. For simplicity in notation, we write $D_k^\theta$ instead of $D_{s_k}^\theta$.

\begin{definition}
Let $\Omega^\theta=\lc\lp\Psi_k^\theta,\,\mkt,\,\pkt\rp\rc$ be a $\theta$-module sequence. Assume that $\lc\glk:\lambda_k\in\Lambda_k\rc$ be the atoms of the semi-discrete frame $\Psi_k^\theta$ and $s_k>1$ be the pooling factors associated with the $k^{\mathrm{th}}$ network layer. Then, we define the operator $U_k^\theta$ associated to the $k^{\mathrm{th}}$ layer of the network as
\begin{align*}
&U_k^\theta:\Lambda_k\times\ltn\to\ltn,\\
&U_k^\theta(\lambda_k,\, f):=U^\theta[\lambda_k]f:=D_k^\theta\pkt\mkt\lp f\tc\glk\rp.
\end{align*}  
\end{definition} 

In order to see that the operator $U_k^\theta$ is well-defined, we consider
\begin{align}\label{eq: ukt well defined}
    \|U^\theta[\lambda_k]f\|^2\leq \lnm\pkt\mkt\lp f\tc\glk\rp\rnm^2\leq R_k^2 L_k^2\lnm f\tc\glk\rnm^2\leq B_kR_k^2 L_k^2\|f\|^2,
\end{align}
where $B_k$ is an upper frame bound for the semi-discrete frame $\lc T_s^\theta \glk:s\in\R^n,\,\lambda_k\in\Lambda_k\rc$ and $L_k, R_k$ are the Lipschitz constants for $\mkt$ and $\pkt$, respectively. 


For $1\leq k<\infty$, define the set 
\[
\Lambda^k:=\Lambda_1\times\cdots\times\Lambda_k. 
\]
An ordered tuple $q=(\lambda_1,\cdots,\lambda_k)\in\Lambda^k$ is called a {\it path}.

\begin{definition}
    Let $\mathcal{Q}:=\bigcup\limits_{k=1}^\infty\Lambda^k$. We define the operator $U^\theta:\mathcal{Q}\times \ltn\to\ltn$ by
    \[
    U^\theta[q]f :=U^\theta[\lp \lambda_1,\cdots,\lambda_k\rp]f :=U_k^\theta[\lambda_k]\cdots U_1^\theta[\lambda_1]f,
    \]
    where $q=\lp \lambda_1,\cdots,\lambda_k\rp\in\Lambda^k$ and $f\in\ltn$.
\end{definition}

Appealing to \eqref{eq: ukt well defined}, it is easy to see that the operator $U^\theta$ is well-defined. More precisely,
\begin{equation}\label{eq: U theta well defined}
    \|U^\theta[q]f\|^2\leq \prod_{j=1}^k(B_jL_j^2R_j^2)\|f\|^2,\quad q=(\lambda_1,\cdots, \lambda_k).
\end{equation}
 

\begin{definition}
    Let $\Omega^\theta=\lc\lp\Psi_k^\theta,\,\mkt,\,\pkt\rp\rc$ be a $\theta$-module sequence. Then, the feature extractor $\Phi_{\Omega^\theta}^\theta$ based on $\Omega^\theta$ maps $f\in\ltn$ to its feature vector 
    \[
    \Phi^\theta_{\Omega^\theta}:=\bigcup_{k=0}^\infty \Phi_{\Omega^\theta}^{k,\theta}(f),
    \]
    where 
    \[
    \Phi^{k,\theta}_{\Omega^\theta}(f):=\lc \lp U^\theta[q]f\rp \star_\theta \varphi_k:q\in\Lambda^k\rc
    \]
    and $\varphi_k$ is the output generating atom at the $k$-th network level. 
\end{definition}

\begin{propn}\label{pro: f.e. well defined}
    Let $\Omega^\theta=\lc\lp\Psi_k^\theta,\,\mkt,\,\pkt\rp\rc$ be a $\theta$-module sequence. Let $B_k^\theta$ be an upper frame bound for $\Psi_k^\theta$ and $L_k^\theta,~R_k^\theta$ be the Lipschitz constants for $\mkt$ and $\pkt$, respectively. If 
    \[
    \max\lc B_k^\theta,~B_k^\theta\lp L_k^\theta\rp^2\lp R_k^\theta\rp^2\rc\leq 1,~~\text{ for all }k\in\N,
    \]
    then the feature extractor $\Phi^\theta_{\Omega^\theta}:\ltn\to\lp\ltn\rp^{\mathcal{Q}}$ is well defined.
\end{propn}
\begin{proof}
    The proof is similar to that of Proposition 1 in \cite{Wiatowski} and therefore omitted.
\end{proof}

\begin{remark}
	The claim of Proposition \ref{pro: f.e. well defined} is established by showing $\Phi^\theta_{\Omega^\theta}(f)\in \lp\ltn\rp^{\mathcal{Q}}$ for all $f\in\ltn$, where the space $\lp\ltn\rp^{\mathcal{Q}}$ is equipped with the following norm:
	\[
	\vvvert\{f_q\}\vvvert^2:=\sum_{q\in\Lambda_k}\|f_q\|_{L^2(\R^n)}^2.
	\]  
\end{remark}

\begin{definition}
    Let $\Omega^\theta=\lc\lp\Psi_k^\theta,\,\mkt,\,\pkt\rp\rc$ be a $\theta$-module sequence. Let $B_k^\theta$ be an upper frame bound for $\Psi_k^\theta$ and $L_k^\theta$, $R_k^\theta$ be the Lipschitz constants for $\mkt$ and $\pkt$, respectively. Then, the condition
    \begin{equation}\label{eq: admissible}
    \max\lc B_k^\theta,~B_k^\theta\lp L_k^\theta\rp^2\lp R_k^\theta\rp^2\rc\leq 1,~~\text{ for all }k\in\N,
    \end{equation}
    is called the {\it admissibility condition} and the corresponding module sequence $\Omega^\theta$ is called an {\it admissible $\theta$-module sequence}.
\end{definition}

Now, we are ready to discuss the translation invariance properties of the network defined above. Towards this end, we prove the following two lemmas.

\begin{lemma}
    We have the following equalities.
    \begin{itemize} 
    \item[(i)] If $f, g\in \ltn$, then 
	\begin{equation}\label{eq: translation convolution}    
    T_t^\theta(f\star_\theta g)=(T_t^\theta f)\star_\theta g.
    \end{equation}
    \item[(ii)] For $t\in\R^n$ and $s\neq 0$, one has
    \begin{equation}\label{eq: ds tt theta}
    e^{-\pi it^2\cot\theta}D_s^\theta T_t^\theta=e^{-\pi i\frac{t^2}{s^2}\cot\theta}T^\theta_{t/s}D_s^\theta f.
    \end{equation}
    \end{itemize}
\end{lemma}
\begin{proof}
(i) Consider
\begin{align*}
\ch((T_t^\theta f)\star_\theta g)=\ch(T_t^\theta f)\star\ch g=e^{\pi it^2\cot\theta}(T_t(\ch f))\star \ch g= e^{\pi it^2\cot\theta} T_t(\ch f\star\ch g)=T_t^\theta(f\star_\theta g),
\end{align*}    
where we used \eqref{eq:A-convolution, convolution}, \eqref{eq:T_xT_x^A}, and $T_t(f\star g)=T_tf\star g$, which is easy to verify. Similarly, the other equality follows.\\
  (ii)  Consider
    \begin{align*}
        D_s^\theta T_t^\theta f&=\nch D_s\ch T_t^\theta f=e^{\pi it^2\cot\theta}\nch D_s T_t\ch f=e^{\pi it^2\cot\theta}\nch T_{t/s}D_s\ch f\\
        &=e^{\pi it^2\cot\theta}\nch T_{t/s}\ch\nch D_s\ch f\\
        &=e^{\pi it^2\cot\theta}e^{-\pi i\frac{t^2}{s^2}\cot\theta}T_{t/s}^\theta D_s^\theta f,
    \end{align*}
    where \eqref{eq:T_xT_x^A} and $D_sT_t=T_{t/s}D_s$ were used.
\end{proof}

\begin{lemma}
     Let $\Omega^\theta=\lc\lp\Psi_k^\theta,\,\mkt,\,\pkt\rp\rc$ be a $\theta$-module sequence. Assume that $\lc\glk:\lambda_k\in\Lambda_k\rc$ be the atoms of the frame $\Psi_k^\theta$ and $s_k\geq 1$ be the pooling factors associated to the $k^{\mathrm{th}}$ network layer. For $t\in\R^n$, assume that $T_t^\theta$ commutes with $\mkt$ and $\pkt$, for $1\leq k\leq N$. Then, for $q\in \Lambda^N$, we have
    \begin{equation}\label{eq: lemma 2}
    U^\theta[q]T_t^\theta f=e^{\pi iNt^2\cot\theta}e^{-\pi it^2\lp 1/s_1^2+\cdots+1/s_N^2\rp\cot\theta}U^\theta[q]f.
    \end{equation}
\end{lemma}
\begin{proof}
    For $\lambda_k\in\Lambda_k$, consider
    \begin{align*}
      U_k^\theta[\lambda_k]T_t^\theta f&=D_k^\theta\pkt\,\mkt\lp T_t^\theta f\star_\theta \glk\rp=D_k^\theta \pkt\mkt T_t^\theta\lp f\star_\theta \glk\rp=D_k^\theta T_t^\theta \pkt\mkt\lp f\star_\theta\glk\rp\\
      &=e^{\pi it^2\cot\theta}e^{-\pi i(t^2/s_k^2)\cot\theta}T^\theta_{t/s_k}D_k^\theta\pkt\mkt\lp f\tc\glk\rp\\
      &=e^{\pi it^2\cot\theta}e^{-\pi i(t^2/s_k^2)\cot\theta}T^\theta_{t/s_k}U_k^\theta[\lambda_k]f,
    \end{align*}
    where we used \eqref{eq: translation convolution} and \eqref{eq: ds tt theta}. Thus, for $q\in\Lambda^N$, writing \[U^\theta[q]T_t^\theta f=U_N^\theta[\lambda_N]\cdots U_1^\theta[\lambda_1]T_t^\theta f,\]
    we arrive at the required equality.
\end{proof}

\begin{theorem}\label{th: theorem 1}
    Let $\Omega^\theta=\lc (\Psi_k^\theta,\mkt,\pkt)\rc$ be an admissible $\theta$-module sequence with pooling factors $s_k\geq 1$. Let the operators $\mkt$ and $\pkt$ commute with $T_t^\theta$. Assume that the output generating atoms satisfy
    \begin{equation}\label{eq: th1 decay}
    |\omega\fft(\varphi_k)(\omega)|<K_1,
    \end{equation}
    for a.e. $\omega\in\R^n$, for all $k\in\N$, and for some $K_1>0$. Then, there exists a $K>0$ such that
    \begin{align}\label{eq: th 1 estimate}
    \nonumber&|\|e^{-\pi ikt^2\cot\theta}\Phi_{\Omega^\theta}^{k,\theta}(T_t^\theta f)-\Phi^{k,\theta}_{\Omega^\theta}(f)\||^2\\
    \nonumber\leq &\Big[\frac{t^4}{s_1^4\cdots s_k^4}\cot^2\theta+t^4\lp\frac{1}{s_1^2}+\cdots+\frac{1}{s_k^2}\rp^2\cot^2\theta+\frac{4t^2}{s_1^2\cdots s_k^2}\csc^2\theta\\
    &+4\frac{|t|^{3}}{(s_1\cdots s_k)^3}|\cot\theta\cdot\csc\theta|+4\frac{|t|^{3}}{s_1\cdots s_k}\lp\frac{1}{s_1^2}+\cdots+\frac{1}{s_k^2}\rp|\cot\theta\cdot\csc\theta|\Big]\pi^2 K^2\|f\|^2.
    \end{align}
\end{theorem}
\begin{proof}
  Consider,
  \begin{align*}
      &|\|e^{-\pi ikt^2\cot\theta}\Phi_{\Omega^\theta}^{k,\theta}(T_t^\theta f)-\Phi^{k,\theta}_{\Omega^\theta}(f)\||^2\\
      =&\sum_{q\in\Lambda^k}\lnm e^{-\pi ikt^2\cot\theta}\lp U^\theta[q]T_t^\theta f\rp\tc\varphi_k-\lp U^\theta[q]f\rp\tc\varphi_k\rnm^2\\
      =&\sum_{q\in\Lambda^k}\lnm e^{-\pi it^2\lp \frac{1}{s_1^2}+\cdots+\frac{1}{s_k^2}\rp\cot\theta}\lp T^\theta_{t/s_1\cdots s_k}U^\theta[q]f\rp\tc\varphi_k-\lp U^\theta[q]f\rp\tc\varphi_k\rnm^2\\
      =&\sum_{q\in\Lambda^k}\lnm e^{-\pi it^2\lp \frac{1}{s_1^2}+\cdots+\frac{1}{s_k^2}\rp\cot\theta}\lp T^\theta_{t/s_1\cdots s_k} f_q\rp\tc\varphi_k-f_q\tc\varphi_k\rnm^2\\
      = &\sum_{q\in\Lambda^k}\intn \left|A_{k,\theta}\lp T^\theta_{t/s_1\cdots s_n}f_q\rp\tc\varphi_k(x)-f_q\tc\varphi_k(x)\right|^2\diff x\\
      =&\sum_{q\in\Lambda^k}\intn\left| A_{k,\theta}\fft\lp\lp T^\theta_{t/s_1\cdots s_n}f_q\rp\tc\varphi_k\rp(\omega)-\fft\lp f_q\tc\varphi_k\rp(\omega)\right|^2\diff\omega\\
      =&\sum_{q\in\Lambda^k} \intn \left|A_{k,\theta}\fft\lp T^\theta_{t/s_1\cdots s_k}f_q\rp(\omega)\fft(\varphi_k)(\omega)-\fft(f_q)(\omega)\fft(\varphi_k)(\omega)\right|^2\diff\omega\\
      =&\sum_{q\in\Lambda^k} \intn \left|A_{k,\theta}e^{\pi i\lp\frac{t}{s_1\cdots s_k}\rp^2\cot\theta}e^{-\frac{2\pi it\omega}{s_1\cdots s_k}\csc\theta}-1\right|^2\left|\fft(f_q)(\omega)\fft(\varphi_k)(\omega)\right|^2\diff\omega,
  \end{align*}
  where we used \eqref{eq: lemma 2} and \eqref{eq: preli translation modulation}, and set $f_q:=U^\theta[q]f$ and $A_{k,\theta}:=e^{-\pi it^2\lp \frac{1}{s_1^2}+\cdots +\frac{1}{s_k^2}\rp\cot\theta}$. Let
  \[
  A^\prime_{k,\theta}(\omega):=A_{k,\theta}e^{\pi i\lp\frac{t}{s_1\cdots s_k}\rp^2\cot\theta}e^{-\frac{2\pi it\omega}{s_1\cdots s_k}\csc\theta}=e^{\pi i\xi},
  \]
  where 
  \[
  \xi:=-\underbrace{\left[t^2\lp \frac{1}{s_1^2}+\cdots+\frac{1}{s_k^2}\rp\cot\theta\right]}_{:=\alpha}+\underbrace{\left[\frac{t^2}{s_1^2\cdots s_k^2}\cot\theta\right]}_{:=\beta}-\underbrace{\left[\frac{2\omega t}{s_1\cdots s_k}\csc\theta\right]}_{:=\gamma}.
  \]
  Then,
  \[
  |A^\prime_{k,\theta}(\omega)-1|^2=|e^{\pi i\xi}-1|^2=4\sin^2\lp \pi\xi/2\rp\leq \pi^2\xi^2.
  \]
  Next, we provide an estimate for $\xi^2$.
  \begin{align*}
      \xi^2=\alpha^2+\beta^2+\gamma^2-2\alpha\beta-2\beta\gamma+2\gamma\alpha &\leq \alpha^2+\beta^2 +\gamma^2-2\beta\gamma+2\gamma\alpha\\
      & \leq \alpha^2+\beta^2+\gamma^2+2|\beta\gamma|+2|\gamma\alpha|.
  \end{align*}
  Using the Cauchy-Schwarz inequality on $\R^n$, it is easy to see that
  \begin{align*}
  \gamma^2&\leq \frac{4|\omega|^2|t|^2}{(s_1\cdots s_k)^2}\csc^2\theta,\\
  |\beta\gamma|&\leq \frac{2|t|^3}{(s_1\cdots s_k)^3}|\cot\theta\cdot\csc\theta||\omega|,\\
  |\alpha\gamma|&\leq 2|t|^3\lp \frac{1}{s_1^2}+\cdots+\frac{1}{s_k^2}\rp\frac{1}{s_1\cdots s_k}|\cot\theta\cdot\csc\theta||\omega|.
  \end{align*}
  Hence,
  \begin{align*}
  &\frac{1}{\pi^2}|\|e^{-\pi ikt^2\cot\theta}\Phi^{k,\theta}_{\Omega^\theta}(T^t_\theta f)-\Phi^{k,\theta}_{\Omega^\theta}(f)\||^2\\
  \leq &\alpha^2\sum_{q\in\Lambda^k}\intn|\fft(f_q)(\omega)\fft(\varphi_k)(\omega)|^2\diff\omega:=R_1\\
  &+\beta^2\sum_{q\in\Lambda^k}\intn|\fft(f_q)(\omega)\fft(\varphi_k)(\omega)|^2\diff\omega:=R_2\\
  &+\frac{4t^2|\csc^2\theta|}{(s_1\cdots s_k)^2}\sum_{q\in\Lambda^k}\intn|\omega|^2|\fft(f_q)(\omega)\fft(\varphi_k)(\omega)|^2\diff\omega:=R_3\\
  &+\frac{4|t|^3|\csc\theta\cdot\cot\theta|}{(s_1\cdots s_k)^3}\sum_{q\in\Lambda^k}\intn|\omega||\fft(f_q)(\omega)\fft(\varphi_k)(\omega)|^2\diff\omega:=R_4\\
  &+\frac{4|t|^3|\csc\theta\cdot\cot\theta|}{s_1\cdots s_k}\lp\frac{1}{s_1^2}+\cdots+\frac{1}{s_k^2}\rp\sum_{q\in\Lambda^k}\intn|\omega||\fft(f_q)(\omega)\fft(\varphi_k)(\omega)|^2\diff\omega=:R_5.
  \end{align*}
  As $\varphi_k\in L^1(\R^n)\cap L^2(\R^n)$, appealing to the Riemann-Lebesgue lemma for the FrFT (see, for instance, Lemma 2.2 in \cite{chen2021fractional}), we get $\fft(\varphi_k)\in C_0(\R^n)$. Hence, there exists a $K_2>0$ such that
  \begin{equation}\label{eq: th1 rl decay}
  |\fft(\varphi_k)(\omega)|\leq K_2,~~\text{ for all }\omega\in\R^n.
  \end{equation}
  Further, using \eqref{eq: th1 decay} and \eqref{eq: th1 rl decay}, we obtain
  \[
      |\omega|^{1/2}|\fft(\varphi_k)(\omega)|\leq
      \begin{cases}
          \frac{K_1}{|\omega|^{1/2}}\leq K_1,\quad& |\omega|>1,\\
          K_2,\quad& |\omega|\leq 1.
      \end{cases}
  \]
  Set $K:=\max\{K_1,K_2\}$. Then,
  \begin{align*}
      R_1&\leq \alpha^2K^2\sum_{q\in\Lambda^k}\intn |\fft(f_q)(\omega)|^2\diff\omega=K^2t^4\lp\frac{1}{s_1^2}+\cdots\frac{1} {s_k^2}\rp^2\cot^2\theta\sum_{q\in\Lambda^k}\|f_q\|^2.
  \end{align*}
  Similarly,
  \begin{align*}
      R_2&\leq K^2\frac{t^4}{(s_1\cdots s_k)^2}\cot^2\theta\sum_{q\in\Lambda^k}\|f_q\|^2,\\
      R_3&\leq \frac{4K^2t^2\csc^2\theta}{s_1^2\cdots s_k^2}\sum_{q\in\Lambda^k}\|f_q\|^2,\\
      R_4&\leq \frac{4K^2|t|^{3}|\csc\theta\cdot\cot\theta|}{(s_1\cdots s_k)^3}\sum_{q\in\Lambda^k}\|f_q\|^2,\\
      R_5&\leq \frac{4|t|^{3}K^2|\csc\theta\cdot\cot\theta|}{s_1\cdots s_k}\lp\frac{1}{s_1^2}+\cdots+\frac{1}{s_k^2}\rp\sum_{q\in\Lambda^k}\|f_q\|^2.
  \end{align*}
  Finally, proceeding as in the classical case (cf. Appendix I in \cite{Wiatowski}), we arrive at
  \begin{equation}\label{eq: DCNN 1}
  \sum_{q\in\Lambda^k}\|f_q\|^2\leq\sum_{q\in \Lambda^{k-1}}\|f_q\|^2.
  \end{equation}
  By repeated use of the inequality \eqref{eq: DCNN 1}, we obtain
  \[
  \sum_{q\in\Lambda^k}\|f_q\|^2\leq\|f\|^2,
  \]
  from which the result follows.
\end{proof}

\begin{corollary}
For an appropriate choice of scaling factors $\{s_k\}$, the right-hand side of \eqref{eq: th 1 estimate} can be made arbitrarily small.
\end{corollary}
\begin{proof}
    Let $\{x_k\}$ be a square-summable sequence of positive real numbers. Let $\epsilon>0$. Then, choose $N\in\N$ such that
    \[
    \frac{1}{x^2_{N+1}}+\frac{1}{x^2_{N+2}}+\cdots <\epsilon.
    \]
    Take $s_k:=x_{N+k}$. Then, we can choose an $N^\prime\in\N$ such that
    \[
    |\|e^{-\pi ikt^2\cot\theta}\Phi^{k,\theta}_{\Omega^\theta}\lp T_t^\theta f\rp-\Phi_{\Omega^\theta}^{k,\theta}(f)\||<\epsilon.
    \]
\end{proof}

\begin{corollary}\label{cor: translation invariance}
    Let $\Omega^\theta=\lc \Psi_k^\theta,\mkt,\pkt\rc$ be an admissible $\theta$-module sequence. Let $s_k\geq 1$ be the pooling factors and $\mkt$ and $\pkt$ commute with $T_t^\theta$. Assume that the output generating atoms satisfy
    \begin{equation*}
    |\omega\fft(\varphi_k)(\omega)|<K_1,
    \end{equation*}
    for a.e. $\omega\in\R^n$ and $k\in\N$. Then,
    \begin{align*}
        &|\|e^{-\pi ikt^2\cot\theta}\Phi_{\Omega^\theta}^{k,\theta}(T_t^\theta f)-T_t^\theta\lp\Phi_{\Omega^\theta}^{k,\theta}(f)\rp\||^2\\
        \leq & \Big[t^4\lp\frac{1}{s_1^2}+\cdots+\frac{1}{s_k^2}\rp^2\cot^2\theta+t^4\lp 1-\frac{1}{s_1^2\cdots s_k^2}\rp\cot^2\theta\\
        &+4t^2\lp 1-\frac{1}{s_1\cdots s_k}\rp^2\csc^2\theta+4|t|^{5/2}|\cot\theta\cdot\csc\theta|\lp\frac{1}{s_1^2}+\cdots+\frac{1}{s_k^2}\rp\\
        &+4|t|^{5/2}|\cot\theta\cdot\csc\theta|\lp 1-\frac{1}{s_1^2\cdots s_k^2}\rp\lp 1-\frac{1}{s_1\cdots s_k}\rp\\
        &+2t^4\lp \frac{1}{s_1^2}+\cdots+\frac{1}{s_k^2}\rp\lp 1-\frac{1}{s_1^2\cdots s_k^2}\rp|\csc\theta|\cot^2\theta\Big] \pi^2K^2\|f\|^2.
    \end{align*}
\end{corollary}
\begin{proof}
    The proof is similar to that of Theorem \ref{th: theorem 1} and left to the reader.
\end{proof}

\section{Data approximation problem using the FrFT}
The aim of this section is to study the data approximation problem using the FrFT. In the first part, given a finite set of functional data, $\fl=\{f_1,\cdots,f_m\}$ in $L^2(\R^n)$, we construct a $\theta$-shift invariant subspace of $L^2(\R^n)$ which approximates $\fl$. In the final part, we obtain a subspace of bandlimited functions in the FrFT domain which approximates $\fl$. Towards this end, we first define and study some fundamental properties of the fiber map in the context of the FrFT. Throughout this section, we write 
\[
I:=[0,|\sin\theta|]^n.
\]
Let $f\in L^2(\R^n)$. Then, appealing to the Plancherel theorem for the FrFT, we obtain
\begin{align}\label{eq: fiber defn}
    \nonumber\|f\|^2=\|\fft(f)\|^2&=\int_I\sum_{k\in\Z^n}|\fft (f)(\omega+k\sin\theta)|^2\diff\omega\\
    \nonumber&=\int_I\sum_{k\in\Z^n}|e^{\pi i(\omega+k\sin\theta)^2}(\ch f)~\widehat{}~\lp(\omega+k\sin\theta)\csc\theta\rp|^2\diff\omega\\
    &=\int_I\sum_{k\in\Z^n}|(\ch f)~\widehat{}~(\omega\csc\theta+k)|^2\diff\omega.
\end{align}
This leads us to the following
\begin{definition}
    The {\it fiber map} in connection with the FrFT, denoted by $\tau^\theta$, is defined as
    \[
    \tau^\theta(f)(\omega):=\lc(\ch f)~\widehat{}~(\omega\csc\theta+k)\rc_{k\in\Z^n},\quad\omega\in I,\, f\in L^2(\R^n).
    \]
\end{definition}
Using \eqref{eq: fiber defn}, we can conclude that $\tau^\theta:L^2(\R^n)\to L^2\lp I,\ell^2(\Z^n)\rp$ is an isometric isomorphism. Moreover, for $k\in\Z^n$, $\omega\in I$, we obtain
\begin{align}\label{eq: fiber translation}
    \nonumber\tau^\theta\lp T_k^\theta f\rp(\omega)&=\lc \lp\ch T_k^\theta f\rp~\widehat{}~(\omega\csc\theta+\ell)\rc_{\ell\in\Z^n}=e^{\pi ik^2\cot\theta}\lc (T_k\ch f)~\widehat{}~(\omega\csc\theta+\ell)\rc_{\ell\in\Z^n}\\
    &=e^{\pi ik^2\cot\theta}e^{-2\pi ik\omega\csc\theta}\tau^\theta(f)(\omega),
\end{align}
using \eqref{eq:T_xT_x^A}. Let $V\subset L^2(\R^n)$ be a closed subspace. Then, for $\omega\in I$, we write
\[
J_V^\theta(\omega)=\cl_{\ell^2(\Z^n)}{\lc\tau^\theta(f)(\omega):f\in V\rc}=\lc \lc (\ch f)~\widehat{}~(\omega\csc\theta+\ell)\rc_{\ell\in\Z^n}:f\in V\rc\subset\ell^2(\Z^n).
\]
\begin{lemma}
    Let $V\subset L^2(\R^n)$ be a finitely generated $\theta$-shift invariant space. Then, we have the following.
    \begin{itemize}
        \item[(i)] $J_V^\theta(\omega)$ is a closed subspace of $\ell^2(\Z^n)$, for a.e. $\omega\in I$.
        \item[(ii)] $V=\cl_{L^2(\R^n)}\Span\{ f\in L^2(\R^n):\tau^\theta(f)(\omega)\in J_V^\theta(\omega)\}$, a.e. $\omega\in I$.
        \item[(iii)] \textcolor{black}{Let $P_V: L^2(\R^n)\to L^2(\R^n)$ denote the orthogonal projection onto $V$}. For each $f\in L^2(\R^n)$, the map $\omega\mapsto \|\tau^\theta(P_Vf)(\omega)\|_{\ell^2}$, is measurable and 
        \begin{equation}\label{eq: tau theta proj}
        \tau^\theta(P_Vf)(\omega)=P_{J^\theta_V(\omega)}(\tau^\theta (f)(\omega)),\quad\text{ a.e. }\omega\in I.
        \end{equation}
        \item[(iv)] Let $\phi_1,\cdots\phi_m\in L^2(\R^n)$. Then,\\
        \begin{itemize}
            \item[(a)] $V=\cl_{L^2(\R^n)}{\Span}\lc T^\theta_k\phi_\ell:~k\in\Z^n,~\ell=1,\cdots, m\rc$ if and only if $\lc \tau^\theta(\phi_\ell)(\omega):\ell=1,\cdots m\rc$ spans $J_V^\theta(\omega)$, for a.e. $\omega\in I$.\\
            \item[(b)] The collection $\lc T^\theta_k\phi_\ell:~k\in\Z^n,~\ell=1,\cdots, m\rc$ is a Riesz basis (resp. frame) for $V$ if and only if $\lc \tau^\theta(\phi_\ell)(\omega):\ell=1,\cdots m\rc$ is a Riesz basis (resp. frame) for $J_V^\theta(\omega)$, for a.e. $\omega\in I$.
        \end{itemize}
    \end{itemize}
\end{lemma}
\begin{proof}
    (i) The closedness of $J_V^\theta(\omega)$ follows from the definition.\\
    (ii) Let 
    \[
    W:=\lc f\in L^2(\R^n):\tau^\theta(f)(\omega)\in J_V^\theta(\omega),~~\text{ a.e. }\omega\in I\rc.
    \]
    Then it is clear from the definition of $J_V^\theta(\omega)$ that $V\subset W$. In order to see the other inclusion, let $g\in V^\perp$. Then, for all $k\in\Z^n$ and $f\in V$,
    \begin{align*}
        0=\lng g, T_k^\theta f\rng_{L^2(\R^n)}&=\lng \tau^\theta(g),\tau^\theta(T^\theta_kf)\rng_{L^2\lp I,\ell^2(\Z^n)\rp}\\
        &=\int_I e^{-\pi ik^2\cot\theta}e^{2\pi ik\omega\csc\theta}\lng \tau^\theta(g)(\omega),\tau^\theta(f)(\omega)\rng_{\ell^2(\Z^n)}\diff\omega,
    \end{align*}
    using \eqref{eq: fiber translation}. In other words, all the Fourier coefficients of the function $\omega\mapsto\lng \tau^\theta(g)(\omega),\tau^\theta(f)(\omega)\rng_{\ell^2(\Z^n)}$ vanish. Thus,
    \[
    \lng \tau^\theta(g)(\omega),\tau^\theta(f)(\omega)\rng_{\ell^2(\Z^n)}=0,~~\text{ for all }f\in V,~\text{ a.e. }\omega\in I,
    \]
    which implies that $g\notin W$.\\
    (iii) Let $f=f_1+f_2\in L^2(\R^n)$, where $f_1\in V$ and $f_2\in V^\perp$. Then 
    \begin{align*}
        \tau^\theta(P_Vf)(\omega)&=\tau^\theta(f_1)(\omega)=\lc (\ch f_1)~\widehat{}~(\omega\csc\theta+k)\rc_{k\in\Z^n}\\
        &=J_V^\theta(\omega)\lp\lc(\ch f_1)~\widehat{}~(\omega\csc\theta+k)\rc_{k\in\Z^n}+\lc(\ch f_2)~\widehat{}~(\omega\csc\theta+k)\rc_{k\in\Z^n}\rp\\
        &=J^\theta_V(\omega)(\tau^\theta(f)(\omega)).
    \end{align*}
    (iv)(a) Assume that $V=\cl_{L^2(\R^n)}{\Span}\lc T_k^\theta\phi_\ell:k\in\Z^n,~\ell=1,\cdots,m\rc$. Take 
    \[
    W:=\Span\{T_k^\theta\phi_\ell:k\in\Z^n,\ell=1,\cdots m\}.
    \]
    Let $f\in W$. Then, there exist $\{c_k^j\}\in c_{00}$, for $j=1,\cdots, m$, such that
    \[
    f=\sum_{k\in\Z^n}\sum_{j=1}^mc_k^j\tau^\theta(T^k_\theta\phi_j).
    \]
    This, however, is same as
    \[
    \tau^\theta(f)(\omega)=\sum_{j=1}^m\sum_{k\in\Z^n} C_k^j e^{\pi ik^2\cot\theta}e^{-2\pi ik\omega\csc\theta}\tau^\theta(\phi_j)(\omega),
    \]
    using \eqref{eq: fiber translation} and yields
    \[
    \{\tau^\theta(f)(\omega):f\in W\}\subset \Span\{\tau^\theta(\phi_j)(\omega):j=1,\cdots, m\},
    \]
    which implies that
    \[
    \cl_{L^2(\R^n)}{\{\tau^\theta(f)(\omega):f\in W\}}\subset \Span\{\tau^\theta(\phi_j)(\omega):j=1,\cdots, m\},
    \]
    as the set in the right-hand side is closed. Similarly, one can show the reverse inclusion. 

    Conversely, assume that $\{\tau^\theta\phi_j(\omega):j=1\cdots,m\}$ spans $J_V^\theta(\omega)$, for a.e. $\omega\in I$. Then, for a fixed $\omega\in I$, and sequences $\{d_k^j\}$, $j=1,\cdots,m$,
    \[
    \sum_{j=1}^m\sum_{k\in\Z^n}d_k^je^{-2\pi ik\omega\csc\theta}e^{\pi ik^2\cot\theta}(\ch \phi_j)~\widehat{}~(\omega\csc\theta+k)\in J_V^\theta(\omega).
    \]
    However, this is same as 
    \[
    \tau_\theta\lp \sum_{j=1}^m\sum_{k\in\Z^n}d_k^jT_k\phi_j\rp(\omega)\in J_V^\theta(\omega).
    \]
    Let $f=\sum\limits_{j=1}^m\sum\limits_{k\in\Z^n}d_k^jT_k\phi_j$. Then, using (ii), it is clear that
    \[
    W=\cl_{L^2(\R^n)}\Span\{T_k^\theta\phi_j:k\in\Z^n,j=1,\cdots,m\}\subset V.
    \]
    Proceeding as in (ii), we can show that $V=W$.\\
    (iv)(b) The proof follows using similar ideas as in Theorem 2.3 in \cite{Bownik}. However, we include an outline of the proof for the characterization of Riesz basis. The proof in the case of frames is similar. For $j=1,\cdots,m$, let $\{a_k^j\}$ be a family of sequences in $c_{00}$. Let 
    \[
    p_j(\omega):=\sum_{k\in\Z^n} a_k^je^{\pi ik^2\cot\theta}e^{-2\pi ik\omega\csc\theta},\quad\omega\in I.
    \]
    Consider
    \begin{align*}
        \lnm \sum_{j=1}^m\sum_{k\in\Z^n} a_k^jT_k^\theta\phi_j\rnm^2_{\ltn}&=\lnm \sum_{j=1}^m\sum_{k\in\Z^n} a_k^jT_k^\theta\phi_j\rnm^2_{L^2\lp I,\ell^2(\Z^n)\rp}=\int_I\lnm \sum_{j=1}^mp_j(\omega)\tau^\theta(\phi_j)(\omega)\rnm^2_{\ell^2(\Z^n)}\diff\omega.
    \end{align*}
    Let us assume that $\{\tau^\theta(\phi_j)(\omega):j=1,\cdots,m\}$ is a Riesz basis with bounds $C_1, C_2$ for $J_V^\theta(\omega)$ for a.e. $\omega\in I$. Then
    \[
    C_1\sum_{j=1}^m|p_j(\omega)|^2\leq \lnm \sum_{j=1}^mp_j(\omega)\tau^\theta(\phi_j)(\omega)\rnm^2_{\ell^2(\Z^n)}\leq C_2\sum_{j=1}^m|p_j(\omega)|^2,~~\text{ a.e. }\omega\in I.
    \]
    Now, integrating over $I$, and using the fact that
    \[
    \int_I|p_j(\omega)|^2\diff\omega=\sum_{k\in\Z^n}|a_k^j|^2, \quad j=1,\cdots,m,
    \]
    it is easy to see that $\{T_k^\theta\phi_j:k\in\Z^n, j=1,\cdots,m\}$ is a Riesz sequence with the same bounds. The completeness of the system follows from part (a) and the given hypothesis.\\
    Conversely, assume that $\{T_k^\theta\phi_j:k\in\Z^n, j=1,\cdots,m\}$ is a Riesz basis for $V$ with Riesz bounds $C_1,C_2$. Then,
    \[
    C_1\sum_{j=1}^m\sum_{k\in\Z^n}|a_k^j|^2\leq\lnm\sum_{j=1}^m\sum_{k\in\Z^n}a_k^jT_k^\theta\phi_j\rnm^2\leq C_2\sum_{j=1}^m\sum_{k\in\Z^n}|a_k^j|^2,
    \]
    holds for every $\{a_k^j\}_{k\in\Z^n}\in c_{00}$, for $j=1,\cdots,m$. Now, proceeding as in the necessary part, we obtain that
    \[
    C_1\int_I\sum_{j=1}^m|p_j(\omega)|^2\diff\omega\leq \int_I\lnm \sum_{j=1}^mp_j(\omega)\tau^\theta(\phi_j)(\omega)\rnm^2_{\ell^2(\Z^n)}\diff\omega\leq C_2\int_I \sum_{j=1}^m|p_j(\omega)|^2\diff\omega,
    \]
    holds for every trigonometric polynomial $p_j$. Using Lusin's theorem and the dominated convergence theorem, we can show that
    \[
    C_1\int_I\sum_{j=1}^m|m_j(\omega)|^2\diff\omega\leq \int_I\lnm \sum_{j=1}^mm_j(\omega)\tau^\theta(\phi_j)(\omega)\rnm^2_{\ell^2(\Z^n)}\diff\omega\leq C_2\int_I\sum_{j=1}^m|m_j(\omega)|^2\diff\omega,
    \]
    holds for any family $\{m_j\in L^\infty(I):j=1,\cdots,m\}$. Then, proceeding as in Theorem 2.3 in \cite{Bownik} one can show that
    \[
    C_1\sum_{j=1}^m|c_j|^2\leq\lnm\sum_{j=1}^m c_j\tau^\theta(\phi_j)(\omega)\rnm^2_{\ell^2}\leq C_2\sum_{j=1}^m|c_j|^2,
    \]
    for all $\{c_j\}\in\C^m$, a.e. $\omega\in I$. 
\end{proof}

\subsection{Data approximation using $\theta$-shift invariant spaces}

Let $\fl=\{f_1,\cdots,f_m\}\subset\ltn$ be a functional data set. We construct a $\theta$-shift invariant space $V$ such that $\sum\limits_{j=1}^m\|f_j-P_Vf_j\|^2$ is minimal. Towards this end, we make use of the following results. 
\begin{theorem}[\cite{AldroubiACHA07}]\label{th: aldroubi}
    Given is a data set $\mathcal{F}=\{f_1,\cdots, f_m\}$ in an infinite dimensional Hilbert space $\Hi$ and $\ell\leq m$. Let $\lambda_1\geq\cdots\geq\lambda_m\geq 0$ be the eigenvalues of the matrix $[\mathbf{B}(\mathcal{F})]_{i,j}=\lng f_i,f_j\rng$ and $y_1,\cdots y_m\in\C^m$ with $y_i=(y_{i1},\cdots,y_{im})^t$ be orthonormal eigenvectors associated to the eigenvalues $\lambda_1,\cdots,\lambda_m$. Define
    \[
    q_i=\Tilde{\sigma_i}\sum_{j=1}^my_{ij}f_j,\quad i=1,\cdots,\ell,
    \]
    where $\Tilde{\sigma_i}=\lambda_i^{-1/2}$ if $\lambda\neq 0$ and $\Tilde{\sigma_i}=0$, otherwise. Then $\{q_1,\cdots,q_\ell\}$ is a Parseval frame for $M:=\Span\{q_1,\cdots,q_\ell\}$ and the subspace $M$ is optimal in the sense that
    \[
    \sum_{i=1}^m\|f_i-P_Mf_i\|^2\leq \sum_{i=1}^m\|f_i-P_{M^\prime}f_i\|^2,~~\text{ for all }M^\prime\subset\Hi,~~\text{ with } dim(M^\prime)\leq \ell.
    \]
    Moreover,
    \[
    E(\mathcal{F},\ell)=\sum_{i=1}^m\|f_i-P_Mf_i\|^2=\sum_{i=\ell+1}^m\lambda_i.
    \]
\end{theorem}

\begin{lemma}[\cite{RonShen95}]\label{lemma: ron shen}
    Let $G=G(\omega)$ be an $m\times m$ self-adjoint matrix of measurable functions on a measurable subset $E\subset\R^n$ with eigenvalues $\lambda_1(\omega)\geq\cdots\geq\lambda_m(\omega)\geq 0$. Then, the eigenvalues $\lambda_i(\omega)$, $i=1,\cdots,m$, are measurable on $E$ and there exists an $m\times m$ matrix of measurable functions $U=U(\omega)$ such that $U^*(\omega)U(\omega)=I$, a.e. $\omega\in E$, and
    \[
    G(\omega)=U(\omega)\Lambda(\omega)U^*(\omega), ~~\text{ a.e. }\omega\in E,
    \]
    where $\Lambda(\omega)=\diag(\lambda_1(\omega),\cdots,\lambda_m(\omega))$.
\end{lemma}

Now we present the main result of this subsection.

\begin{theorem}\label{th: data sis}
    Let $\mathcal{F}=\{f_1,\cdots,f_m\}\subset L^2(\R^n)$ and $\ell\in\N$. Then, there exist $\phi_1,\cdots,\phi_\ell\in L^2(\R^n)$ such that
    \[
    \sum_{i=1}^m\|f_i-P_{V_\theta}f_i\|^2\leq\sum_{i=1}^m\|f_i-P_{W_\theta}f_i\|^2,
    \]
    holds for all $\theta$-shift invariant spaces $W_\theta$ generated by at most $\ell$ elements, where $V_\theta=\cl_{L^2(\R^n)}{\Span}\{T_k^\theta\phi_i:i=1,\cdots,\ell;\,k\in\Z^n\}$. Further, the system $\{T_k^\theta\phi_i:i=1,\cdots,\ell,k\in\Z^n\}$ forms a Parseval frame for $V_\theta$. Moreover,
    \begin{equation}\label{eq: error data approx}
    E(\mathcal{F},\ell):=\sum_{i=1}^m\|f_i-P_{V_\theta}f_i\|^2=\sum_{i=\ell+1}^m\int_I\lambda_i(\omega)\diff\omega,
    \end{equation}
    where $\lambda_1(\omega)\cdots\geq\lambda_m(\omega)\geq 0$ are eigenvalues of the matrix
    \[
    [G^\theta_\mathcal{F}]_{i,j}=\sum_{k\in\Z^n}(\ch f_i)~\widehat{}~(\omega\csc\theta+k)\overline{(\ch f_j)~\widehat{}~(\omega\csc\theta+k)}.
    \]
\end{theorem}
\begin{proof}
    We apply Lemma \ref{lemma: ron shen} to the matrix $G_{\fl}^\theta(\omega)$ of $|\sin\theta|\,\Z^n$-periodic functions and choose a unitary matrix $U(\omega)$ of $|\sin\theta|\,\Z^n$-periodic functions such that
    \begin{equation}\label{eq: sec5 eq1}
    G_{\fl}^\theta(\omega)=U(\omega)\Lambda(\omega)U^*(\omega),~~\omega\in I,
    \end{equation}
    where $\Lambda(\omega)=\diag\lp\lambda_1(\omega),\cdots,\lambda_m(\omega)\rp$. Let $U_i(\omega)$ denote the $i^{th}$ row of $U(\omega)$. Then, multiplying \eqref{eq: sec5 eq1} by $U^*(\omega)$ on the left, we obtain
    \[
    U^*(\omega)G_{\fl}^\theta(\omega)=\Lambda(\omega)U^*(\omega).
    \]
    This shows that $y_i(\omega)=U_i^*(\omega)$ is a left eigenvector of $G_{\fl}^\theta(\omega)$ associated to the eigenvalue $\lambda_i(\omega)$. Moreover, the left eigenvectors $y_i(\omega)=(y_{i1}(\omega),\cdots,y_{im}(\omega))^\top$, $i=1,\cdots,m$, form an orthonormal basis for $\C^m$.\\
    Now for each $\omega\in I$, apply Theorem \ref{th: aldroubi} to the Hilbert space $\ell^2(\Z^n)$ with the data set
    \[
    \fl_\omega=\{\tau^\theta f_1(\omega),\cdots,\tau^\theta f_m(\omega)\}.
    \]
    Define,
    \[
    q_i(\omega):=\Tilde{\sigma_i}(\omega)\sum_{j=1}^my_{i,j}(\omega)\tau^\theta f_j(\omega),\quad i=1,\cdots,\ell,
    \]
    where $\Tilde{\sigma_i}(\omega)=\lambda_i(\omega)^{-1/2}$, if $\lambda_i(\omega)\neq 0$ and $\Tilde{\sigma_i}(\omega)=0$, otherwise.
    Then, $S_\omega^\theta:=\Span\{q_1(\omega),\cdots,q_\ell(\omega)\}$ satisfies
    \begin{equation}\label{eq: sec 5 eq2}
    \sum_{j=1}^m\lnm\tau^\theta (f_j)(\omega)-P_{S_\omega^\theta}\tau^\theta(f_j)(\omega)\rnm^2\leq\sum_{j=1}^m\lnm\tau^\theta(f_j)(\omega)-P_W\tau^\theta(f_j)(\omega)\rnm^2,
    \end{equation}
    for any subspace $W$ of $\ell^2(\Z^n)$ with $\dim(W)\leq \ell$,
    \begin{equation}\label{eq: sec 5 eq3}
    \sum_{j=1}^m\lnm\tau^\theta (f_j)(\omega)-P_{S_\omega^\theta}\tau^\theta(f_j)(\omega)\rnm^2=\sum_{i=\ell+1}^m\lambda_i(\omega),
    \end{equation}
    and the collection
    $\{q_1(\omega),\cdots,q_\ell(\omega)\}$ forms a Parseval frame for $S_\omega^\theta$.\\
    Define
    \[
    h_i(\omega):=\Tilde{\sigma_i}(\omega)\sum_{j=1}^my_{ij}(\omega)(\ch f_j)~\widehat{}~(\omega\csc\theta),~~i=1,\cdots,\ell.
    \]
    Then, $h_i\in L^2(\R^n)$. In fact,
    \[
    \intn|h_i(\omega)|^2\diff\omega=\int_I\sum_{k\in\Z^n}|h_i(\omega+k\sin\theta)|^2\diff\omega
    \]
    and
    \begin{align*}
        \sum_{k\in\Z^n}|h_i(\omega+k\sin\theta)|^2&=\sum_{k\in\Z^n}h_i(\omega+k\sin\theta)\overline{h_i(\omega+k\sin\theta)}=\Tilde{\sigma_i}^2(\omega)\sum_{j=1}^my_{ij}(\omega)\sum_{s=1}^m\overline{y_{is}(\omega)}[G_\fl^\theta(\omega)]_{js}\\
        &=\lambda_i(\omega)\Tilde{\sigma}(\omega)=\bigchi_{\lc\omega:\lambda_i(\omega)>0\rc},
    \end{align*}
    using the fact that $(\overline{y_{i1}},\cdots,\overline{y_{ij}})^\top$ is a left eigenvector of $G_\fl^\theta(\omega)$ with the eigenvalue $\lambda_i(\omega)$.\\
    Now choose $\phi_i\in \ltn$ such that
    \[
    (\ch\phi_i)~\widehat{}~(\omega)=h_i(\omega\sin\theta).
    \]
    Let $V=\cl_{L^2(\R^n)}{\Span}\{T_k^\theta\phi_i:k\in\Z^n,~i=1,\cdots,\ell\}$. Then $\{\tau_\theta\phi_i(\omega):i=1,\cdots,\ell\}$ spans $J_V^\theta(\omega)$. Further,
    \[
    \tau^\theta(\phi_i)(\omega)(k)=(\ch \phi_i)~\widehat{}~(\omega\csc\theta+k)=h_i(\omega+k\sin\theta)=q_i(\omega)(k).
    \]
    This shows that $J_V^\theta(\omega)=S^\theta_\omega$. Moreover, integrating \eqref{eq: sec 5 eq2}, \eqref{eq: sec 5 eq3}, we get the optimality of the space $V$ and the error estimate. Finally, using the fact that $\{\tau^\theta(\phi_i)(\omega):i=1,\cdots,\ell\}$ is a Parseval frame for $J_V^\theta(\omega)$, for a.e. $\omega\in I$, we conclude that $\{T_k^\theta\phi_i:k\in\Z^n,~i=1,\cdots,\ell\}$ is a Parseval frame for $V$.
\end{proof}

\begin{example}\label{ex4.6}
(i) Let $n=1$. For $m\in\N$, consider 
\[
f_m(x)=me^{-\pi ix^2\cot\theta}e^{\pi im^2x}\sinc (mx),~~x\in\R.
\]

\begin{figure}[h!]
\begin{center}
\includegraphics[width=4cm, height= 3cm]{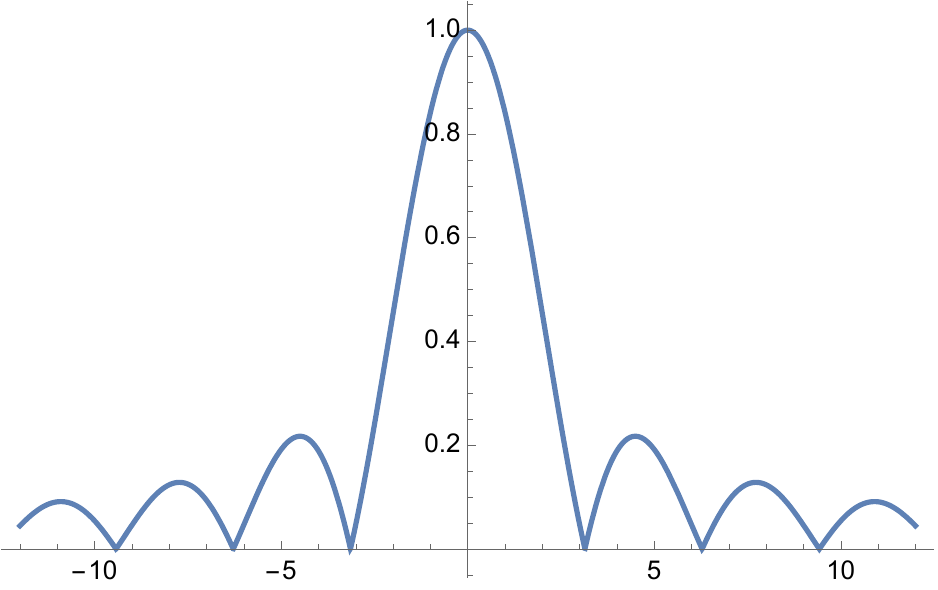}\hspace*{1cm}
\includegraphics[width=4cm, height= 3cm]{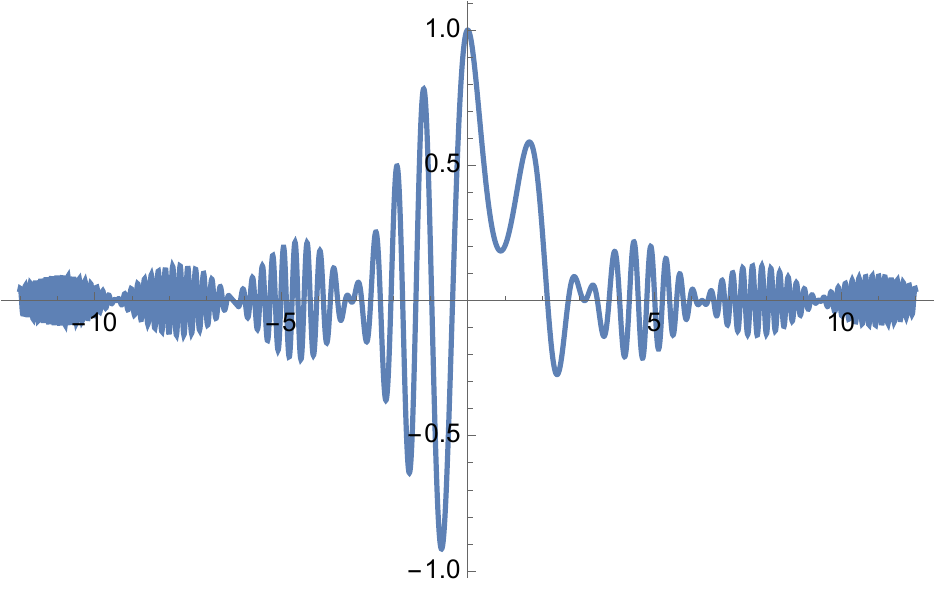}\hspace*{1cm}
\includegraphics[width=4cm, height= 3cm]{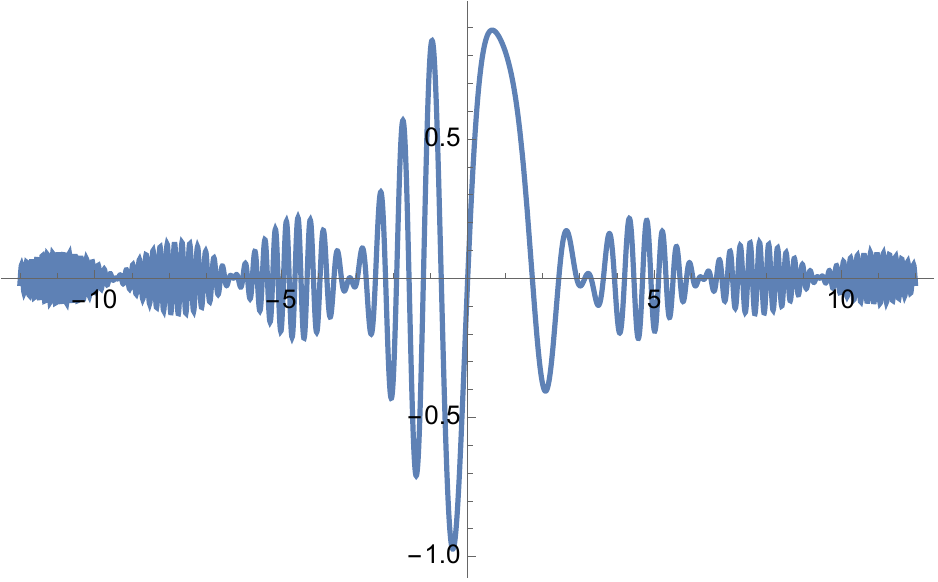}\vspace*{0.5cm}
\includegraphics[width=4cm, height= 3cm]{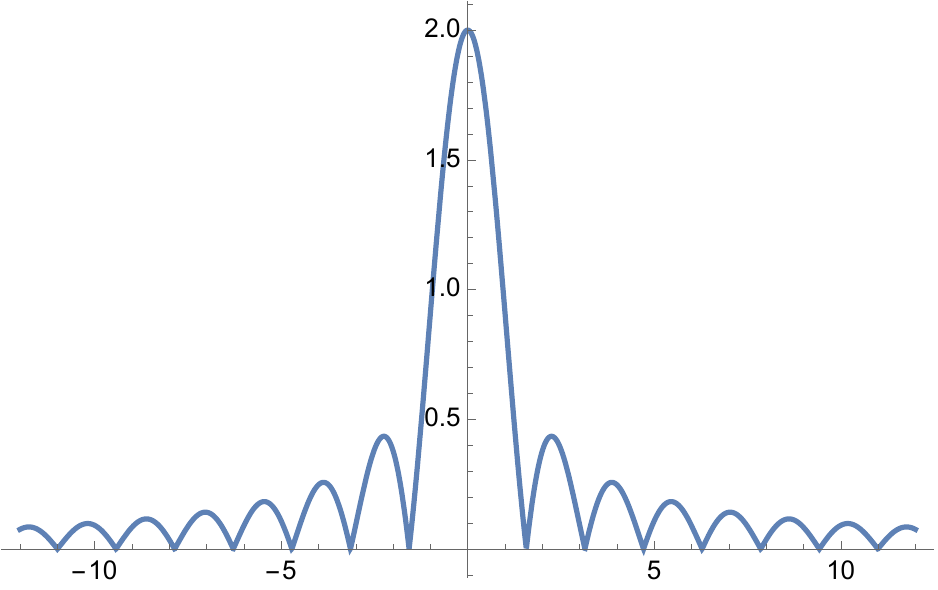}\hspace*{1cm}
\includegraphics[width=4cm, height= 3cm]{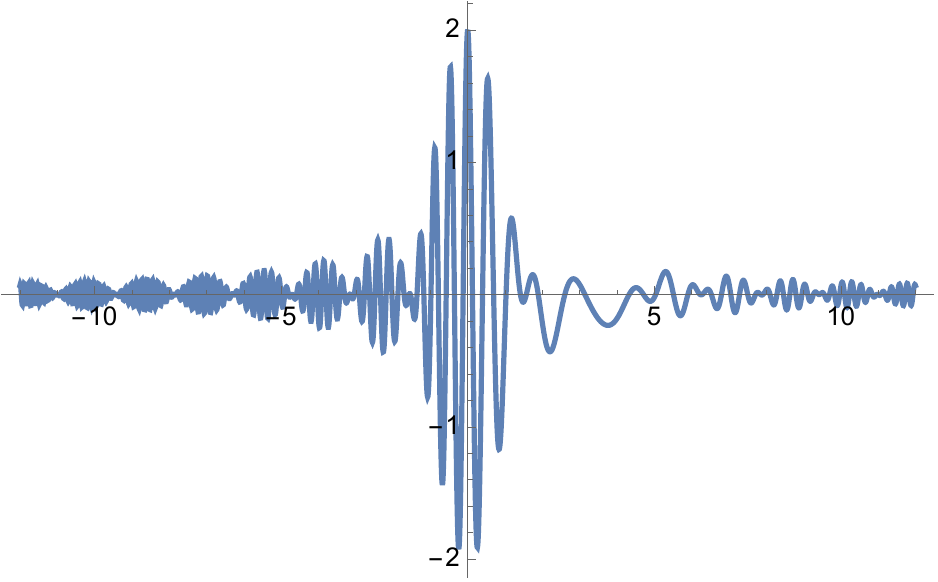}\hspace*{1cm}
\includegraphics[width=4cm, height= 3cm]{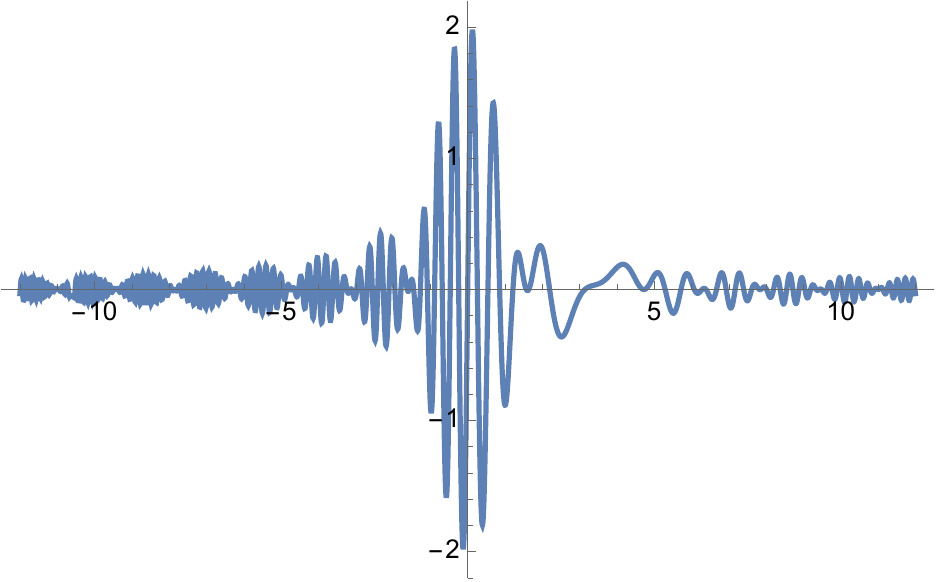}\vspace*{0.5cm}
\includegraphics[width=4cm, height= 3cm]{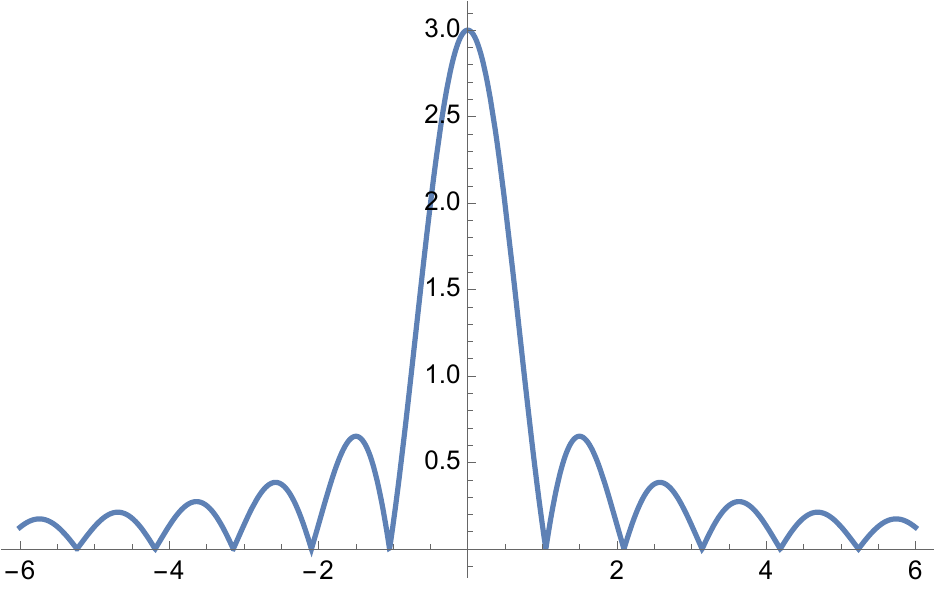}\hspace*{1cm}
\includegraphics[width=4cm, height= 3cm]{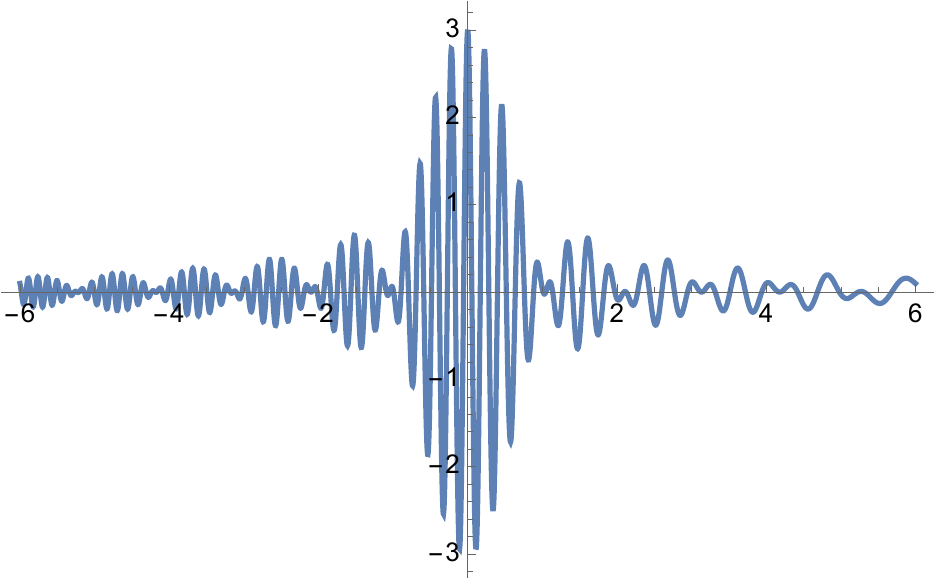}\hspace*{1cm}
\includegraphics[width=4cm, height= 3cm]{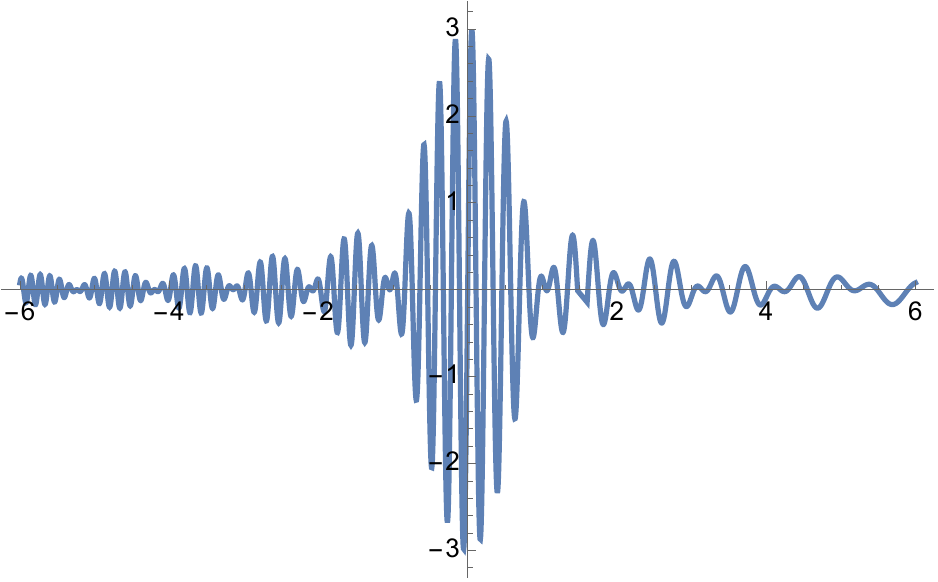}\\
\caption{The absolute value and the real and imaginary parts of $f_1$ (top), $f_2$ (middle), and  $f_3$  (bottom) for $\theta = \pi/3$.}
\end{center}
\end{figure}

Then, $(\ch f)~\widehat{}~(\omega)=\bigchi_{[0,m)}(\omega)$. Let $G^\theta_{\mathcal{F}_m}(\omega)$ be the Gramian associated with $\fl_m=\{f_1,f_2,\cdots, f_m\}$. Then
\begin{align*}
[G^\theta_{\mathcal{F}_m}(\omega)]_{i,j}&=\sum_{k\in\Z}(\ch f_i)~\widehat{}~(\omega\csc\theta+k)(\ch f_j)~\widehat{}~(\omega\csc\theta+k)\\
&=\sum_{k\in\Z}\bigchi_{[0,i|\sin\theta|)}(\omega+k\sin\theta)\bigchi_{[0,j|\sin\theta|)}(\omega+k\sin\theta)\\
&=\min\{i,j\}.
\end{align*}
Now, take $m=4$. Then,
\[
G^\theta_{\mathcal{F}_4}(\omega)=
\begin{pmatrix}
1 & 1 & 1 & 1\\
1 & 2 & 2 & 2\\
1 & 2 & 3 & 3\\
1 & 2 & 3 & 4
\end{pmatrix}.
\]
The eigenvalues and corresponding eigenvectors of $G^\theta_{\mathcal{F}_4}(\omega)$ are
\begin{align*}
\lambda_1 &= 3+2\sqrt{7}\cos\left(\tfrac13\arctan\left(\tfrac{\sqrt{3}}{37}\right)\right),\\
\lambda_2 &= 1,\\
\lambda_3 &= 3+\sqrt{21} \sin \left(\tfrac{1}{3} \arctan\left(\tfrac{\sqrt{3}}{37}\right)\right)-\sqrt{7} \cos \left(\tfrac{1}{3} \arctan\left(\tfrac{\sqrt{3}}{37}\right)\right),\\
\lambda_4 &= 3-\sqrt{21} \sin \left(\tfrac{1}{3} \arctan\left(\tfrac{\sqrt{3}}{37}\right)\right)-\sqrt{7} \cos \left(\tfrac{1}{3} \arctan\left(\tfrac{\sqrt{3}}{37}\right)\right),\\
\end{align*}
and
\begin{align*}
y_1&= \left(\cos\tfrac\pi9 - \sqrt{3}\sin\tfrac\pi9,\tfrac{1}{2} \left(2-2 \sin \left(\tfrac{\pi
}{18}\right)+\sqrt{3} \sin \left(\tfrac{\pi}{9}\right)-\cos \left(\tfrac{\pi }{9}\right)\right),2 \cos \left(\tfrac{\pi }{9}\right)-1,1\right)^\top,\\
   y_2&= \left(-1,-1,0,1\right)^\top,\\
   y_3 &= \left(\cos\tfrac\pi9 + \sqrt{3}\sin\tfrac\pi9,1-\sqrt{3} \sin \left(\tfrac{\pi }{9}\right)-\cos \left(\tfrac{\pi }{9}\right), \tfrac{1}{2} \left(-2-2 \sin \left(\tfrac{\pi}{18}\right)+\sqrt{3} \sin \left(\tfrac{\pi}{9}\right)-\cos \left(\tfrac{\pi }{9}\right)\right),1\right)^\top,\\
   y_4 &= \left(-2 \cos \left(\tfrac{\pi }{9}\right), 1+2 \cos \left(\tfrac{\pi }{9}\right),-1-\sqrt{3} \sin \left(\tfrac{\pi }{9}\right)-\cos\left(\tfrac{\pi }{9}\right),1\right)^\top
\end{align*}
respectively. We now calculate $\phi_1,~\phi_2, ~\phi_3$ using Theorem \ref{th: data sis}.

\begin{align*}
(\ch \phi_j)~\widehat{}~(\omega)&=\frac{1}{\sqrt{\lambda_j}}\lp y_{j1}(\ch f_1)~\widehat{}~(\omega)+y_{j2}(\ch f_2)~\widehat{}~(\omega)+y_{j3}(\ch f_3)~\widehat{}~(\omega)+y_{j4}(\ch f_4)~\widehat{}~(\omega)\rp.
\end{align*}
Thus
\begin{align*}
\phi_j(x)&=\frac{y_{j1}}{\sqrt{\lambda_j}}f_1(x)+\frac{y_{j2}}{\sqrt{\lambda_j}}f_2(x)+\frac{y_{j3}}{\sqrt{\lambda_j}}f_3(x)+\frac{y_{j4}}{\sqrt{\lambda_j}}f_4(x).
\end{align*}
Hence,
\begin{align*}
\phi_1(x)=& e^{-\pi ix^2\cot\theta} ( 0.1206e^{\pi ix}\sinc (x)+ 0.4534e^{4\pi ix}\sinc (2x)+0.9162 e^{9\pi ix}\sinc (3x)\\
&+1.3892 e^{16\pi ix}\sinc (4x)),\\
\phi_2(x)=&e^{-\pi ix^2\cot\theta} \lp -e^{\pi ix}\sinc (x)-2 e^{4\pi ix}\sinc (2x)+ 4e^{16\pi ix}\sinc (4x)\rp,\\
\phi_3(x)=&e^{-\pi ix^2\cot\theta} ( 2.3473e^{\pi ix}\sinc (x)-1.6304 e^{4\pi ix}\sinc (2x)-6.1925 e^{9\pi ix}\sinc (3x)\\
&+6.1284 e^{16\pi ix}\sinc (4x)).
\end{align*}

\begin{figure}[h!]
\begin{center}
\includegraphics[width=4cm, height= 3cm]{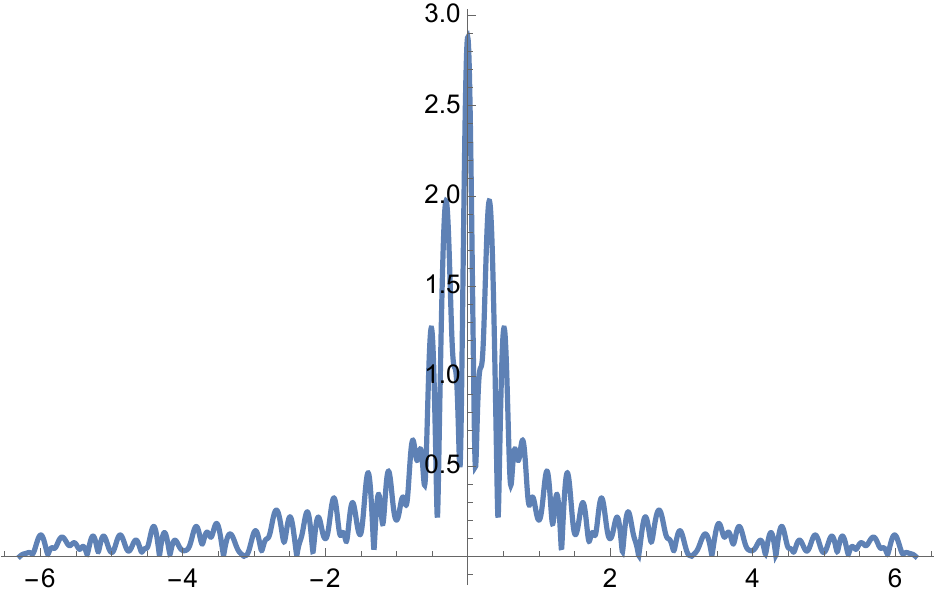}\hspace*{0.5cm}
\includegraphics[width=4cm, height= 3cm]{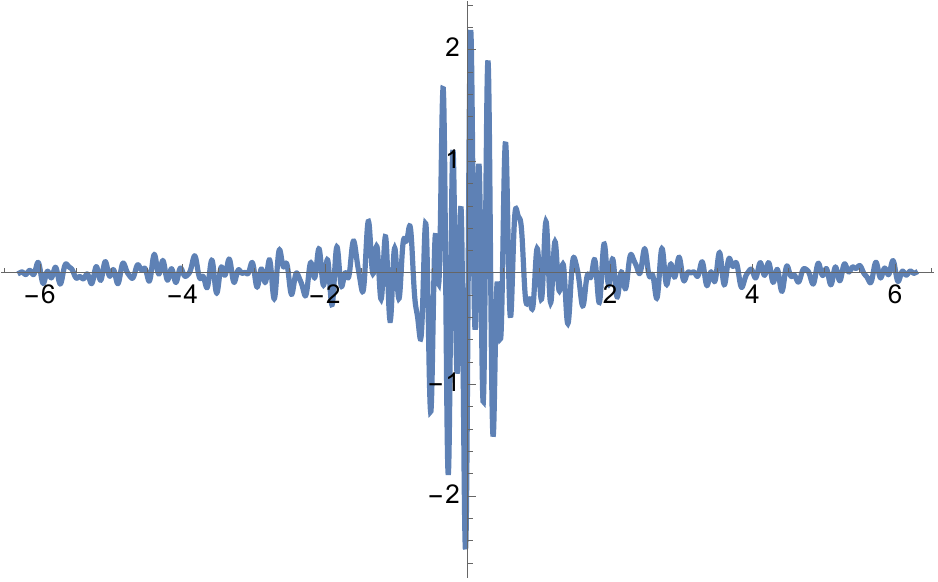}\hspace*{0.5cm}
\includegraphics[width=4cm, height= 3cm]{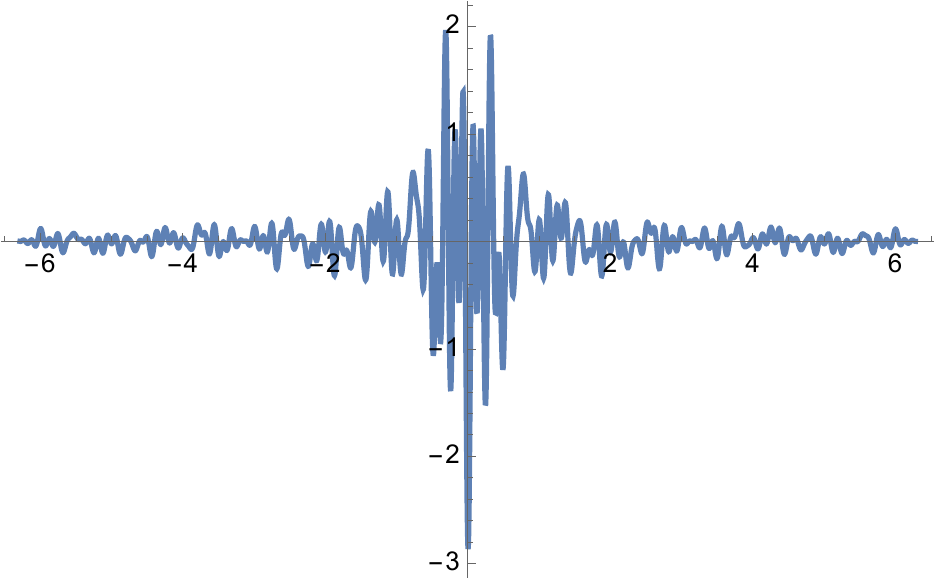}\vspace*{0.5cm}
\includegraphics[width=4cm, height= 3cm]{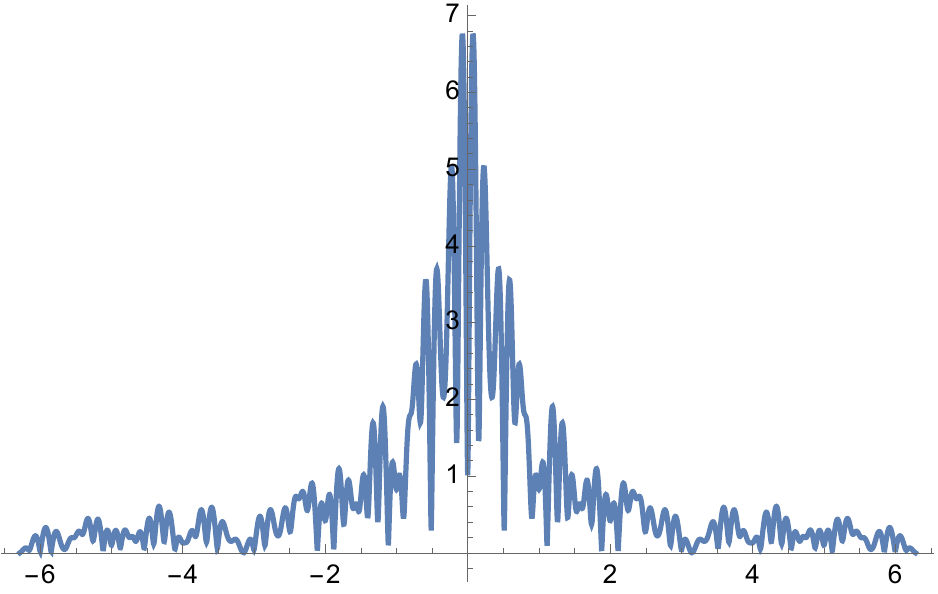}\hspace*{0.5cm}
\includegraphics[width=4cm, height= 3cm]{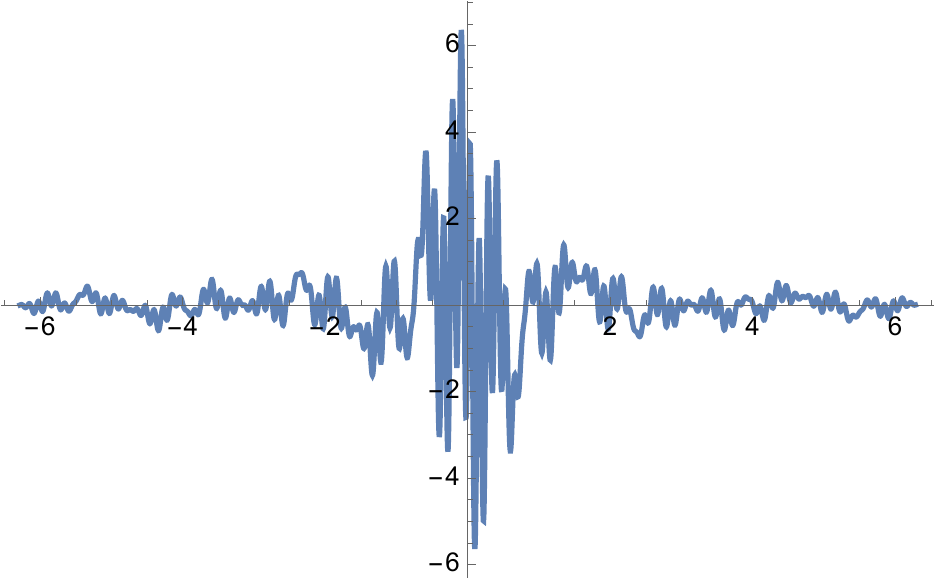}\hspace*{0.5cm}
\includegraphics[width=4cm, height= 3cm]{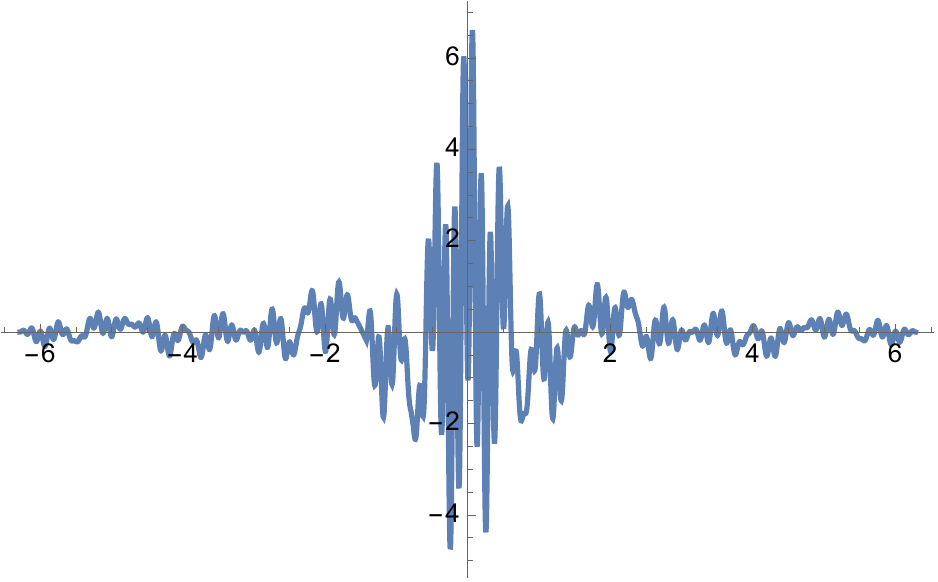}\vspace*{0.5cm}
\includegraphics[width=4cm, height= 3cm]{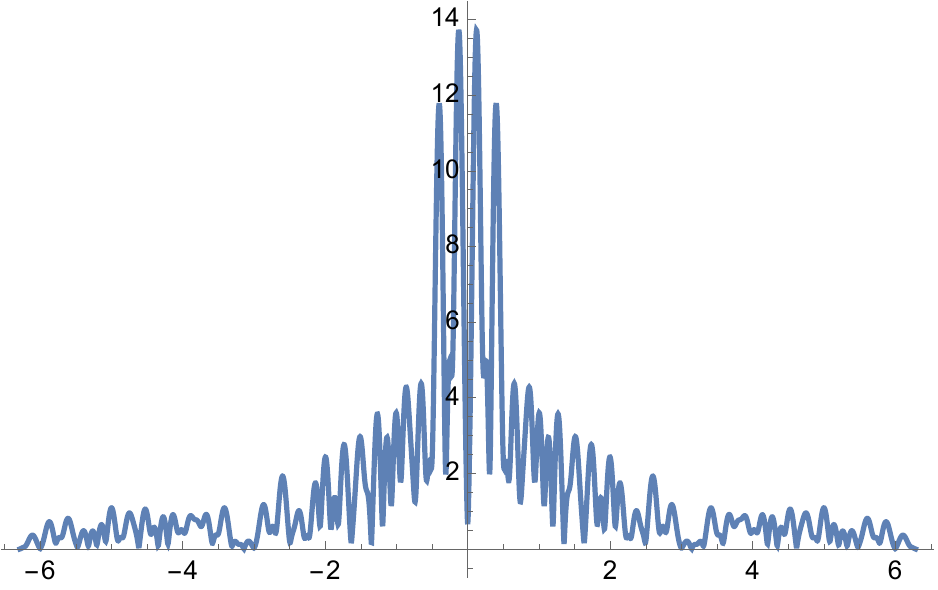}\hspace*{0.5cm}
\includegraphics[width=4cm, height= 3cm]{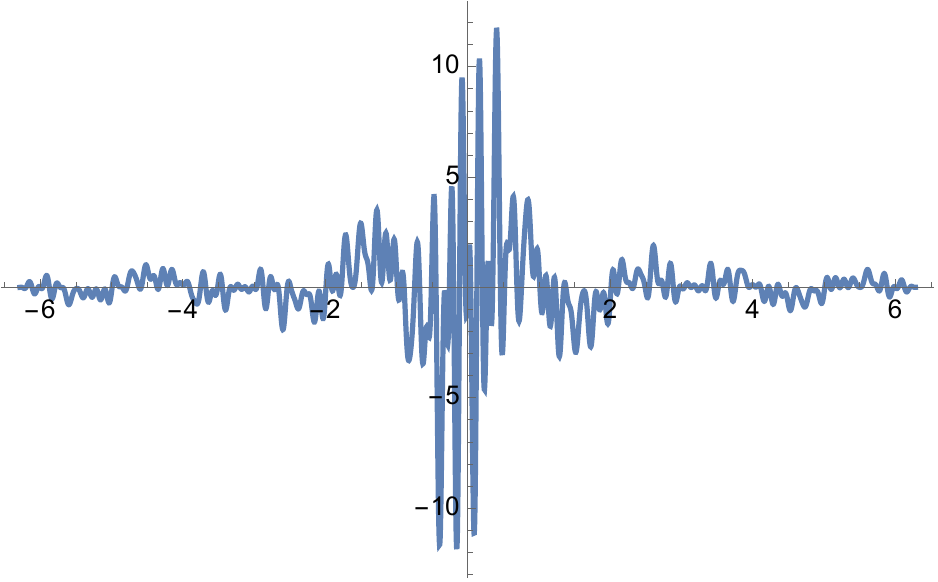}\hspace*{0.5cm}
\includegraphics[width=4cm, height= 3cm]{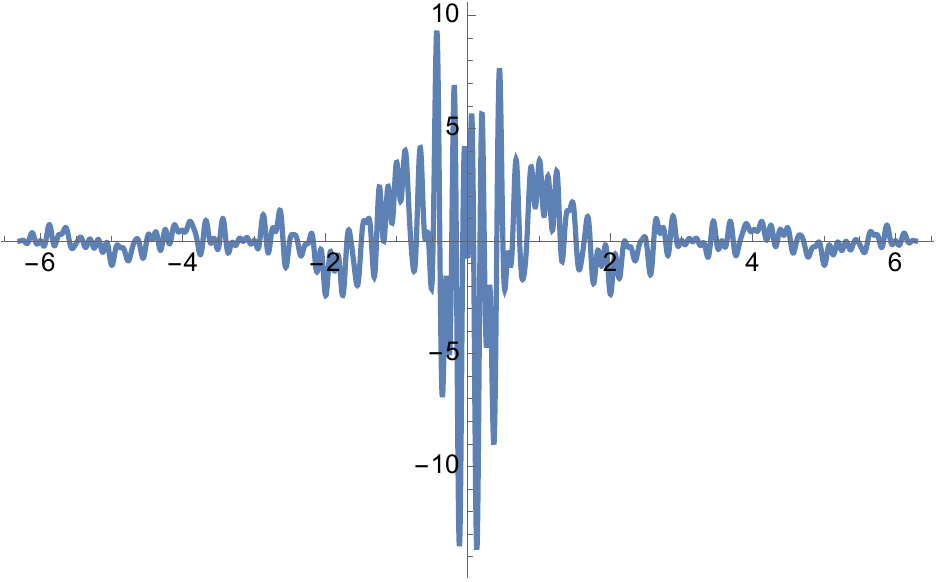}
\caption{The absolute value and the real and imaginary parts of $\phi_1$ (top), $\phi_2$ (middle), and  $\phi_3$  (bottom) for $\theta = \pi/3$.}
\end{center}
\end{figure}

(ii) Let $n=2$. For $m\in\N$, consider
\[
f_m(x, y)=m^2e^{-\pi i(x^2+y^2)\cot\theta}e^{\pi im^2(x+y)}(\sinc (mx))(\sinc (my)),~~(x,y)\in\R^2.
\]
Then $(\ch f)~\widehat{}~(\xi, \omega)=\bigchi_{[0,m)}(\xi)\bigchi_{[0,m)}(\omega)$. Further, by a similar calculation as in the one dimensional case, we obtain $[G^\theta_{\mathcal{F}_4}(\xi,\omega)]_{i,j}=\lp\min \{i,j\}\rp^2$.

\begin{figure}[h!]
\begin{center}
\includegraphics[width=4cm, height= 3.5cm]{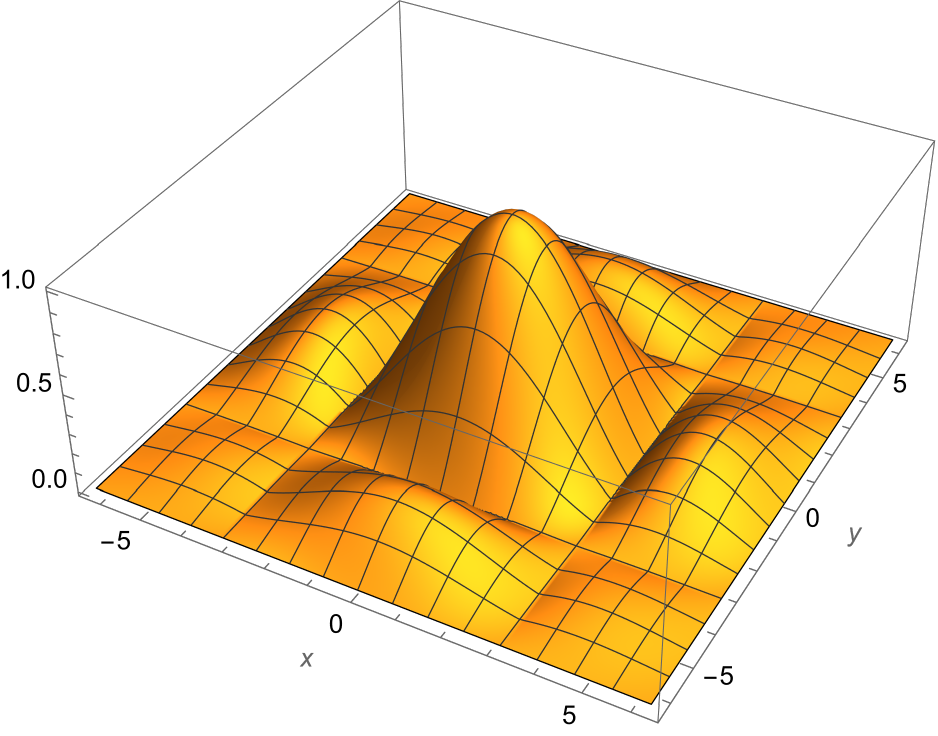}\quad
\includegraphics[width=4cm, height= 3.5cm]{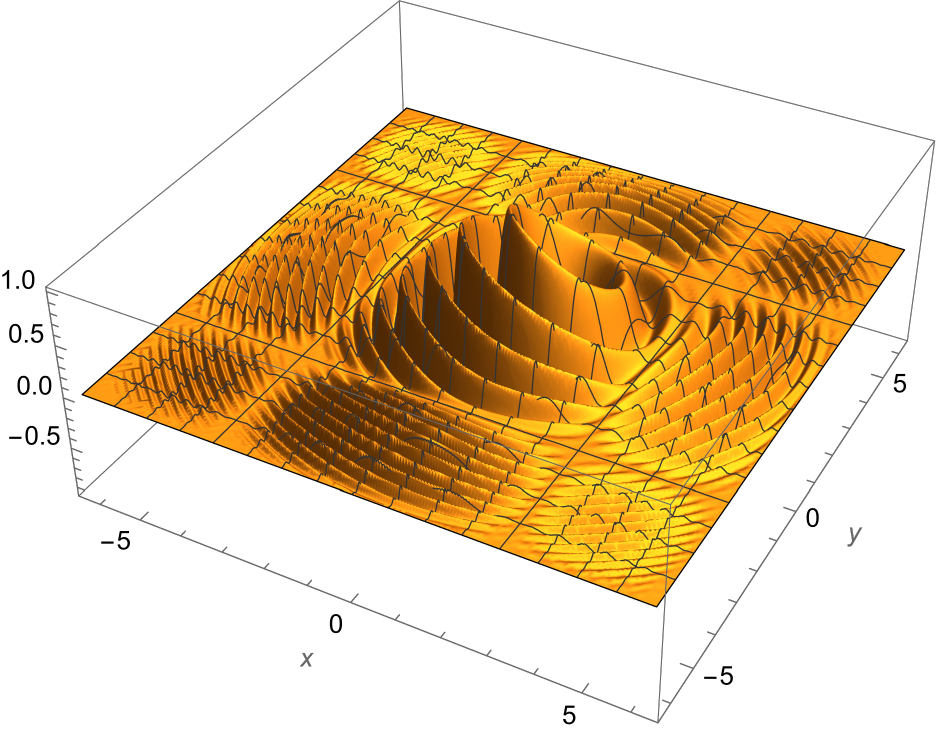}\quad
\includegraphics[width=4cm, height= 3.5cm]{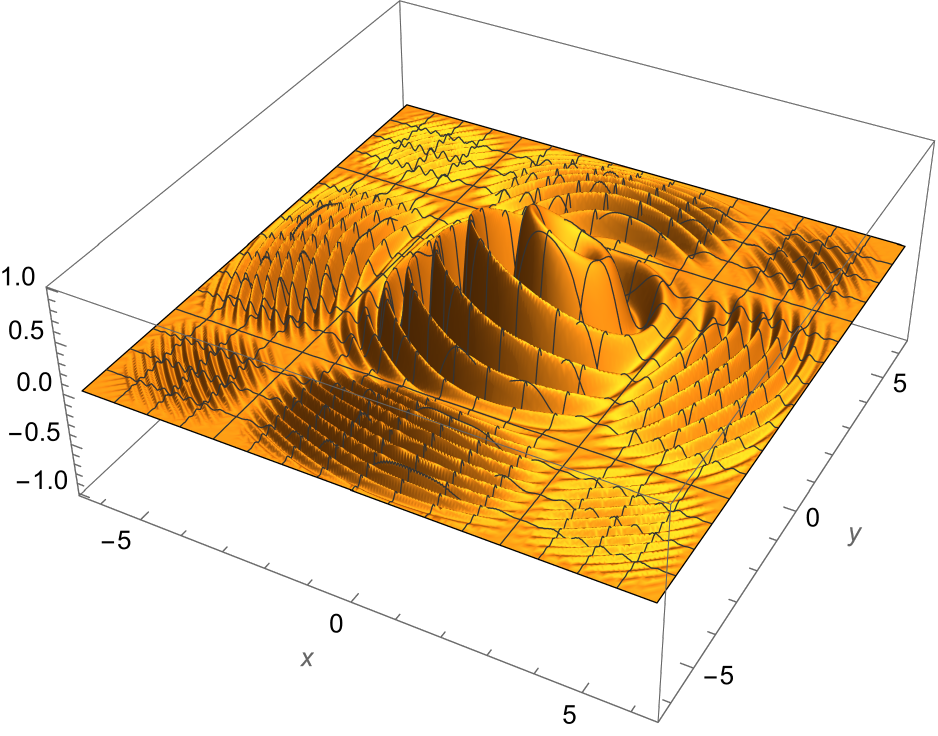}\\ \vspace*{0.5cm}\includegraphics[width=4.cm, height= 3.5cm]{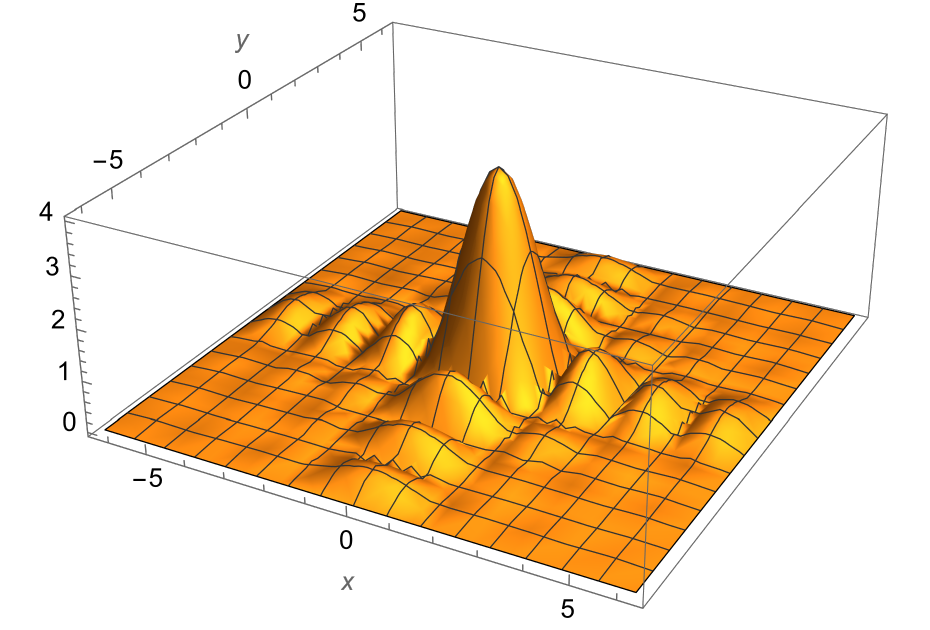}\quad
\includegraphics[width=4.cm, height= 3.5cm]{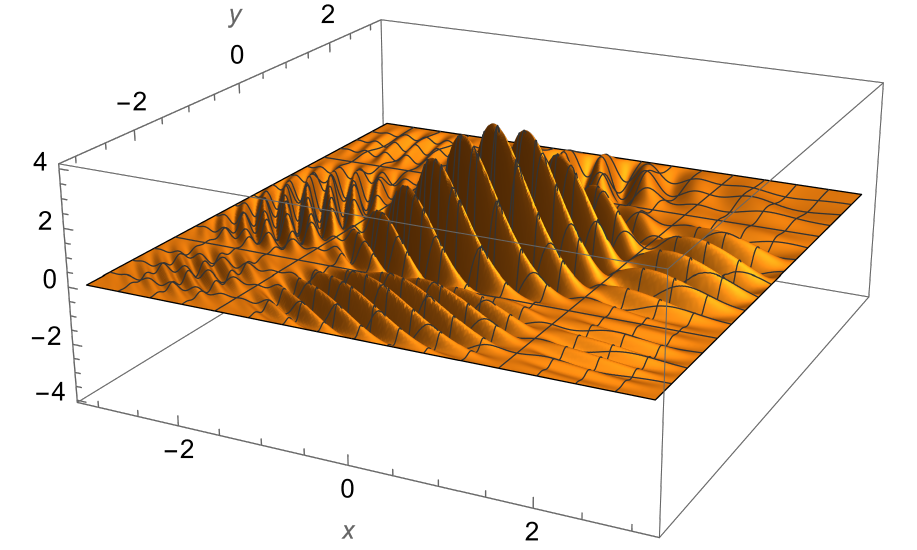}\quad
\includegraphics[width=4.cm, height= 3.5cm]{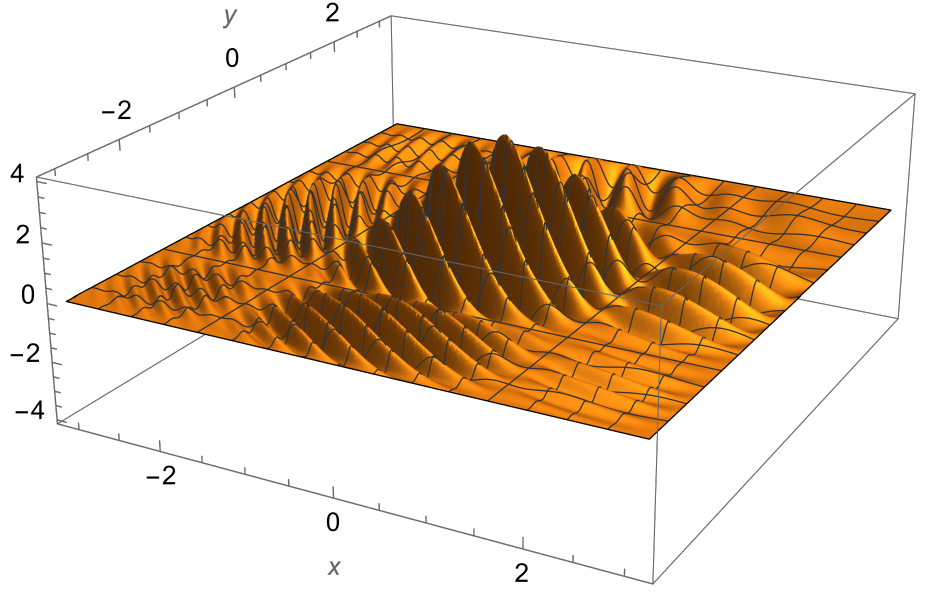}\vspace*{0.5cm}
\includegraphics[width=4.cm, height= 3.5cm]{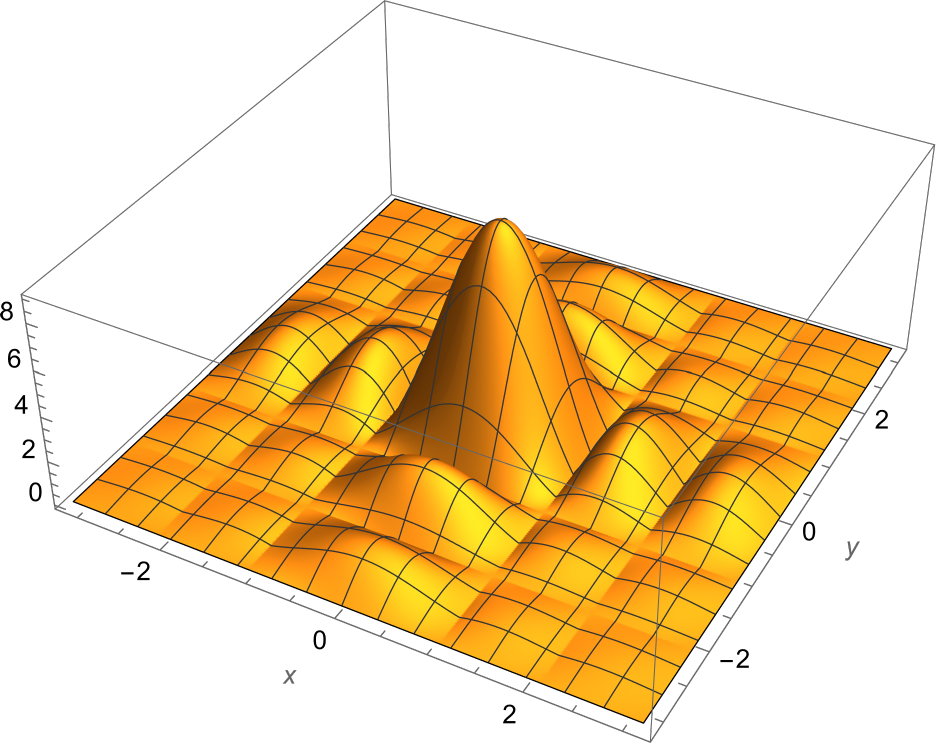}\quad
\includegraphics[width=4.cm, height= 3.5cm]{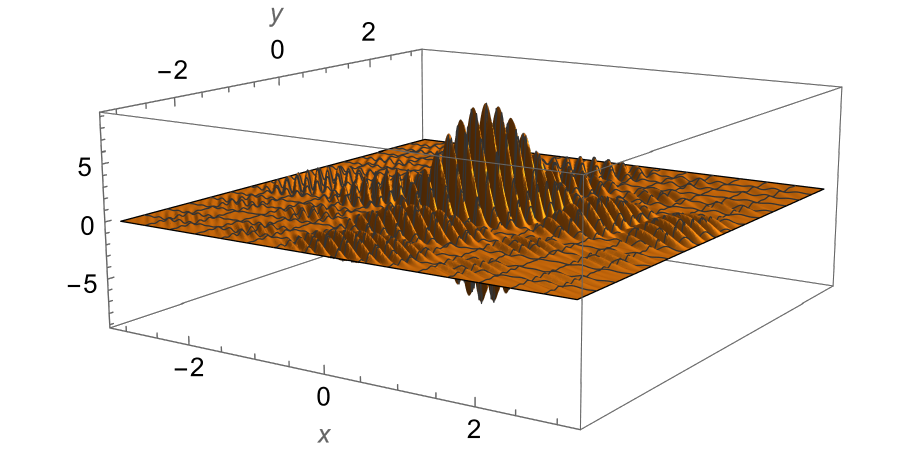}\quad
\includegraphics[width=4.cm, height= 3.5cm]{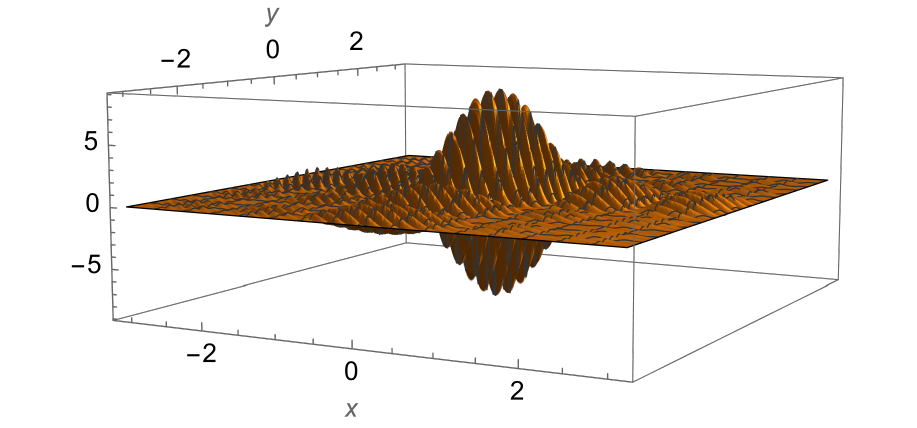}
\caption{The absolute value and the real and imaginary parts of the bivariate functions $f_i$, $i = 1,2,3$ (top to bottom) for $\theta = \pi/3$.}
\end{center}
\end{figure}

Thus, for $m=4$, the Gramian matrix is given by
\[
G^\theta_{\mathcal{F}_4}(\xi,\omega)=
\begin{pmatrix}
1 & 1 & 1 & 1\\
1 & 4 & 4 & 4\\
1 & 4 & 9 & 9\\
1 & 4 & 9 & 16
\end{pmatrix}.
\]
The numerical values of the eigenvalues of $G^\theta_{\mathcal{F}_4}(\xi,\omega)$ are given by
\[
\lambda_1 \approx  23.8417, \quad\lambda_2 \approx 3.76815,\quad\lambda_3 \approx 1.70448,\quad\text{and}\quad\lambda_4 \approx 0.6857,
\]
and the associated eigenvectors by
\begin{align*}
& y_1 \approx (0.0899642, 0.348537, 0.706397, 1)^\top, \quad y_2 \approx (-0.326646, -1.04653, -0.857678, 1)^\top,\\
& y_3 \approx (1.37213, 3.07346, -3.10683, 1)^\top,\quad\text{and}\quad y_4 \approx (-135.024, 50.6467, -9.20855, 1)^\top.
\end{align*}
Then, using these eigenvalues and eigenvectors, we obtain $\phi_1$, $\phi_2$, $\phi_3$. In this case,
\[
\phi_j(x,y)=\frac{y_{j1}}{\sqrt{\lambda_j}}f_1(x,y)+\frac{y_{j2}}{\sqrt{\lambda_j}}f_2(x,y)+\frac{y_{j3}}{\sqrt{\lambda_j}}f_3(x,y)+\frac{y_{j4}}{\sqrt{\lambda_j}}f_4(x,y),
\quad j=1,2,3,\]
where $\lambda_j$ is an eigenvector of $G^\theta_{\mathcal{F}_4}(\xi,\omega)$ with eigenvalue $(y_{j1}, y_{j2}, y_{j3}, y_{j4})^\top$. Thus
\begin{align*}
\phi_1(x,y)=& e^{-\pi i(x^2+y^2)\cot\theta} \big(0.0184e^{\pi i(x+y)}(\sinc (x))(\sinc (y))+0.2855e^{4\pi i(x+y)}(\sinc (2x))(\sinc (2y))\\
&+1.3020e^{9\pi i(x+y)}(\sinc (3x))(\sinc (3y))+3.2768e^{16\pi i(x+y)}(\sinc (4x))(\sinc (4y))\big),\\
\phi_2(x,y)=& -e^{-\pi i(x^2+y^2)\cot\theta}\big(0.1683e^{\pi i(x+y)}(\sinc (x))(\sinc (y))+2.1565e^{4\pi i(x+y)}(\sinc (2x))(\sinc (2y))\\
&+3.9765e^{9\pi i(x+y)}(\sinc (3x))(\sinc (3y))-8.2424e^{16\pi i(x+y)}(\sinc (4x))(\sinc (4y))\big),\\
\phi_3(x,y)=& e^{-\pi i(x^2+y^2)\cot\theta}\big(1.0510e^{\pi i(x+y)}(\sinc (x))(\sinc (y))+9.4166e^{4\pi i(x+y)}(\sinc (2x))(\sinc (2y))\\
&+21.4173e^{9\pi i(x+y)}(\sinc (3x))(\sinc (3y))+12.2553e^{16\pi i(x+y)}(\sinc (4x))(\sinc (4y))\big).
\end{align*}

The absolute value, the real and imaginary part of the functions $\phi_j$, $j=1,2,3$, (top to bottom) are displayed in Figure \ref{fig9}.

\begin{figure}[h!]
\begin{center}
\includegraphics[width=4.5cm, height= 4cm]{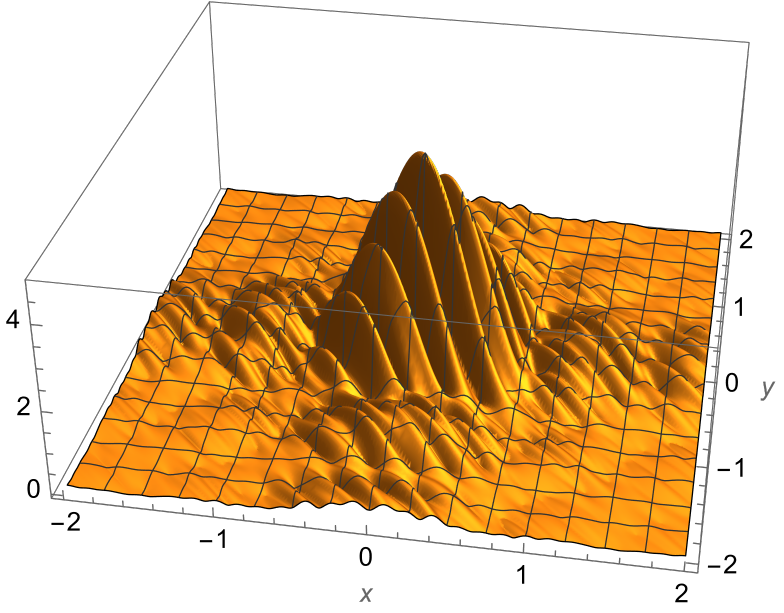}\quad
\includegraphics[width=4.5cm, height= 4cm]{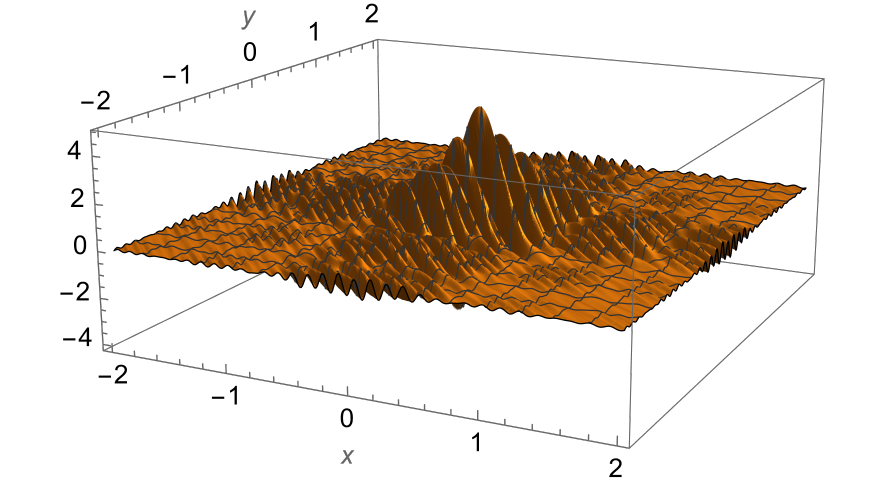}\quad
\includegraphics[width=4.5cm, height= 4cm]{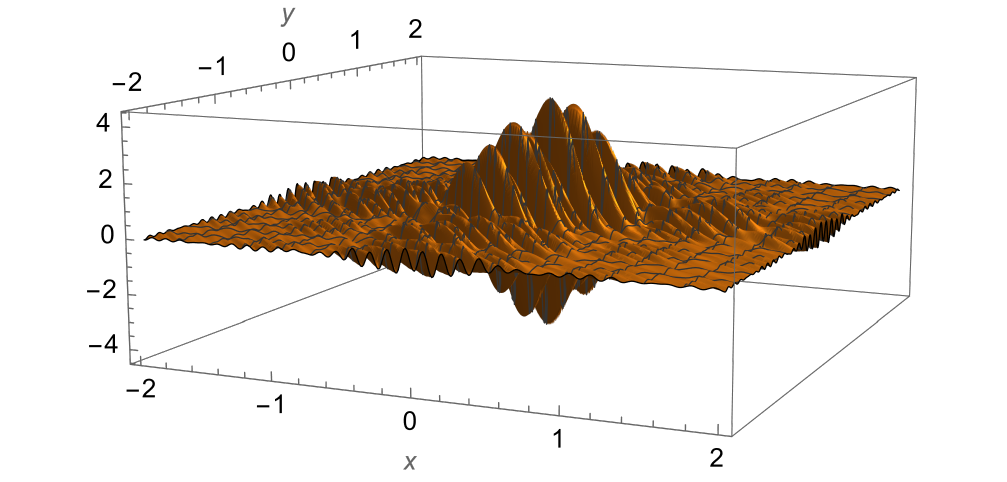}\vspace*{0.5cm}
\includegraphics[width=4.5cm, height= 4cm]{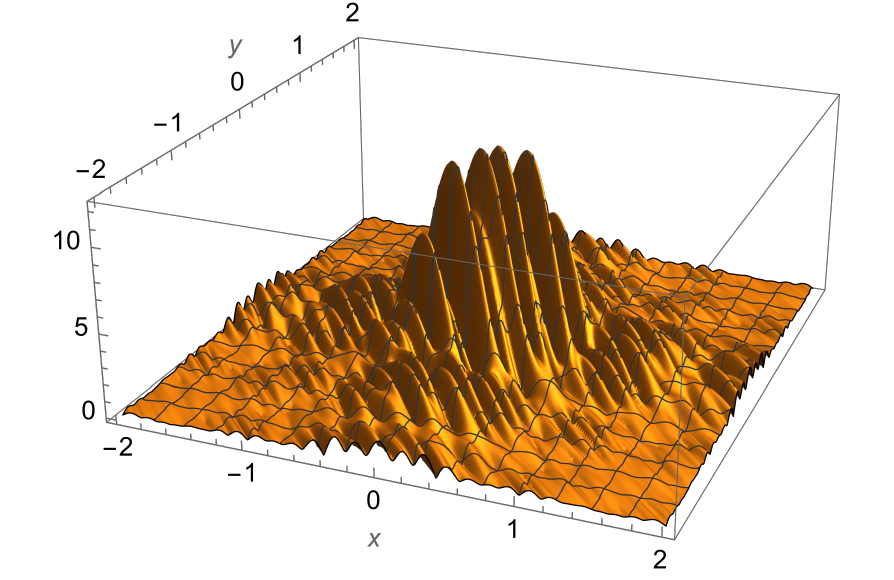}\quad
\includegraphics[width=4.5cm, height= 4cm]{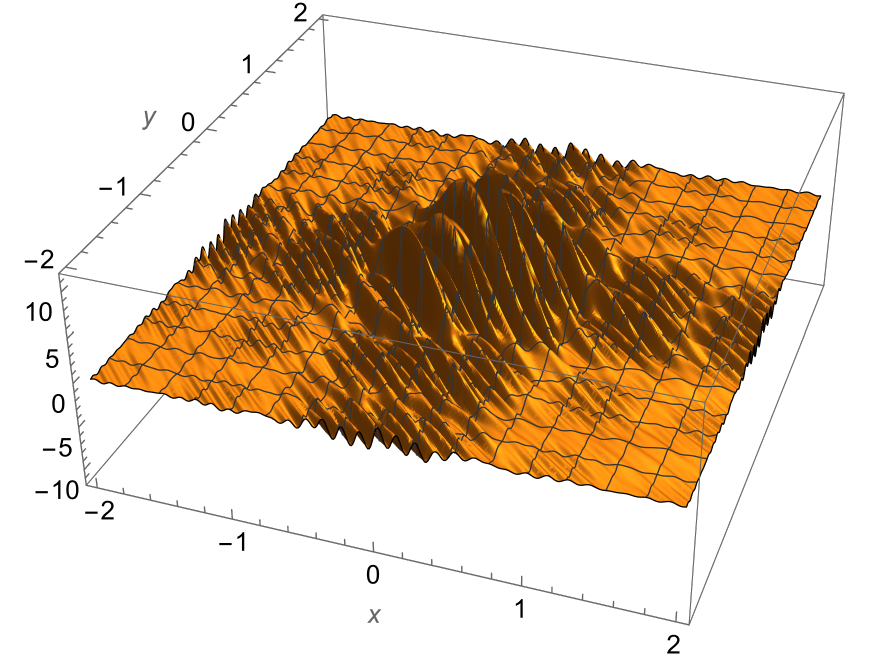}\quad
\includegraphics[width=4.5cm, height= 4cm]{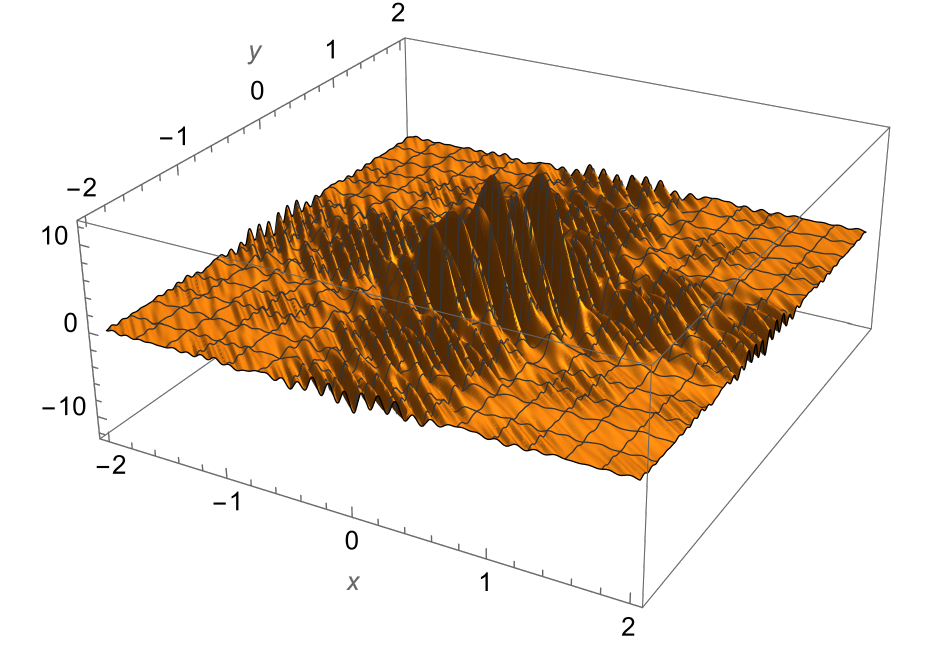}\\\includegraphics[width=4.25cm, height= 4cm]{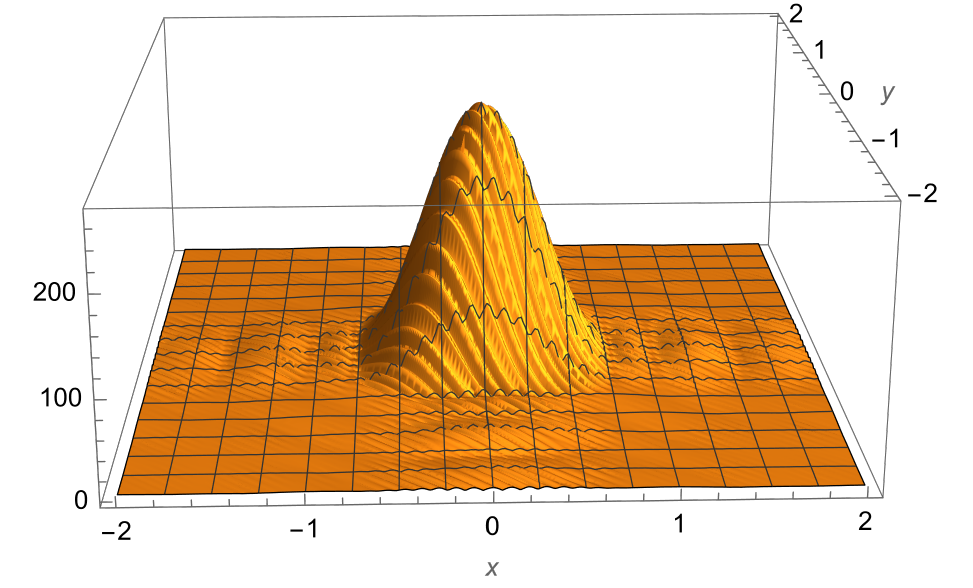}\quad
\includegraphics[width=4.25cm, height= 4cm]{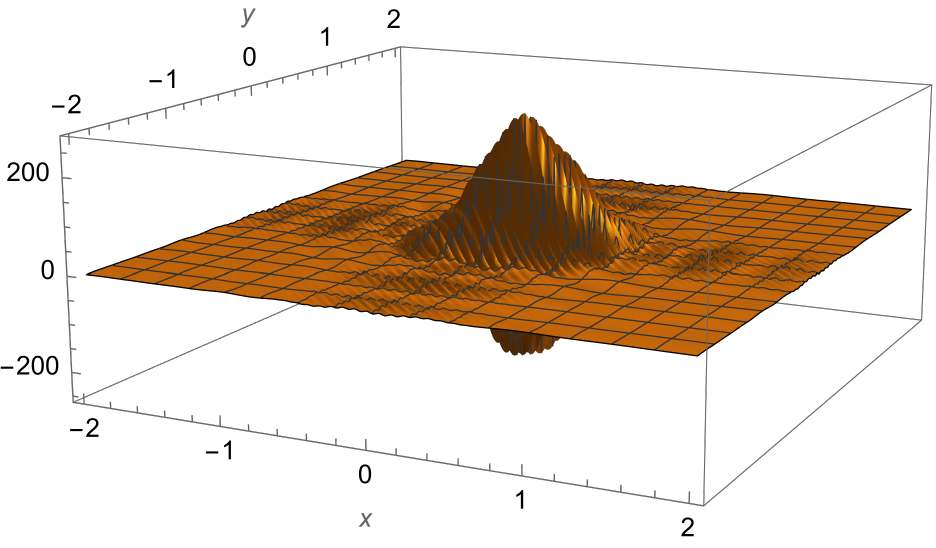}\quad
\includegraphics[width=4.25cm, height= 4cm]{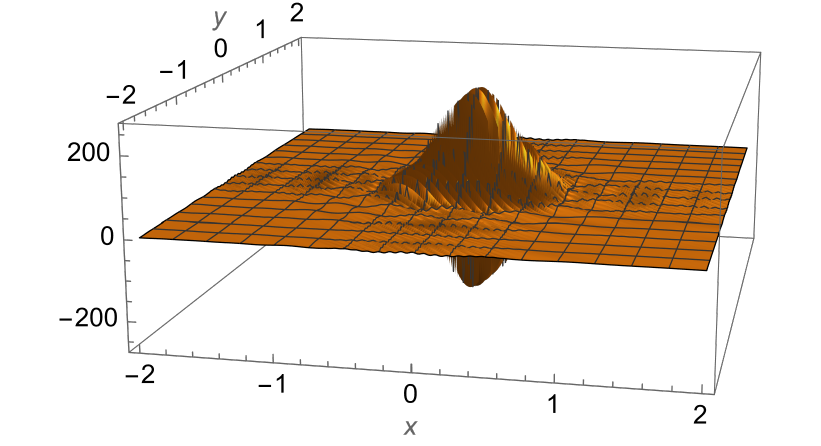}
\caption{The absolute value, the real and imaginary part of the functions $\phi_j$, $j=1,2,3$ (top to bottom) for $\theta = \pi/3$.}\label{fig9}
\end{center}
\end{figure}

Now, using \eqref{eq: error data approx}, we tabulate the error of approximation for different number of generators.
\begin{table}[H]
\begin{tabular}{|c|c|}
\hline
Number of generators & Error\\
\hline
1 & 6.1583$|\sin\theta|^2$\\
\hline
2 & 2.3902$|\sin\theta|^2$\\
\hline
3 & 0.6857$|\sin\theta|^2$\\
\hline
\end{tabular}
\caption{Error in data approximation}
\end{table}
\end{example}

\subsection{Data approximation using the space of bandlimited functions in the FrFT domain}
Let $\Omega\subset\R^n$. Then, using \eqref{eq: preli translation modulation}, it is clear that the subspace 
\[
V=\lc f\in L^2(\R^n):\fft(f)(\omega)=0,~~\text{ a.e. }\omega\in \R^n\setminus \Omega\rc
\] 
is invariant under all $\theta$-translations. Further, using Wiener's theorem associated with the FrFT, we obtain the following.

\begin{propn}
    Let $V\subset L^2(\R^n)$ be a $\theta$-translation invariant subspace. Then, there exists a measurable $\Omega\subset \R^n$ such that
    \[
    V=\lc f\in L^2(\R^n):\fft(f)(\omega)=0,~~\text{ a.e. }\omega\in \R^n\setminus \Omega\rc.
    \]
\end{propn}
\begin{proof}
    Let $V=\ltn$. Then $\Omega=\R^n$. Assume that $V$ is a proper subspace of $\ltn$. Then $V^\perp$ is a non-trivial proper subspace of $\ltn$.\\
    Since, $V$ is a proper subspace of $\ltn$, by Wiener's theorem in connection with the FrFT (see Theorem 3.7 in \cite{H2}), for each $f\in V$, there exists a measurable set $\Omega_f$ having positive measure such that
    \[
    \fft(f)(\omega)=0,~~\text{ a.e. }\omega\in\R^n\setminus\Omega_f.    
    \]
    Let $\Omega=\bigcup\limits_{f\in V}\Omega_f$. Then, $\Omega$ is a measurable subset of $\R^n$ having positive measure.\\
    On the other hand, using Parseval's formula for the FrFT, we can show that
    \[
    V^\perp=\lc g\in L^2(\R^n):\fft(g)(\omega)=0,~~\text{ a. e. }\omega\in\Omega\rc.
    \]
    Since $V^\perp\neq\{0\}$. $\Omega$ is a proper subspace of $\R^n$. This proves our claim.
\end{proof}

In order to study the approximation problem, we use the following setup. Let $\T^{\ell, \theta}$ be the set of all $\theta$-shift invariant spaces generated by at most $\ell$ elements, i.e., $V=\cl_{L^2(\R^n)}{\Span}\{T_k^\theta\phi_j:k\in\Z^n, ~j=1,\cdots,\ell\}$, for some $\phi_1,\cdots,\phi_\ell\in L^2(\R^n)$ such that $\{T_k^\theta\phi_j:k\in\Z^n, ~j=1,\cdots,\ell\}$ forms a Riesz basis for $V$. We obtain the  following characterization of the elements in $\tlt$.

\begin{propn}\label{pro: sec 3 pro}
    A subset $V\in\T^{\ell, \theta}$ if and only if $V=V_\Omega^\theta$ with $\Omega$ is a measurable, fractional $\ell$-multi-tile of $\R^n$.
\end{propn}
\begin{proof}
    Let $V\in\tlt$. Then, $V=V_\Omega^\theta$, for some measurable $\Omega\subset\R^n$. It remains to show that $\Omega$ is a fractional $\ell$-multi-tile of $\R^n$.\\
    First, we observe that
    \[
    J_V^\theta(\omega)\cong \ell^2(O_{\omega,\theta}), ~~\text{ where }O_{\omega,\theta}=\{k\in\Z^n:\omega+k\sin\theta\in\Omega\}.
    \]
    In fact, $J_V^\theta(\omega)\subset\ell^2(O_{\omega,\theta})$ follows from the definition of $J_V^\theta(\omega)$. In order to see the reverse inclusion, let $E_k=(I+k\sin\theta)\cap\Omega$. Then, the $E_k$'s are mutually disjoint, $\Omega=\bigcup\limits_{k\in\Z^n}E_k$ and $k\in O_{\omega,\theta}$ if and only if $\omega+k\sin\theta\in E_k$.\\
    Let $\{a_k\}\in\ell^2(O_{\omega,\theta})$ and define 
    \[
    G_\omega:=\sum_{k\in O_{\omega,\theta}}a_k\bigchi_{E_k}.
    \]
    Then, it is easy to see that $G_\omega\in \ltn$. Let us choose $g\in\ltn$ such that
    \[
    \nch \fft(g)=G_\omega.
    \]
    Then, $\tau^\theta(g)=a_k$. This establishes the required isomorphism.\\
    As, $\{T^\theta_k\phi_j:k\in\Z^n,j=1,\cdots,\ell\}$ forms a Riesz basis for $V$, $\{\tau^\theta(\phi_j)(\omega):j=1,\cdots,\ell\}$ forms a Riesz basis for $J_V^\theta(\omega)$. Thus, $\dim[J_V^\theta(\omega)]=\ell$. But $\dim[J_V^\theta(\omega)]=\#O_{\omega,\theta}=\ell$. This shows that $\Omega$ is a fractional $\ell$-multi-tile $\R^n$.\\
    For the converse, assume that $\Omega$ is a fractional $\ell$-multi-tile of $\R^n$. Let $V:=V^\theta_\Omega$. So $V$ is $\theta$-translation invariant. Further, using Lemma \ref{lemma: partition of tile}, 
    \[
    \Omega=\Omega_1\cup\cdots\cup\Omega_\ell,
    \]
    up to a set of measure zero, where the $\Omega_j$'s are mutually disjoint and each of them tiles $\R^n$, when translated by $\Z^n\sin\theta$. Choose $\phi_j\in\ltn$ such that $\fft(\phi_j)=\bigchi_{\Omega_j}$, $j=1,\cdots,\ell$. Since the collection $\{e^{\pi ik^2\cot\theta} e^{-2\pi ik\omega\csc\theta}\bigchi_{\Omega_j}:k\in\Z^n\}$ is an orthonormal basis for $L^2(\Omega_j)$, $\{e^{\pi ik^2\cot\theta} e^{-2\pi ik\omega\csc\theta}\bigchi_{\Omega_j}:k\in\Z^n, j=1,\cdots,\ell\}$ is an orthonormal basis for $L^2(\Omega)$. This implies that $\{T_k^\theta\phi_j:k\in\Z^n,j=1,\cdots,\ell\}$ is an orthonormal basis for $V$, which leads to $V\in\tlt$.
\end{proof}

Now, we are ready to prove the main theorem of this subsection. Towards this end, we fix the following notation. We denote the $n$-dimensional cube $[-(N+1/2),N+1/2]^n$ by $C_N$. Let 
\[
M_N^{\ell,\theta}:=\lc \Omega\subset C_N: \Omega\text{ is measurable and fractional $\ell$-multi-tile $\R^n$}\rc
\]
and
\[
\T^{\ell,\theta}_N:=\lc V\in \T^{\ell,\theta}:V=V^\theta_\Omega \text{ with }\Omega\in M^{\ell,\theta}_N\rc.
\]

\begin{theorem}
    Assume that $m,\ell\in\N$ and a set $\mathcal{F}=\{f_1,\cdots,f_m\}\subset L^2(\R^n)$ are given. Then, for each $N\geq \ell$, there exists a $V^*\in \T^{\ell,\theta}_N$ such that
    \[
    V^*=\argmin\limits_{V\in\T^{\ell,\theta}_N}\sum_{j=1}^m\|f_j-P_Vf_j\|^2.
    \]
\end{theorem}
\begin{proof}
    It is enough to show that there exists a $V^*\in\tlt_N$ such that
    \[
    V^*=\argmax\limits_{V\in\tlt_N}\|P_Vf_j\|^2.
    \]
    Now, using the definition of $\tlt_N$, we have
    \[
    \max_{V\in\tlt_N}\sum_{j=1}^m\|P_Vf_j\|^2=\max_{\Omega\in \mlt}\sum_{j=1}^m\|P_{V_\Omega^\theta}f_j\|^2.
    \]
    Further, for any $\Omega\in\mlt$,
    \begin{align}\label{eq: sec3 thm}
        \sum_{j=1}^m\lnm P_{V_\Omega^\theta}f_j\rnm^2=\sum_{j=1}^m\lnm\tau^\theta\lp P_{V_\Omega^\theta}f_j\rp\rnm_{L^2\lp I,\ell^2(\Z^n)\rp}=\sum_{j=1}^m\int_I \lnm P_{J^\theta_{V^\theta_\omega}}\lp\tau^\theta(f_j)\rp\rnm^2_{\ell^2}\diff\omega,
    \end{align}
    using \eqref{eq: tau theta proj}. Furthermore, if $\Omega\in\mlt$, we know from Proposition \ref{pro: sec 3 pro} that $\dim(J^\theta_{V^\theta_\Omega(\omega)})=\ell$, for a.e. $\omega\in I$. In other words, $J^\theta_{V^\theta_\Omega(\omega)}$ agrees with a subspace of $\ell^2(\Z^n)$ of the sequences supported in $O_{\omega,\theta}$.\\
    Let $k^\Omega(\omega):=\{k_1^\Omega(\omega),\cdots,k_\ell^\Omega(\omega)\}\subset\Z^n$ be such that $\{\delta_{k_j^\Omega}(\omega):j=1,\cdots,\ell\}$ spans $J^\theta_{V^\theta_\Omega(\omega)}$, for a.e. $\omega\in I$. Since, $\Omega\subset C_N$, $\|k_j^\Omega(\omega)\|_\infty\leq N$ for each $j$ and $\omega$. Combining this observation with \eqref{eq: sec3 thm}, we obtain
    \begin{align*}
        \sum_{j=1}^m\|P_{V^\theta_\Omega}f_j\|^2&=\sum_{j=1}^m\int_I\lnm J^\theta_{V^\theta_\Omega(\omega)}\lp\tau^\theta f_j(\omega)\rp\rnm^2_{\ell^2}\diff\omega=\sum_{j=1}^m\int_I\sum_{s=1}^\ell|(\ch f_j)~\widehat{}~(\omega\csc\theta+k_s^\Omega(\omega)|^2\diff\omega.
    \end{align*}
    Thus, we need to maximize the right hand side over all $\Omega\in\mlt$. \\
    Again, from the proof of Proposition \ref{pro: sec 3 pro}, given $\Omega\in\mlt$, for a.e. $\omega\in I$, the set $\Omega$ contains exactly $\ell$ elements from the sequence $\{\omega+k\sin\theta\}$.\\
    Then, for each $\omega\in I$, we pick $\ell$-translations $k_s^*(\omega)$ such that
    \[
    \sum_{j=1}^m\sum_{s=1}^\ell|(\ch f_j)~\widehat{}~(\omega\csc\theta+k^*_s(\omega))|^2=\sum_{j=1}^m\sum_{s=1}^\ell|\fft(f_j)(\omega+k_s^*(\omega)\sin\theta)|^2
    \]
    is maximal over all sets of $\ell$-translations $k=(k_1,\cdots,k_\ell)\subset\Z^n$ with $\|k_j\|_\infty\leq N$.\\
    Let $\mathcal{K}:=\{k=(k_1,\cdots,k_\ell)\subset\Z^n:\|k_j\|_\infty\leq N\}$ and for $k\in\mathcal{K}$, set
    \[
    H_k(\omega):=\sum_{j=1}^m\sum_{s=1}^\ell |(\ch f_j)~\widehat{}~(\omega\csc\theta+k_s(\omega))|^2,
    \]
    and
    \[
    E_k:=\{\omega\in I:H_k(\omega)\geq H_r(\omega), ~~\text{ for all }r\in\mathcal{K}\}.
    \]
    Define $\Omega^*:=\bigcup\limits_{k\in\mathcal{K}}\bigcup\limits_{j=1}^\ell E_k+k_j\sin\theta$. Then, writing $E_k=\bigcap\limits_{r\in\mathcal{K}}F_r^k$, where $F^k_r:=\{\omega\in I:H_k(\omega)\geq H_r(\omega)\}$, we can show that $\Omega^*$ is measurable. Moreover, $\Omega^*\in\mlt$, by construction. Finally, for all $\Omega\in\mlt$,
    \[
    \sum_{j=1}^m\sum_{s=1}^\ell |(\ch f_j)~\widehat{}~(\omega\csc\theta+k_s^\Omega(\omega)|^2)|^2\leq \sum_{j=1}^m\sum_{s=1}^\ell |(\ch f_j)~\widehat{}~(\omega\csc\theta+k_s^*(\omega)|^2)|^2,
    \]
 for a.e. $\omega\in I$. Now, integrating over $I$, we see that $V_{\Omega^*}^\theta$ is the desired space.
\end{proof}

\begin{example}
For the case $n=1$ and $m=4$ (see Example \ref{ex4.6}(i)), we display the graphs of $P_V^N f_j := \sum\limits_{i=1}^3 \sum\limits_{k=-N}^N \inn{f_j}{T^\theta_k \phi_i}\,T^\theta_k \phi_i$ for $j = N = 1$ in Figure \ref{fig10}.
\begin{figure}[h!]
\begin{center}
\includegraphics[width=4.5cm, height= 3cm]{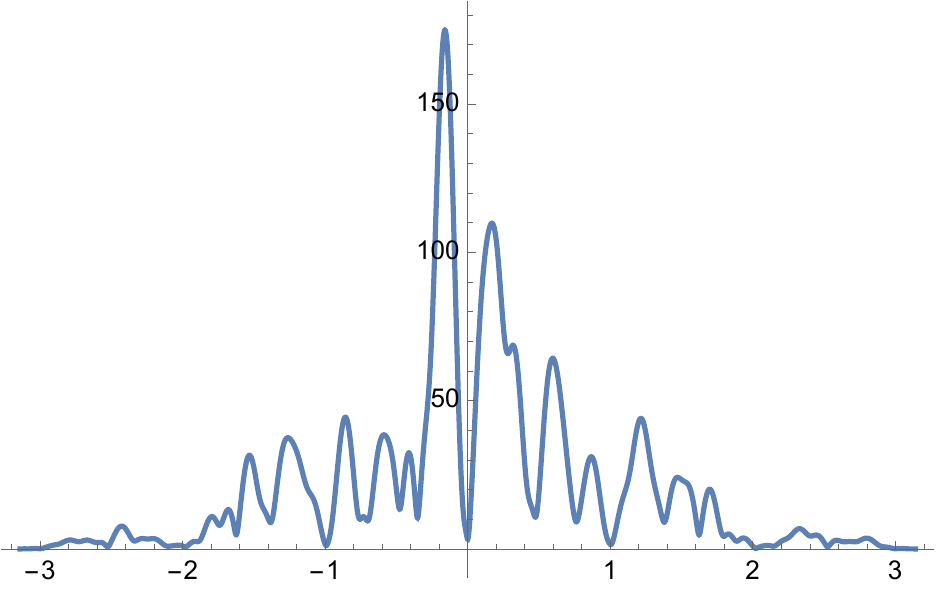}\hspace*{0.5cm}
\includegraphics[width=4.5cm, height= 3cm]{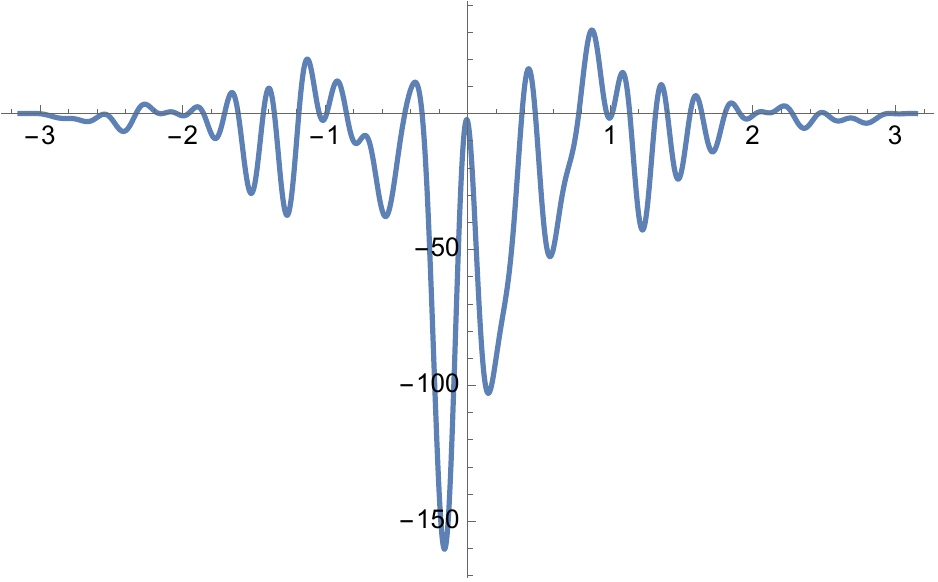}\hspace*{0.5cm}
\includegraphics[width=4.5cm, height= 3cm]{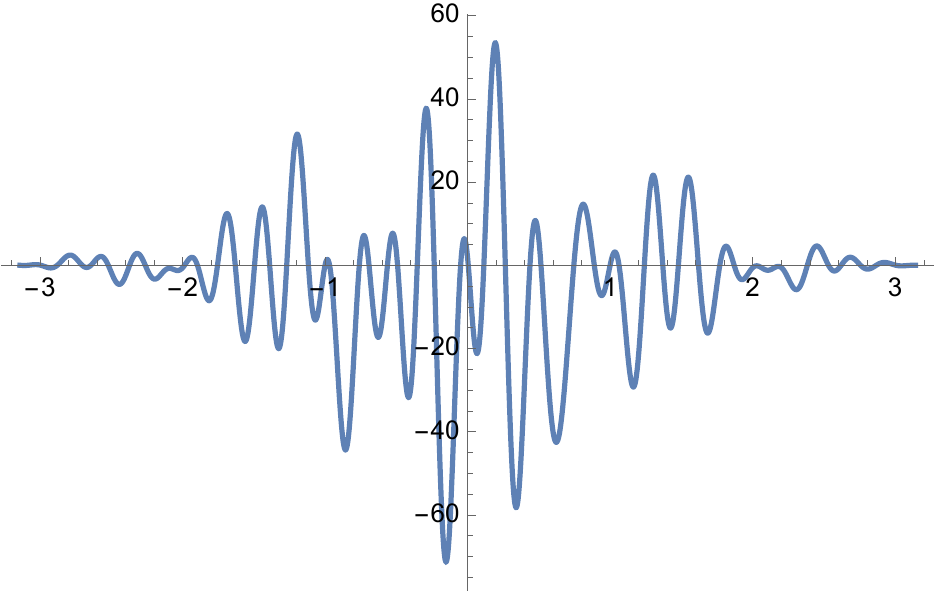}\\
\caption{The absolute value, the real and imaginary part of the function $P_V^N f_1$ for $N=1$ and $\theta = \pi/3$.}\label{fig10}
\end{center}
\end{figure}
\end{example}

\bibliography{ref24A}
\bibliographystyle{amsplain}
\end{document}